\documentclass[12pt, reqno]{amsart}
\usepackage{amsfonts}
\usepackage{bbm}
\usepackage{amscd,amsfonts}
\usepackage{amssymb, eucal, amsfonts, amsmath, xypic, latexsym, tikz}
\usepackage{pifont}
\usepackage{mathrsfs,color}
\usepackage{amsthm,indentfirst,bm,fancyhdr,dsfont}
\usepackage{graphicx}
\usepackage[all]{xy}
\usepackage[CJKbookmarks=true]{hyperref}
\usepackage[normalem]{ulem}

\usepackage{mathrsfs}
\usepackage{amsmath}
\usepackage{amssymb}
\usepackage{hyperref}

\numberwithin{equation}{section}

\renewcommand{\subset}{\subseteq}

\newtheorem{thm}{Theorem}[section]
\newtheorem{coro}[thm]{Corollary}

\newtheorem{lem}[thm]{Lemma}
\newtheorem{defn}[thm]{\noindent Definition}

\newtheorem{prop}[thm]{Proposition}
\newtheorem{convention}[thm]{Convention}
\newtheorem{rem}[thm]{Remark}
\def\pf{\noindent\textit{Proof}. }
\renewcommand{\qed}{\hfill $ \square $\medskip}

\def\degE{{\lfloor E\rfloor}}

\def\hht{\textrm{ht}}
\def\ch{\mathsf{ch}}
\def\gl{\mathfrak{gl}(n)}

\def\Hom{\text{\rm Hom}}
\def\End{\text{\rm End}}

\def\diag{\text{\rm diag}}
\def\dim{\text{\rm dim\,}}

\def\max{\text{max}}
\def\min{\text{min}}

\def\bz{{\bar 0}}
\def\bo{{\bar 1}}
\def\id{\textsf{id}}

\def\frakb{{\mathfrak {B}}}
\def\frakn{{\mathfrak{N}}}
\def\h{{\frak h}}

\def\comx{{\mathcal{O}_X^{\text{min}}}}
\def\comi{{\mathcal{O}^{\text{min}}}}
\def\comit{\mathcal{O}_\theta^{\text{min}}}
\def\comitl{{\mathcal{O}^{\text{min}}_\lambda}}
\def\comif{\mathcal{O}^{\text{min}}_{\mathfrak{f}}}
\def\comid{\mathcal{O}^{\text{min}}_{\geq d}}

\def\comiid{\mathcal{O}^{\text{min}}_{\geq d}}

\def\comidd{\mathcal{O}^{\text{min}}_{\geq d'}}

\def\comis{{\mathcal{O}^{\text{min}}}}

\def\gmf{{\mathfrak{g}\mbox{-\bf{mod}}^f}}

\def\dpt{\mathsf{dpt}}
\def\pty{\mathsf{pty}}

\def\cf{\mathcal{F}}
\def\calm{\mathcal{M}}
\def\caln{\mathcal{N}}
\def\calc{\mathcal{C}}

\def\co{{\mathcal{O}}}
\def\cof{{\mathcal{O}^0_{\text{fin}}}}
\def\cae{{\mathcal{E}}}
\def\bfe{{\mathbf{E}}}
\def\bfb{{\mathbf{B}}}

\def\bsn{{\bar S(n)}}
\def\bhn{{\bar H(n)}}
\def\wn{{W(n)}}
\def\ochn{{\overline{C{\hskip-0.1cm}H}(n)}}
\def\och{{\overline{C{\hskip-0.1cm}H}}}
\def\chn{{C{\hskip-0.1cm}H}(n)}

\def\Phig{{\Phi_{\geq 1}}}

\def\sft{{\mathsf{t}}}
\def\sfp{{\mathsf{P}}}
\def\sfd{{\mathsf{d}}}
\def\sfS{{\mathsf{S}}}

\def\wt{{\mathsf{Wt}}}

\def\ggg{\mathfrak{g}}
\def\sss{\mathfrak{s}}
\def\hhh{\mathfrak{h}}

\def\nnn{\mathfrak{n}}

\def\bbb{\mathfrak{b}}

\def\uk2{U_{\chi}{(\mathfrak{gl}}(2))}

\newcommand{\zz}{\mathbb{Z}}

\newcommand{\red}[1]{{\color{red}#1}}

\def\bo{{\bar 1}}
\def\bz{{\bar 0}}

\newcommand{\bbz}{\mathbb{Z}}
\newcommand{\bbn}{\mathbb{N}}
\newcommand{\bbc}{\mathbb{C}}
\newcommand{\bbr}{\mathbb{R}}

{\vskip-\lastskip\medskip
  \noindent
  {\em #1.}\enspace
  }%
{\qed\par\medskip
  }

  \def\almf{{A_\lambda\mbox{-\bf{mod}}^f}}

\newcommand{\evid}[1]{\textcolor{red}{#1}}

\begin{document}

\title[Representations of Lie superalgebras of Cartan type]
{Parabolic BGG categories  and their block decomposition for  Lie superalgebras of Cartan type}

\author{Fei-Fei Duan, Bin Shu and Yu-Feng Yao}

\address{College of Mathematics and Information Science, Hebei Normal University, Shijiazhuang, Hebei, 050024, China.}\email{duanfeifei0918@126.com}

\address{School of Mathematical Sciences,  Key Laboratory of MEA (Ministry of Education) \& Shanghai Key Laboratory of PMMP,  East China Normal University, Shanghai 200241, China} \email{bshu@math.ecnu.edu.cn}

\address{Department of Mathematics, Shanghai Maritime University, Shanghai, 201306, China.}\email{yfyao@shmtu.edu.cn}

\subjclass[2010]{17B10, 17B66, 17B70}

\keywords{Lie superalgebras of Cartan type, parabolic BGG category, blocks, tilting modules, projective covers, semi-infinite characters}

\thanks{This work is partially supported by the National Natural Science Foundation of China (Grant Nos. 12071136 and 12271345), supported in part by Science and Technology Commission of Shanghai Municipality (No. 22DZ2229014), and Hebei Natural Science Foundation of China (No.A2021205034)}

\begin{abstract} In this paper, we study the parabolic BGG categories  for  graded Lie superalgebras of Cartan type over the field of complex numbers. The gradation of such a  Lie superalgebra $\ggg$ naturally arises, with the zero component $\ggg_0$ being a reductive Lie algebra.
We first show that there are only two proper parabolic  subalgebras containing Levi subalgebra $\ggg_0$: the ``maximal one" $\sfp_\max$ and the ``minimal one" $\sfp_\min$. Furthermore, the parabolic BGG category arising  from  $\sfp_\max$ essentially turns out to be a subcategory of the one arising from $\sfp_\min$. Such a priority of  $\sfp_\min$ in the sense  of representation theory reduces the question to the study of the ``minimal parabolic" BGG category $\comi$ associated with $\sfp_\min$. We prove the existence of projective covers of simple objects in these categories, which enables us to establish a satisfactory  block theory. Most notably, our main results are as follows.
	
	(1) We classify and  obtain a precise description of the blocks of $\comi$.
	
		(2) We investigate indecomposable  tilting and indecomposable projective modules in  $\comi$, and compute their character formulas.
	\end{abstract}
\maketitle
\setcounter{tocdepth}{1}\tableofcontents
\section*{Introduction}

\subsection{}
By Kac's classification theorem (\cite{Kac77}), finite-dimensional simple Lie superalgebras over the field of complex numbers are either of classical type or of Cartan type, with the latter consisting of infinite series of the four types $W(n)$, $S(n)$, $\tilde S(n)$ and $H(n)$. The simple Lie superalgebra $W(n)$ ($n \ge 3$) is the derivation algebra of the Grassmann superalgebra $\Lambda(n)$ on $n$ generators. Arising from the natural $\mathbb Z$-grading on $\Lambda(n)$, $W(n)$ is also naturally $\mathbb Z$-graded. The Lie superalgebras $S(n)$ ($n\geq 4$), $\tilde{S}(n)$ ($n\geq 4$) and $H(n)$ ($n\geq 5$) are Lie subalgebras of $W(n)$. The superalgebra $\tilde{S}(n)$ is not $\mathbb Z$-graded, but carries a filtration induced by the filtration of $W(n)$.

Irreducible finite-dimensional representations of Lie superalgebras of Cartan type were studied earlier (\cite{BL83}, \cite{Shap82}, etc.), motivated by Rudakov's work  on  irreducible representations of infinite-dimensional Lie algebras of Cartan type (\cite{Rud73} and \cite{Rud75}). In \cite{Ser05}, Serganova considered the category of $\bbz$-graded irreducible representations of graded Lie superalgebras of Cartan type, determined the character formulas of their  $\bbz$-graded  irreducible highest weight  modules. After that, there were a few papers on finite-dimensional representations over $W(n)$. For example, in \cite{BKN} the authors computed the cohomological support varieties of irreducible $W(n)$-modules in a certain category, the objects of which are finite-dimensional and completely reducible over the zero component $W(n)_0$. In \cite{Shom02} Shomron studied the blocks of a certain category whose objects are finite-dimensional $W(n)$-modules by constructing extensions between irreducible modules. However,  when considering categories containing infinite-dimensional objects,  the situation becomes very  complicated. 

\subsection{}

Let $\ggg$ be a Lie superalgebra of Cartan type $X(n)$, where $X \in \{W, S, H\}$. Then $\ggg$ is naturally endowed with a $\bbz$-graded structure, i.e., $\ggg=\sum_{i\geq -1}\ggg_i.$ In addition, $\ggg_0$ is a reductive Lie algebra.   When $X \in \{S, H\}$, it will be convenient to study, in place of $X(n)$, the representation category of  the one-dimensional toral  extension $\bar X(n)$ determined by the following exact sequence
$$X(n)\hookrightarrow \bar X(n) \twoheadrightarrow \mathbb{C}\sfd,$$
where $\sfd$ is a canonical toral element measuring degrees in $W(n)$
(see \S\ref{toral extension} for details).

In the present paper, we introduce and study a parabolic BGG category for $X(n)$ with $X\in\{W,\bar{S},\bar{H}\}$, in analogy to the Bernstein-Gelfand-Gelfand category of complex semisimple Lie algebras (see \cite{BGG} and \cite{Hum08}). Our purpose is to investigate blocks in this category, develop a tilting module theory, and give character formulas of indecomposable tilting and indecomposable projective modules. Recall that a Lie superalgebra of Cartan type
admits many mutually non-conjugate Borel subalgebras  (\cite[\S4]{Ser05}), also many mutually non-conjugate Borel subalgebras containing the standard Borel subalgebra of the core reductive Lie subalgebra $\ggg_0$ (see \S\ref{lem:borel chain}), and hence possibly admits many ``parabolic" subalgebras.
An important ingredient in our work is to discriminate these parabolic subalgebras, and choose  a suitable ``parabolic subalgebra''. Surprisingly, there are only two such  parabolic (proper) subalgebras,  the maximal one $\sfp_\max$ which is actually $\sum_{i\geq 0}\ggg_{i}$, and the minimal one $\sfp_\min$ which is equal to $\sum_{i\leq 0}
\ggg_i$ (see Proposition \ref{prop: det para}). Furthermore, the ``parabolic BGG category" associated with $\sfp_\max$ turns out to be less interesting (see \S\ref{parabolic func}) since  whose $U(\ggg)$-finitely-generated objects
are finite-dimensional. Actually,  the parabolic BGG category associated with $\sfp_\max$  is a subcategory of the parabolic BGG category associated with $\sfp_\min$ if only considering objects finitely generated over $U(\ggg)$.

Based on the above analysis, we only need to focus on $\sfp_\min=\ggg_{-1}\oplus\ggg_0$. In this article, we simply write it as  $\sfp$ which is naturally regarded as a ``minimal parabolic" subalgebra of $\ggg$ containing the reductive Lie algebra $\ggg_0$.
We then introduce the parabolic BGG category $\comi$ associated with $\sfp$.
This category is by definition a subcategory of the $\bbz$-graded $U(\ggg)$-module category, satisfying some standard  axioms (see Definition \ref{defn}).
What is completely different from \cite{Ser05} is that all standard modules  have infinite composition factors. Nevertheless, we can prove the following fundamental result.
\begin{thm} (See Theorem \ref{projective thm}) Any simple object in $\comi$ has a projective cover which admits a flag of standard modules.
\end{thm}

\subsection{} Along the direction just mentioned above, we can define blocks of $\comi$ via projective covers of irreducible modules, and get into the next topic---to classify and describe all blocks of $\comi$.

It can be proven that all simple objects in $\comi$ are parameterized by what we denote here by $\mathbf{E}$, that is, a combination of finite-dimensional irreducible modules over $\mathfrak g_0$ and their so-called ``depths" associated with the $\bbz$-gradation.
We give one of our main results in the following.

\begin{thm}\label{intro second thm} (See Theorems \ref{block thm WS}, \ref{block thm H odd} and \ref{block thm H even}) Let $\ggg=X(n)$, $X\in \{W,\bar S, \bar H\}.$ For any given $L(\lambda), L(\mu)\in \bfe$, $L(\lambda)$ and $L(\mu)$ lie in the same block if and only if the following three conditions are satisfied.
	\begin{itemize}
		\item[(1)] $\mu\in \lambda+Q;$
		\item[(2)]   $\dpt(L(\mu))=\dpt(L(\lambda))+\ell(\lambda-\mu);$
		\item[(3)]  $\pty(L(\mu))=\pty(L(\lambda))+\overline{\ell(\lambda-\mu)},$
	\end{itemize}
	where $\dpt(L(\lambda))$ denotes the depth of $L(\lambda)$ associated with its $\bbz$-graded structure; $\pty(L(\lambda))$ is the parity of the ``maximal vector" $v^0_\lambda$ of $L(\lambda)$; and for each $\alpha \in Q$ we write $\ell(\alpha)$ for the length of $\alpha$ (see (\ref{length}) for the definition) and $\overline{\ell(\alpha)}$ for the parity of $\ell(\alpha)$.
\end{thm}

A more precise structural description of blocks can be found in Theorems \ref{block thm WS}, \ref{block thm H odd} and \ref{block thm H even}.
Below, we will give an outline of the proof of Theorem 0.2.

While establishing the existence of the projective cover $P(\lambda)$ of an irreducible module $L(\lambda)$ in $\mathcal O^{\rm min}$,
we consider an  ``enveloping" projective module $I(\lambda)$, which is induced from irreducible modules $L^0(\lambda)$ over the graded-zero component $\ggg_0$, endowed with a flag of standard modules. Then we can prove that $I(\lambda)$ lies in the same block $\mathbf{B}(\lambda)$ as $ P(\lambda)$.
Based on the construction of $I(\lambda)$, we use various strategies to read off information about $\bfb(\lambda)$. In particular, we examine maximal vectors.
Along this way, the block decomposition becomes easy for $W(n)$ and $\bar{S}(n).$
However, it does not work well for  $\bhn$.  The solution is to establish the relations  between the standard modules of $\ochn$ and the standard modules of $\bhn.$ (Here $\ochn$ is a Lie subalgebra of $W(n)$  while  $\bhn$ is the derived subalgebra of $\ochn$ with codimension one in $\ochn$).  The most important step in this approach is the non-trivial observation that  all standard modules for $\ochn$ are indecomposable over $\bhn$  (see Corollary \ref{coro for H}). Another thing to notice is that
the behavior of $\bar H(2r+1)$ at the root lattice is critically different from that of $\bar H(2r)$, which is ultimately a consequence of
the difference of orthogonal classical Lie algebras of types $B_r$ and  $D_r$. So proving the final results on blocks for $\bar H(2r+1)$ and $\bar H(2r)$ will require separate arguments (see Theorems \ref{block thm H odd} and \ref{block thm H even}). 


The above block theorem actually reveals a somewhat degenerate behavior of blocks for algebraic models of Cartan series, in comparison with the classical theory of complex semi-simple Lie algebras and basic classical Lie superalgebras (see \cite{Hum08} and \cite{CW12}). The intrinsic  mechanism  should be further investigated.

\subsection{} Another important ingredient in our arguments is to prove that each of $W(n)$, $\bar S(n)$ and $\bar H(n)$ admits a semi-infinite character. The notation of semi-infinite character put forward by Soergel was derived  from the work on semi-cohomology by  Feigin, Voronov and Arkhipov (cf. \cite{Ar97}, \cite{Fe84} and \cite{Vor93}).  For $\bbz$-graded Lie algebras admitting semi-infinite characters,  Soergel established in \cite{Soe98} a framework for some $\bbz$-graded representation category.  Following Soergel's work \cite{Soe98}, Brundan  investigated some general theory of category $\mathcal{O}$ for a general $\bbz$-graded Lie superalgebra in \cite{Brun04}, which can be used to study representations of classical Lie superalgebras, and especially to deal with $\mathfrak{gl}(m,n)$ and $\mathfrak{q}(n)$.
Fortunately, the general theory of Brundan's work is available to the case of $\zz$-graded Lie superalgebras of Cartan type, so we have the category $\comis$ in the present paper. Especially, a BGG reciprocity for  truncated categories in \cite{Brun04} is true for $\comis$. Furthermore,  we can investigate tilting modules in $\comis$ on the basis of Soergel's and Brundan's work{\color{red}.} Most notably, we establish Soergel's reciprocity for tilting modules in our $\comis$.

Recall that our category $\comi$ is associated with the ``minimal parabolic" subalgebra $\sfp$, which enables us to obtain a realization of co-standard modules in $\comi$ via Kac modules. This is very important for us to go further, and in particular it leads to the following reciprocities.

\begin{thm} (See Theorem \ref{iso-2}, Propositions \ref{iso1}) Let $P(\lambda)$(resp. $T(\lambda)$) be the indecomposable projective (resp. tilting) module  in $\comis$ corresponding to the simple  object $L(\lambda)\in \bfe$,  and $K(\lambda)$ be the corresponding Kac module. Let $[P(\lambda):\Delta(\mu)]$ (resp. $[T(\lambda):\Delta(\mu)]$) denote the multiplicity of the standard module $\Delta(\mu)$ in $P(\lambda)$ (resp. $T(\lambda)$). Then the following statements hold.
	\begin{itemize}
		\item[(1)] If $\ggg=W(n),$ then
        \begin{eqnarray*}
        [P(\mu):\Delta(\lambda)]&=&(K(\lambda+\Xi):L(\mu));\cr
        [T(\mu):\Delta(\lambda)]&=&(K(-w_0\lambda+2\Xi):L(-w_0\mu+ \Xi)).
        \end{eqnarray*}
		\item[(2)] If $\ggg=\bar{S}(n),$ then
         \begin{eqnarray*}
        \hspace{-9mm}[P(\mu):\Delta(\lambda)]&=&(K(\lambda+\Xi):L(\mu));\cr
		\hspace{-9mm} [T(\mu):\Delta(\lambda)]&=&(K(-w_0\lambda+\Xi):L(-w_0\mu)).
        \end{eqnarray*}
		\item[(3)] If $\ggg=\bar{H}(n),$ then
        \begin{eqnarray*}
        \hspace{-7mm}[P(\mu):\Delta(\lambda)]&=&(K(\lambda+ n\delta):L(\mu));\cr
		\hspace{-7mm}[T(\mu):\Delta(\lambda)]&=&(K(-w_0\lambda+ n\delta):L(-w_0\mu)).
        \end{eqnarray*}
	\end{itemize}
Here $w_0$ is the longest element of the Weyl group of $\ggg_0$, $\delta$ is the linear dual of extended toral element $\sfd$, $\Xi=\epsilon_{1}+\epsilon_{2}+\cdots+\epsilon_{n}$ and $(K(\cdot):L(\cdot))$ denotes the multiplicity of a composition factor in certain Kac module.
\end{thm}
From the above theorem,  Serganova's character formulas on Kac modules in \cite{Ser05}  allow us finally to  obtain the character formulas of both indecomposable projective and indecomposable tilting modules in $\comis$.

\subsection{} This paper is organized as follows. In Section \ref{yubei}, we introduce  some basic notions and notations for Lie superalgebras of Cartan type. Most notably,  we show in \S\ref{semi-inf} that the $\zz$-graded Cartan type Lie superalgebras admit semi-infinite characters.
In Section \ref{sec: bor-para}, we make a precisely construction of  adjacent Borel subalgebras, and then show the surprising result that any parabolic subalgebra containing  $\ggg_0$ is either the maximal one or the minimal one. Then the BGG category arising from any parabolic subalgebra is essentially  the subcategory of the one arising from the minimal.
In Section \ref{cat o}, we introduce the category $\comi$, investigate some natural representations and list some properties of their weights. These arguments are very important to the study of  blocks.
In Section \ref{proj and block}, we consider the projective modules in $\comi$. In particular, we establish that all simple objects in $\comi$ have projective covers, and every indecomposable projective module admits a flag of standard modules (Theorem \ref{projective thm}).
 In Section \ref{degen}, we  obtain a degenerate BGG reciprocity (Theorem \ref{deg bggthm}).
In  Section \ref{block sec}, we investigate and describe the blocks in $\comi$,  see Theorems \ref{block thm WS}, \ref{block thm H odd} and \ref{block thm H even}. In Section \ref{Costandard M}, we obtain a version of Soergel's reciprocity for indecomposable tilting modules via a realization of co-standard modules in terms of Kac modules. Then we apply the degenerate BGG reciprocity and Soergel's reciprocity to  give  character formulas of the indecomposable projective and the indecomposable tilting modules.  The last two sections are  appendixes, where we give a detailed computation for semi-infinite characters (Appendix A) and character formulas of tilting modules (Appendix B).

In the same spirit, it is also possible to extend parts of the theory of the category $\comi$  (including  tilting modules and their character theory)  to infinite-dimensional Lie algebras of Cartan type  (see \cite{DSY18}).

\subsection*{\textsc{Acknowledgement}} B.S. is deeply indebted to  Shun-Jen Cheng and Toshiyuki Tanisaki for stimulating and helpful discussions,  and to the Institute of Mathematics at Academia Sinica for their hospitality during his visit in the winter of 2018 when this work was partially done. The authors are also thankful to Salvatore Tringali (Hebei Normal University, School of Mathematics) for an attentive reading of the introduction of this paper.

\section{Preliminaries}\label{yubei}

In this paper, we always assume that the base field is the complex field $\mathbb{C}$. All vector superspaces (resp. supermodules, superalgebras) are over $\mathbb{C}$,  and will be simply called spaces (resp. modules, algebras).

\subsection{The Lie superalgebras  of Cartan type}

In this subsection, we recall the definitions of  finite-dimensional Lie superalgebras of Cartan type (see \cite{Kac77} for details).

Let $\Lambda(n)$ be the Grassmann superalgebra on $n$ odd generators $\xi_1,\ldots, \xi_n$ ($n\geq 2$).  Let $\mathrm{deg}(\xi_i)=1$ for $1\leq i\leq n$. Then $\Lambda(n)$ has a natural $\mathbb{Z}$-grading with $\Lambda(n)_j=\text{span}_{\mathbb{C}}\{\xi_{k_1}\wedge\cdots\wedge\xi_{k_j}\mid 1\leq k_1<\cdots<k_j\leq n\}$. The Witt type Lie superalgebra $W(n)$ is defined to be the set of all superderivations of $\Lambda(n).$ Then
$$W(n)=\left\{ \sum_{i=1}^{n}f_iD_i \mid f_i\in\Lambda(n)  \right\},$$
where $D_i$ is the superderivation of $\Lambda(n)$ defined through $D_i(\xi_j)=\delta_{ij}$ for $1\leq i,j\leq n.$ The Witt type Lie superalgebra $W(n)$ has a natural $\mathbb{Z}$-grading with
\begin{eqnarray}\label{w(n)}
W(n)_j=\left\{\sum_{i=1}^{n}f_iD_i \mid f_i\in\Lambda(n)_{j+1}\right\}.
\end{eqnarray}
Let \textbf{div} be the divergence mapping from the Witt type Lie superalgebra $W(n)$ to the  Grassmann superalgebra $\Lambda(n)$ defined as:
\[
\begin{array}{rcl}
\textbf{div}:W(n)&\rightarrow&\Lambda(n)\\
\sum\limits_{i=1}^{n}f_iD_i&\mapsto&\sum\limits_{i=1}^{n}D_i(f_i).\\
\end{array}
\]
The special Lie superalgebra $S(n)$ is defined as the Lie subalgebra of $W(n)$, consisting of all elements $x\in W(n)$ such that $\textbf{div}(x)=0$. Since the divergence mapping is a homogeneous operator of degree 0, the special Lie superalgebra $S(n)$ inherits the $\mathbb{Z}$-gradation of $W(n)$, i.e., $S(n)=\bigoplus\limits_{i=-1}^{n-2}S(n)_i$, where $S(n)_i=W(n)_i\cap S(n)$. Now we introduce the mapping $D_{ij}:$
\[
\begin{array}{ccll}
D_{ij}:&\Lambda(n)&\rightarrow& W(n)\\
& f        &\mapsto & D_{i}(f)D_j+D_{j}(f)D_i.
\end{array}
\]
We can check that $S(n)$ is the $\mathbb{C}$-linear span of the elements
belonging to $\{D_{ij}(f)\mid f\in\Lambda(n), 1\leq i,j\leq n\}$.

Up to isomorphism, there is a different class of simple Lie superalgebras of another special type  $\tilde{S}(n)$.  The Lie superalgebra $\tilde{S}(n)$ is defined only for even $n$, and it consists of all  $x\in W(n)$ such that
$$ (1+\xi_1\cdots\xi_n)\textbf{div}(x)+ x(\xi_1\cdots\xi_n)=0.$$
It is not a $\zz$-graded subalgebra of  $W(n)$ as the defining condition is not homogeneous.
Hence we ignore $\tilde{S}(n)$ in this paper.

Next, we introduce the Hamiltonian Lie superalgebra $H(n)$ with $n\geq 5$ (Note that $H(4)\cong A(1,1)$. We do not care about this case in the present paper. So we assume $n\geq 5$ for type $H$).
Assume that $n=2r$ or $n=2r+1,$ set
\begin{equation*}
i'=\begin{cases}
i+r, &\text{if}\,\, 1\leq i\leq r;\cr
i-r, &\text{if}\,\, r+1\leq i\leq 2r;\cr
i, &\text{if}\,\, i=2r+1.\cr
\end{cases}
\end{equation*}
 The Hamiltonian operator $D_H$ from the  Grassmann superalgebra $\Lambda(n)$ to the Witt Lie superalgebra $W(n)$ is defined as:

\[
\begin{array}{llcl}
D_H:&\Lambda(n)&\rightarrow& W(n)\\
    &f&\mapsto&D_{H}(f)=\sum_{i=1}^{n}(-1)^{\bar{f}}D_{i}(f)D_{i'},
\end{array}
\]
where $f$ is a homogeneous element in $\Lambda(n)$ and $\bar{f}$ denotes the parity of $f.$
Set $\chn=\{D_{H}(f)\mid f\in\Lambda(n) \}.$ Then the Hamiltonian Lie superalgebra $H(n)$ is by definition, the derived algebra of $\chn$, i.e.,
\begin{align}\label{Hamilton sub}
H(n)=[\chn,\chn)]; \;\; \chn=H(n)+\bbc D_H(\xi_1\cdots\xi_n),
\end{align}
 which can be further described as follows
$$
H(n)=\left\{ D_{H}(f)\mid f\in \sum_{i=0}^{n-1}\Lambda(n)_i\right \}.$$
Moreover, $H(n)$ is a $\mathbb{Z}$-graded subalgebra of $W(n)$ with $H(n)=\bigoplus\limits_{i=-1}^{n-3} H(n)_i$, where $H(n)_i=W(n)_i\cap H(n)$. Generally, for a graded subalgebra $L=\sum_{i=-1}^\infty L_i$ of $W(n)$, we set $L_{\geq j}:=\sum_{i\geq j }L_i$. Especially, we have
 the following structure
\begin{align}\label{gra -1}
L_{-1}=\sum_{i=1}^n\bbc D_i \;\;\mbox{ for } L=X(n), X\in\{W,S,H\}.
\end{align}

Let $L=W(n), S(n)$ or $H(n)$. By the following canonical map

\[
\begin{array}{ccc}
L_0&\rightarrow&\frak{gl}(n)\\
\sum_{1\leq i,j\leq n}k_{ij}\xi_iD_j&\mapsto& \sum_{1\leq i,j\leq n}k_{ij}E_{ij},\\
\end{array}
\]
we get
\begin{equation}\label{grading 0}
L_{0}\cong\begin{cases}
\gl,&\text{if}\,\, \ggg=W(n);\cr
\frak{sl}(n),&\text{if}\,\, \ggg=S(n);\cr
\frak{so}(n),&\text{if}\,\, \ggg=H(n),
\end{cases}
\end{equation}
and correspondingly have the standard triangular decomposition $L_{0}=\nnn^-\oplus\hhh\oplus\nnn^+$.
\subsection{Toral extension $\bar S(n)$, $\bar H(n)$ and $\ochn$}\label{toral extension} Set $\sfd=\sum_{i=1}^{n}\xi_iD_i$. Then $\sfd$ is a canonical toral element of $W(n)$. The element $\sfd$ measures the degrees of homogenous spaces of $W(n)$, thereby it normalizes any graded subalgebra $\sss$ of $W(n)$, i.e., $[\sfd, \sss]\subset \sss$. Set

\begin{align*}\label{bar SH}
&\bar{S}(n)=S(n)\oplus\mathbb{C}\sfd, \cr &\bar{H}(n)=H(n)\oplus\mathbb{C}\sfd,\cr
&\ochn=\chn\oplus \bbc\sfd,
\end{align*}
and $\bar{\hhh}:=\hhh\oplus\mathbb{C}\sfd$ for $\ggg=\bar S(n)$, $\bar H(n)$ or  $ \overline{C{\hskip-0.1cm}H}(n)$, $\bar{\hhh}:=\hhh$ for $\ggg=W(n)$.
We then have  the following standard basis of $\bar\hhh$:
\begin{equation}\label{eq: h integer basis}
\begin{cases}
\{\xi_iD_i\mid 1\leq i\leq n\}, &\text{if}\,\, \ggg=W(n),\bar{S}(n);\cr
\{\xi_{i}D_{i}-\xi_{i+r}D_{i+r},\sfd\mid 1\leq i\leq r\}, &\text{if}\,\, \ggg=\bar{H}(2r), \overline{C{\hskip-0.1cm}H}(2r),\bar{H}(2r+1), \overline{C{\hskip-0.1cm}H}(2r+1),
\end{cases}
\end{equation}
whose dual basis can be described as follows:
\begin{equation*}
\begin{cases}
\{\epsilon_i\mid 1\leq i\leq n\}, &\text{if}\,\, \ggg=W(n),\bar{S}(n);\cr
\{\epsilon_i, \delta\mid 1\leq i\leq r\}, &\text{if}\,\, \ggg=\bar{H}(2r), \overline{C{\hskip-0.1cm}H}(2r),\bar{H}(2r+1), \overline{C{\hskip-0.1cm}H}(2r+1).
\end{cases}
\end{equation*}
This means
$\epsilon_i (\xi_j D_j)=\delta_{ij}$ when $\ggg=\wn$ or $\bsn$; and $\epsilon_i(\xi_jD_j-\xi_{j+r}D_{j+r})=\delta_{ij}$, $\epsilon_i(\sfd)=0$, $\delta(\xi_{j}D_{j}-\xi_{j+r}D_{j+r})=0$ for $1\leq i,j \leq r$, and $\delta(\sfd)=1$ when $\ggg=\bar{H}(2r), \overline{C{\hskip-0.1cm}H}(2r),\bar{H}(2r+1), \overline{C{\hskip-0.1cm}H}(2r+1).$

 Set $V=\sum_{i=1}^n\bbc\xi_i$. 
 We can further regard
\begin{equation}\label{grading 0}
\ggg_{0}=\begin{cases}
\frak{gl}(V),&\text{if}\,\, \ggg=W(n), \bar S(n);\cr
\frak{so}(V)+\bbc\sfd,&\text{if}\,\, \ggg=\bar H(n), \ochn.
\end{cases}
\end{equation}

\begin{convention}\label{conv}
In the sequel, \textit{whenever the context is clear, we don't distinguish $\epsilon_i$ and $\epsilon_i|_{\hhh}$
for $1\leq i\leq m$.
Here,  $m=n$ when $\ggg=W(n)$ or $\bar S(n)$; $m=r$ when $\ggg=\bar H(n)$, $\overline{C{\hskip-0.1cm}H}(n)$}.
\end{convention}

\subsection{Root systems and closed subalgebras}\label{root sys} Note that the Cartan subalgebras of $\ggg$ coincide with the Cartan subalgebras of $\ggg_0$.
Associated with the Cartan subalgebra $\bar\hhh$, there is a root system $\Phi(\ggg)$ and the corresponding root space decomposition $\ggg=\bar\hhh+\sum_{\alpha\in \Phi(\ggg)}\ggg_\alpha$  for $\ggg=X(n)$ $(X\in \{W, \bar S, \bar H, \och\})$. The root system $\Phi(\ggg)$ can be described as below.

\begin{itemize}
\item[(1)]  For $\ggg=\wn$, $\Phi(\ggg)=\{\epsilon_{i_1}+\cdots+\epsilon_{i_k}-\epsilon_j\mid 1\leq i_1<\cdots<i_k\leq n; k=0,1,...,n;\; 1\leq j\leq n\}$.

     \item[(2)] For $\ggg=\bsn$, $\Phi(\ggg)=\Phi(\wn)\backslash \{(\sum_{i=1}^n\epsilon_i)-\epsilon_j\mid j=1,...,n\}$.
 \item[(3)] For $\ggg=\bar H(2r)$,
 \begin{align*}
 \Phi(\ggg)=\{ \pm\epsilon_{i_1}\pm\cdots\pm \epsilon_{i_k} +l\delta\mid \;&1\leq i_1<i_2 <\cdots< i_k \leq r;\\
 &k-2\leq l<n-2,  l-k \in 2\bbz\}.
 \end{align*}
 \item[(4)] For $\ggg=\bar H(2r+1)$,
 \begin{align*}
 \Phi(\ggg)=\{ \pm\epsilon_{i_1}\pm\cdots\pm \epsilon_{i_k} +l\delta\mid \;&1\leq i_1<i_2 <\cdots< i_k \leq r; \\&k-2\leq l<n-2\}.
 \end{align*}
 \item[(5)] For $\ggg=\ochn$, $\Phi(\ggg)=\Phi(\bhn)\cup \{ (n-2)\delta\}$.
\end{itemize}

In particular, $\Phi_0$ (resp. $\Phi_0^+$) will denote the root system of $\ggg_0$ (resp. $\nnn^+$). Correspondingly, we  have the Borel subalgebra $\bbb=\bar\hhh\oplus \nnn^+$.

A subset $\Psi$ of $\Phi$ is called a closed one if for any $\alpha, \beta\in \Psi,$ we always have $\alpha+\beta\in \Psi$ provided that $\alpha+\beta\in \Phi$. We say  a subalgebra $\frak{q}$ of $\ggg$ to be closed if there is a closed subset $\Psi$ of $\Phi$ such that $\frak{q}=\hhh+\sum_{\alpha\in\Psi}\ggg_\alpha$.

We also need a convention $\Xi\in \hhh^*$ for $\ggg=W(n)$ or $\bar S(n)$ which means $\sum_{i=1}^n\epsilon_i$.

\subsection{Semi-infinite characters}\label{semi-inf}
\begin{defn}
Let $\ggg=\sum_{i\in\zz} \ggg_{i}$ be a $\zz$-graded Lie superalgebra with $\dim \ggg_{i}<\infty$ for all $i\in\zz$. A character $\gamma: \ggg_{0} \rightarrow \mathbb{C}$ is called a semi-infinite character for $\ggg$ if the following items are satisfied.
\begin{itemize}
	\item[(SI-1)] As a Lie superalgebra, $\ggg$ is generated by $\ggg_{1}, \ggg_{0}$  and $\ggg_{-1}$;
	\item[(SI-2)] $\gamma([x,y])=\textsf{str}\big((\textsf{ad}x\circ \textsf{ad}y)|_{\ggg_{0}}\big)$, $\forall\, x\in \ggg_{1}$ and $y\in \ggg_{-1}$.
\end{itemize}
\end{defn}

Now we turn to $\ggg=X(n)$ for $X\in \{W, \bar S, \bar H, \och\}$.
We define $\mathcal{E}_W:\ggg_{0}\longrightarrow\mathbb{F}$ to be a linear map with $\mathcal{E}_W(\xi_iD_j)=-\delta_{ij}$ for $1\leq i,j\leq n$. Set $\mathcal{E}_{\bar{S}}=\mathcal{E}_{\bar{H}}=\mathcal{E}_{\och}=0$. By a direct computation, it is not hard (but tedious) to verify the following fact.

\begin{lem}\label{semi-inf} The linear map $\mathcal{E}_X$ is a semi-infinite character for $X(n)$, where  $X\in \{W, \bar S, \bar H, \och\}$.
\end{lem}

\begin{proof} The proof is left in Appendix A.
\end{proof}

\section{Borel subalgebras and  parabolic subalgebras}\label{sec: bor-para} In the following we denote $\Phi:=\Phi(\ggg)$ if the context is clear.
Following Serganova (\cite{Ser05}), we call a root $\alpha\in \Phi$ nonessential if $-\alpha\notin \Phi(\ggg)$, and essential if $-\alpha\in\Phi(\ggg)$. By \eqref{eq: h integer basis}, we have $\bar\hhh=\bbc\otimes_\bbz \bbz\text{-span}\{\text{standard basis}\}$.  Define
$$\bar\hhh_\bbr=\bbr\otimes_\bbz \bbz\text{-span}\{\text{standard basis}\}. $$
Call $h\in \bar\hhh_\bbr$ regular if $\alpha(h)\ne 0$ for all $\alpha\in \Phi(\ggg)$. According to \cite{Ser05}, any regular $h$ deduces a subdivision $\Phi=\Phi_h^+\cup \Phi_h^-$, where $\Phi_h^\pm=\{\alpha\in \Phi\mid \alpha(h)\in \bbr^{\pm}\}$. That defines a triangular decomposition $\ggg=\frakn_h^+\oplus \bar\hhh\oplus \frakn_h^-$ for $\frakn_h^\pm=\sum_{\alpha\in \Phi_h^\pm}\ggg_\alpha,$  where  $\ggg_\alpha$ is the corresponding root space. A Borel subalgebra $\frakb:=\frakb_h$ is defined as $\bar\hhh\oplus \frakn_h^+$. Sometimes, we write  $\Phi_h^+$ as $\Phi(\frakb)$ if without any  confusion.
 There are only finitely many Borel subalgebras (containing the given $\bar\hhh$).
 An $\alpha\in \Phi_h^+$ is called simple for the Borel subalgebra $\frakb$ if after removing $\alpha$ from $\Phi_h^+$ and adding  $-\alpha$ (if it does exist) we obtain a set of positive roots for some other Borel subalgebra $\frakb'$.
 In this case, we call $\frakb$ and $\frakb'$ are adjacent, and related by even reflection if $\alpha$ is even essential, by odd reflection if $\alpha$ is odd essential, by nonessential reflection if $\alpha$ is nonessential. Denote $\frakb'=r_\alpha(\frakb)$.
For any two Borel subalgebras (containing $\bar\hhh$), one is linked to the other one by a chain of reflections (see  \cite{Ser05}).

\subsection{Borel subalgebras containing $\bbb$ and strongly regular toral elements} In the following, what we are interested in are Borel subalgebras $\frakb$ containing $\bbb$. Among such Borel subalgebras we
distinguish $\frakb_{\text{max}}= \bbb+\sum_{i>0}\ggg_i$ and $\frakb_{\text{min}} =\bbb+\ggg_{-1}$, whose dimensions are of maximal and minimal respectively.

 A defining toral element $h\in \hhh_\bbr$ of a Borel subalgebra containing $\bbb$ is said to be strongly regular.
We will denote by $\diag(a_{1},\ldots, a_{n})$ the diagonal matrix of size $n\times n$ with the entries $a_{i}$ on the $i$th diagonal positions. The toral element $h$ can be identified with $\diag(a_{1},\ldots, a_{n})$.
The following facts are clear.

\begin{lem}\label{lem: str reg} Let $\ggg=W(n)$ or $\bar S(n)$. Suppose that $h=\diag(a_1,\ldots,a_n)\in \bar\hhh_\bbr$ is strongly regular. The following statements hold.
\begin{itemize}
\item[(1)]  The toral element $h$ satisfies $a_i>a_j$ for $1\leq i<j\leq n$.
     \item[(2)] If $\frakb_h=\frakb_\max$, then $a_i+a_j>a_k$ for any different $i,j,k\in \{1,2,\ldots n\}$.
         \item[(3)] If $\frakb_h\nsupseteqq \frakb_\min$, then $\epsilon_1\in \Phi_h^+$.
         \end{itemize}
\end{lem}
\begin{proof} (1) It follows from the fact  $\Phi_0^+=\{\epsilon_i-\epsilon_j\mid 1\leq i<j\leq n\}$.

(2) This is due to the fact that $\epsilon_i+\epsilon_j-\epsilon_k\in \Phi^+$.

(3) Suppose $\epsilon_1\notin \Phi_h^+$. Then  $a_1<0$. By (1), we have all $a_i<0, 1\leq i \leq n$. Correspondingly, $\frakb_h$ contains $\ggg_{-1}$. Hence $\frakb_h\supset \frakb_\min$.
\end{proof}

Before the arguments on $\ggg=\bar H(n)$, we need the following information on the root set $\Phi(\ggg_{-1})$ of $\ggg_{-1}$.
\begin{align}\label{eq: minus roots}
\Phi(\ggg_{-1})=\begin{cases}
 \{\pm\epsilon_i-\delta\mid i=1,\ldots,r\},  &\mbox{ for } \bar H(2r);\\
 \{\pm\epsilon_i-\delta\mid i=1,\ldots,r\}\cup\{ -\delta\},  &\mbox{ for } \bar H(2r+1).
   \end{cases}
   \end{align}
\begin{lem}\label{lem: str reg H} Let $\ggg=\bar H(n)$ with $n=2r$ or $n=2r+1$. Suppose that $h$ is strongly regular with $h=\diag(a+a_1,\ldots,a+a_r;a-a_1,\ldots,a-a_r)\in \bar\hhh_\bbr$ when $n=2r$, or $h=\diag(a+a_1,\ldots,a+a_r;a-a_1,\ldots,a-a_r, a) \in \bar\hhh_\bbr$ when $n=2r+1$. The following statements hold.
\begin{itemize}
\item[(1)]  The toral element $h$ satisfies $a_i>a_j$, $a_i+a_j>0$ for $1\leq i<j\leq r$. Additionally, $a_r>0$ for $\ggg=\bar H(2r+1)$.
     \item[(2)] If $\frakb_h=\frakb_\max$,  $(n-3)a>a_1$.
         \item[(3)] If $n=2r$ and $\frakb_h$ does not contain $\ggg_{-1}$, then $\epsilon_1+\delta\in \Phi_h^+$, or $-\epsilon_r+\delta\in\Phi_h^+$.
         \item[(4)] If $n=2r+1$ and $\frakb_h$ does not contain $\ggg_{-1}$, then one of the following
         items occurs.\begin{itemize}
        \item[(i)] $\delta\in\Phi_h^+$.
        \item[(ii)] $-\delta\in\Phi_h^+$, and either $\epsilon_1+\delta\in \Phi_h^+$ or $-\epsilon_r+\delta\in\Phi_h^+$.
        \end{itemize}
     \end{itemize}
\end{lem}
\begin{proof} (1), (2) By the same reason as in the proof of Lemma \ref{lem: str reg}, the first two statements are clear.

(3) Suppose $\epsilon_1+\delta\notin \Phi_h^+$ and $-\epsilon_r+\delta\notin\Phi_h^+$, then $-\epsilon_1-\delta\in \Phi_h^+$ and $\epsilon_r-\delta\in\Phi_h^+$. This implies that $-a_1-a=(-\epsilon_1-\delta)(h)>0$ and $a_r-a=(\epsilon_r-\delta)(h)>0$. Hence,
$a_1-a>a_2-a>\cdots>a_r-a>0$ and $-a_r-a>-a_{r-1}-a>\cdots>-a_1-a>0$ by (1). Consequently, $\Phi(\ggg_{-1})\subseteq \Phi_h^+$ by (\ref{eq: minus roots}), and $\ggg_{-1}\subseteq \frakb_h$, a contradiction.

(4) Suppose $\delta\notin\Phi_h^+$, then $-\delta\in\Phi_h^+$. Assume in contrary that $\epsilon_1+\delta\notin \Phi_h^+$ and $-\epsilon_r+\delta\notin\Phi_h^+$. Similar arguments as in (3) yield that $\{\pm\epsilon_i-\delta\mid i=1,\ldots,r\}\subseteq\Phi_h^+$. Hence $\Phi(\ggg_{-1})\subseteq \Phi_h^+$ by (\ref{eq: minus roots}), and $\ggg_{-1}\subseteq \frakb_h$, a contradiction.
\end{proof}



\subsection{Variation of Borel subalgebras from $\frakb_\max$ to $\frakb_\min$ for $\ggg=W(n)$ or $\bar S(n)$}\label{lem:borel chain}
  Certainly, it is interesting and nontrivial to construct an adjacent chain of  Borel subalgebras. It is a good way to do that  via strongly regular toral elements.  Here, we list them for $\ggg=W(n)$ and $\bar S(n)$.
  \subsubsection{} Set
\begin{align*}
h_\max&=\diag({n\over n+1},{n-1\over n},{n-2\over n-1},\ldots,{2\over 3}, {1\over 2}),\cr
h_\min&=\diag(-{1\over 2},-{2\over 3},-{3\over 4},\ldots,-{n-1\over n}, -{n\over n+1}).
\end{align*}
By a straightforward computation, $h_\max$ and $h_\min$ are exactly the defining strongly regular toral elements in $\bar\hhh_\bbr$ for $\frakb_\max$ and $\frakb_\min$ respectively. Now we can show that there are  a sequence of reflections $r_\alpha,\ldots, r_\gamma$ such that $\frakb_{\text{min}}=r_{\alpha}(\cdots (r_{\gamma}(\frakb_\max)))$. Actually, we take a sequence of regular toral elements in $\hhh_\bbr$ as below. Set $h_0=h_\max$, and consider
\begin{align*}
h_r=\diag({n\over n+1},{n-1\over n},{n-2\over n-1},\ldots,{r+1\over r+2}; -{1\over 2},-{2\over 3},-{3\over 4},\ldots, -{r\over r+1}),
\end{align*}
$r=1,\ldots,n$.
 Naturally $h_n=h_\min$. Readers can verify that each  $h_r$ is strongly regular and the corresponding  positive root set is
\begin{align}\label{eq: h-root first}
\Phi_{h_r}^+=\Phi_0^+\cup \Phi^+\cup (-\Pi_{r})\backslash \mathbb{X}_{r}
\end{align}
where
\begin{align*}
\Pi_{r}=&\{\epsilon_{n-r+j}\mid j=1, \ldots,r\},  \cr -\Pi_{r}=&\{-\epsilon_{n-r+j}\mid j=1,\ldots,r\}, \cr \mathbb{X}_{r}=&\{\sum_{q=1}^k\epsilon_{i_q}+ \sum_{q=1}^l\epsilon_{n-r+j_q}-\epsilon_d
\mid(*)_{r}<0 \text{ for } 2\leq l+k\leq n, \cr
&1\leq i_1<\cdots<i_k\leq n-r; 1\leq j_1<\cdots<j_l\leq r; \; 1\leq d\leq n\}.
\end{align*}
with
\begin{align*}
(*)_{r}=\begin{cases} \sum_{q=1}^k {n-i_q+1\over n-i_q+2}-\sum_{q=1}^l {j_q\over j_q+1}- {d\over d+1} \;&\text{ if } 1\leq d\leq n-r;\cr
\sum_{q=1}^k {n-i_q+1\over n-i_q+2}-\sum_{q=1}^l {j_q\over j_q+1}+{d'\over d'+1} \;&\text{ if } d=n-r+d' \text{ with }1\leq d'\leq r.
\end{cases}
\end{align*}
Obviously, $\mathbb{X}_{r}  \supset  \Pi_{r}$.
\subsubsection{}  Now we continue to refine the above process.
Recall $h_1=\diag({n\over n+1},{n-1\over n},{n-1\over n-3},\ldots,{2\over 3}; -{1\over 2})$. Set $h_1^{(1)}:=h_1$ and for $q=2,\ldots, n$,
$$h_1^{(q)}:=\diag({n\over n+1},\ldots,{q\over q+1},{qn-1\over 2qn},\ldots,{n-1\over 2n};-{1\over 2}).$$
Note that from the beginning, we have set an appointment  $n>2$. So it is easily known that $h_1^{(q)}$ is strongly regular.
Inductively, for a given $r>1$ set $h_r^{(r)}:=h_r$ and
$$h_r^{(q)}:=\diag({n\over n+1},\ldots,{q\over q+1},{rqn-1\over (r+1)qn},\ldots,{r(r+1)n-1\over (r+1)(r+2)n}; -{1\over 2},-{2\over 3},-{3\over 4},\ldots, -{r\over r+1}) $$
for $q=r+1,\ldots,n$. All of $h_r^{(q)}$'s are strongly regular. The corresponding Borel subalgebra of $h_r^{(q)}$ is denoted by $\frakb_{r}^{(q)}$, $r=0,1,\ldots,n$, $q=r,r+1,\ldots,n$.
In particular, $\frakb_0=\frakb_\max$ and  $\frakb_{n}=\frakb_\min$ (here set $\frakb_{r}=\frakb_{r}^{(r)}$).
%
%

\subsection{Parabolic subalgebras containing $\ggg_0$}
 For a given strongly regular element $h$, we have
 $\Phi=\Phi_h^+\cup \Phi_h^-$ and the corresponding Borel subalgebra
$$\frakb_h=\bbb\oplus \sum_{\alpha \in \Phi_h^+ \backslash \Phi_{0}^+}\ggg_\alpha.$$
{\sl{We define a parabolic subalgebra $\sfp_h$ associated with $h$ (and then with $\frakb_h$) as the closed subalgebra generated by $\ggg_0$ and $\frakb_h$.}}

Associated with the maximal Borel subalgebra $\frakb_{\text{max}}=\bbb+\sum_{i>0}\ggg_i$, and the minimal Borel subalgebra $\frakb_{\text{min}}=\bbb+\ggg_{-1}$,
it is readily known that the corresponding  parabolic subalgebras are, respectively,
$$ \sfp_{\text{max}}=\ggg_0+\sum_{i>0}\ggg_i \mbox{ and } \sfp_{\text{min}}=\ggg_0+\ggg_{-1}.$$

The minimal parabolic subalgebra will be the most interesting, playing a crucial role in the theory of parabolic BGG categories. The following basic observation preliminarily reveals its importance.

\begin{prop}\label{prop: det para} Let $\ggg=X(n)$, $X\in \{W,\bar S, \bar H, \och \}$. Then any  proper parabolic subalgebra coincides with either $\sfp_{\mathrm{max}}$ or $\sfp_{\mathrm{min}}$.
\end{prop}

\begin{proof}
Let $\sfp$ be an arbitrarily given parabolic subalgebra generated by $\ggg_0$ and $\frakb_h$, where
 \begin{align*}
h=\begin{cases} \diag(a_{1},\ldots, a_{n}), \;&\text{ if } \ggg=W(n), {\bar S}(n);\cr
\diag(a+a_1,\ldots,a+a_r;a-a_1,\ldots,a-a_r), \;&\text{ if } \ggg=H(2r);\cr
\diag(a+a_1,\ldots,a+a_r;a-a_1,\ldots,a-a_r, a), \;&\text{ if } \ggg=H(2r+1)
\end{cases}
\end{align*}
is a defining strongly regular toral element of $\frakb_h$. Then $a_{1}>a_{2}>\cdots>a_{n}$ for $ \ggg=W(n), {\bar S}(n)$, and $a_{1}>a_{2}>\cdots>a_{r}$ for $ \ggg=H(n)$ with $n=2r, 2r+1$. Denote by $\Psi_h$ the root set of $\sfp$. Then $\Psi_h$ is a closed root subsystem of $\Phi$.

Firstly we assume that $\sfp$ does not contain $\sfp_\min$. In this situation, we will show  $\sfp=\sfp_\max$. We proceed it case by case.

Case 1:  $ \ggg=W(n), {\bar S}(n)$.

By Lemma \ref{lem: str reg}(3),  $\Psi_h$ contains the root $\epsilon_1$. Note that by definition, $\Phi_0^-\subset \Psi_h$. Hence, $\Psi_h$ contains all $\epsilon_k=(\epsilon_{k}-\epsilon_1)+\epsilon_1$ for $k=2,\ldots,n$.
Correspondingly, $\Psi_h$ contains all $\epsilon_k$ for $k=1,\ldots,n$. Hence $\Psi_h$ contains all $\Phi^+$ because $\Psi_h$ is a closed root subsystem, containing $\Phi_0$ and all $\epsilon_k$, $k=1,\ldots,n$. This means that the parabolic subalgebra associated with $\frakb_h$ contains $\sfp_\max$. On the other hand, a parabolic subalgebra containing $\sfp_\max$ is either $\ggg$ itself or equal to $\sfp_\max$. We are done.

Case 2: $\ggg=H(2r)$.

By Lemma \ref{lem: str reg H}(3), $\Psi_h$ contains the root $-\epsilon_r+\delta$ or the root $\epsilon_1+\delta$. If $\epsilon_1+\delta\in\Psi_h$, then $-\epsilon_r+\delta=(-\epsilon_r-\epsilon_1)+(\epsilon_1+\delta)\in\Psi_h$, because $-\epsilon_r-\epsilon_1\in\Phi_0\subset \Psi_h$ and $\Psi_h$ is closed. Consequently, we see that $\Psi_h$ always contains the root $-\epsilon_r+\delta$. Since $\Phi_0=\{\pm\epsilon_i\pm\epsilon_j\mid 1\leq i\neq j\leq n\}\subset \Psi_h$ and $\Psi_h$ is closed, we have $\pm\epsilon_i+\delta=(\pm\epsilon_i+\epsilon_r)+(-\epsilon_r+\delta)\in\Psi_h$ for $1\leq i\leq r-1$. In addition, $\epsilon_r+\delta=(\epsilon_r+\epsilon_1)+(-\epsilon_1+\delta)\in\Psi_h$. Hence, $\pm\epsilon_j+\delta\in\Psi_h$ for any $1\leq j\leq n$. In particular, $2\delta=(\epsilon_1+\delta)+(-\epsilon_1+\delta)\in\Psi_h$, so that $2m\delta\in\Psi_h$ for any $m\in\mathbb{Z}^+$. Now let $\alpha=\pm\epsilon_{i_1}\pm\cdots\pm\epsilon_{i_k}+l\delta$ be an arbitrary root in $\ggg_{\geq 1}$, where $k-2\leq l\leq n-2$ and $l-k$ is even. Since $\alpha$ can be written as
$$\alpha=(\pm\epsilon_{i_1}\pm\epsilon_{i_2})+(\pm\epsilon_{i_3}+\delta)+\cdots+(\pm\epsilon_{i_k}+\delta)+(l-k+2)\delta,$$
we get that $\alpha\in\Psi_h$ by induction on $k$. This implies $\sfp_\max\subseteq\sfp$. On the other hand, a parabolic subalgebra containing $\sfp_\max$ is either $\ggg$ itself or equal to $\sfp_\max$. We are done.

Case 3: $\ggg=H(2r+1)$.

By Lemma \ref{lem: str reg H}(4), if $-\delta\in\Psi_h$, $\Psi_h$ contains the root $-\epsilon_r+\delta$ or the root $\epsilon_1+\delta$. While if $\delta\in\Psi_h$, then $\epsilon_1+\delta\in\Psi_h$, because $\epsilon_1\in\Psi_h$ and $\Psi_h$ is closed. Similar arguments as in Case 2 yield the desired assertion in this case.

Secondly we assume that $\sfp$ contains $\sfp_\min$. In this case, it suffices to show that $\sfp$ must coincide with $\ggg$ itself as long as $\sfp$ properly contains $\sfp_\min$. We also proceed it case by case.

Case 1:  $ \ggg=W(n), {\bar S}(n)$.

In this case, under the assumption $\sfp\supsetneqq\sfp_\min$, the root set $\Psi_h$ of $\sfp$ contains $\Phi(\sfp_\min)\cup\{\epsilon_{i_1}+\epsilon_{i_2}+\cdots+\epsilon_{i_t}-\epsilon_k\}$ for some sequence  $(1\leq i_1<i_2<\cdots<i_t\leq n)$ and $k\in\{1,2,\ldots,n\}$ with $t>1$.  By definition, $\Psi_h$ is closed. Note that $\Phi(\sfp_\min)=\{-\epsilon_i\mid i=1,\ldots,n\}\cup\{\epsilon_i-\epsilon_j\mid 1\leq i\ne j\leq n\}$ is already contained in $\Psi_h$. So it is easily deduced that all $\epsilon_i, i=1,\ldots,n,$ are contained in $\Psi(\sfp)$. Consequently, it is further deduced that $\Psi(\sfp)=\Phi$ and then $\sfp=\ggg$.

Case 2: $\ggg=H(n)$.

In this case, under the assumption $\sfp\supsetneqq\sfp_\min$, the root set $\Psi_h$ of $\sfp$ contains $\Phi(\sfp_\min)\cup\{\pm\epsilon_{i_1}+\cdots+\pm\epsilon_{i_k}+l\delta\}$ for some $k\geq 1$ and $k-2\leq l\leq n-2$. By definition, $\Psi_h$ is closed. Note that
\begin{align*}
\Phi(\sfp_\min)=\begin{cases} \{\pm\epsilon_i-\delta\mid i=1,\ldots,n\}\cup\{\pm\epsilon_i\pm\epsilon_j\mid 1\leq i\ne j\leq n\},\;&\text{ if } n=2r;\cr
\{\pm\epsilon_i-\delta, -\delta\mid i=1,\ldots,n\}\cup\{\pm\epsilon_i\pm\epsilon_j, \pm\epsilon_i\mid 1\leq i\ne j\leq n\},\;&\text{ if } n=2r+1
\end{cases}
\end{align*}
is already contained in $\Psi_h$. So it is easily deduced that all roots in $\ggg_{\geq 1}$ are contained in $\Psi_h$. Consequently, $\Psi_h=\Phi$, and then $\sfp=\ggg$.

The proof is completed.
\end{proof}

\begin{rem} There is a natural question when it is true that the subalgebra generated by $\ggg_0$ and a Borel subalgebra $\frakb_h$ is closed, i.e. it coincides with $\sfp_h$. This question can be positively answered for $\ggg=\bar{S}_n$  because in this case, all $\ggg_i$ ($i\ne 0$) are irreducible $\ggg_0$-modules (see \cite[Proposition 3.3.1]{Kac77}).
\end{rem}

\subsection{Parabolic categories}\label{parabolic func}
 In general,  we can consider a parabolic BGG category $\co_h$ of $\ggg$ associated with $\sfp_h$, whose objects are super $\ggg$-modules endowed with an admissible $\bbz$-graded structure,  locally finite over $\sfp_h$ and semisimple over $\bar\hhh$. The morphisms in $\co_h$ are  even homomorphisms of $\bbz$-graded $\ggg$-modules.

By Proposition \ref{prop: det para}, there are only two possibilities for a proper parabolic subalgebra $\sfp_h$, that is, it coincides with  either $\sfp_{\text{max}}$ or $\sfp_{\text{min}}$. If the objects of $\co_h$ are additionally required to be finitely generated over $U(\ggg)$, then it is readily seen that any objects in the BGG category arising from $\sfp_{\text{max}}$ is finite-dimensional. All such objects belong to the other BGG category $\comi$ arising from $\sfp_{\text{min}}$.

\section{The category $\comi$}\label{cat o}
\subsection{} From now on we always assume that $\ggg=X(n)$ with  $X\in \{W,\bar{S},\bar{H},\och\}$. Keeping in mind, we have $\ggg_0=\nnn^+\oplus \bar\hhh \oplus \nnn^-$ with $\bar\hhh$ defined  in \S\ref{toral extension}, and the minimal parabolic subalgebra $\sfp_{\text{min}}$ defined  in \S\ref{sec: bor-para}. From now on,  we simply write $\sfp=\sfp_{\text{min}}$.

\begin{defn}\label{defn}
We define a category  $\comi$  whose objects are $\bbz_2$-graded vector spaces $M=M_{\bar{0}}\oplus M_{\bar{1}}$ satisfying the following axioms:
\begin{itemize}
\item[(1)] $M$ is an admissible $\bbz$-graded $\ggg$-module, i.e., $M=\bigoplus_{i\in\mathbb{Z}}M_{i}$ with $M_i=(M_i\cap M_{\bar{0}})\oplus (M_i\cap M_{\bar{1}})$, $\dim M_i<\infty$, and $\ggg_iM_j\subseteq M_{i+j}, \forall\, i, j\in\mathbb{Z}$.
\item[(2)] $M$ is locally finite as  a $\sfp$-module.
\item[(3)] $M$ is semisimple over $\bar\hhh$.
\end{itemize}
The morphisms in $\comi$ are always assumed to be even (see Remark \ref{def rem}(4) below), and  they are $\ggg$-module morphisms compatible with the $\mathbb{Z}$-gradation, i.e., for any $M,N\in\comi,$
$$\Hom_{\comi}(M, N)=\{f\in \Hom_{U(\ggg)}(M, N)\mid f(M_i)\subseteq N_i, \forall\,i\in\mathbb{Z}\}.$$
\end{defn}

\begin{rem}\label{def rem}
{\rm(1)} Since $U(\sfp)\cong \bigwedge(\ggg_{-1})\otimes U(\ggg_0)$ and $\dim\bigwedge (\ggg_{-1})=2^n.$ The condition being locally
finite-dimensional over $\sfp$ is equivalent to being locally  finite-dimensional over $\ggg_0$.

{\rm(2)} The isomorphism classes of irreducible finite-dimensional $\ggg_0$-modules are parameterized by $\Lambda^+$, the set of the weights whose restriction to $[\ggg_0,\ggg_0]$ are dominant and integral. Denote by $L^0(\lambda)$ the finite-dimensional irreducible $\ggg_0$-module corresponding to $\lambda\in \Lambda^+$, which is a highest weight module associated with the Borel subalgebra $\bbb=\bar\hhh+\nnn^+$.

{\rm(3)} The $\bbz$-graded module category  of $\bar{X}(n)$ can be  naturally identified with the $\bbz$-graded module category of $X(n)$ ($X\in \{ S, H, C{\hskip-0.1cm}H\}$).

{\rm(4)} Recall that the Lie superalgebra  $\ggg$ is equal to $\ggg_\bz\oplus \ggg_\bo$ with $\ggg_\bz=\sum_{{\textmd{all even }} i}\ggg_i$, and $\ggg_\bo=\sum_{{\textmd{all odd }}i}\ggg_i.$
For any two $\ggg$-modules $M,N$, we say a homomorphism $\varphi: M\rightarrow N$ is of parity $|\varphi|$ if $\varphi(x m)=(-1)^{|\varphi||x|}x\varphi(m)$ for any $\bbz_2$-homogeneous element $x\in \ggg_{|x|}$, and $m\in M$. In this paper we always assume that the homomorphism $\varphi$ in $\comi$ is of even parity,  i.e.,  $\varphi(x m)=x\varphi(m)$ for any $x\in\ggg$ and $m\in M$.  So $\comi$ is an abelian  category.

(5) If forgetting  the $\bbz$-graded structure of $\comi$, then we have the category $\overline\comi$.  A $U(\ggg)$-module $M$ belongs to $\overline\comi$ if and only if it is a weight module and locally $\sfp$-finite. Denote by $\mathcal{F}$ the natural forgetful functor from $\comi$ to $\overline\comi$.
\end{rem}

Let $\mathbf{E}$ be a complete set of pairwise non-isomorphic irreducible $\bbz$-graded modules of $\ggg_0$. Each $E\in \mathbf{E}$  is necessarily concentrated in a single degree $\lfloor E\rfloor\in\bbz$, which means that $E=E_{\degE}$.  So, $\mathbf{E}$ can be parameterized by $\Lambda^+\times \bbz$.

Denote by $\comiid$ the full subcategory of $\comi$ consisting of all objects that are zero in degrees less than $d$ (called a truncated subcategory by $d$).

\subsection{Standard and co-standard modules} For a given $(\lambda,d)\in \mathbf{E}= \Lambda^+\times \bbz$, we have a $\bbz$-graded (finite-dimensional) irreducible $\ggg_0$-module $L^0(\lambda)$ whose degree $\lfloor L^0(\lambda)\rfloor$ is equal to $d$.  Let us introduce the standard
modules $\Delta(\lambda)$ and co-standard modules $\nabla(\lambda)$  in $\comi$ as below:
\begin{equation}\label{standard}
\Delta(\lambda)=U(\ggg)\otimes_{U(\sfp)}L^0(\lambda)
\end{equation}
and
\begin{equation}\label{costandard}
\nabla(\lambda)=\mathrm{Hom}_{\ggg_{\geq0}}(U(\ggg),L^0(\lambda)).
\end{equation}
with trivial $\ggg_{-1}$(resp. $\ggg_{\geq 1}$)-action  on $L^0(\lambda)$ in $\Delta(\lambda)$ (resp. $\nabla(\lambda)$).
Once the parity $|v^0_\lambda|$ of a maximal vector $v^0_\lambda$ in $L^0(\lambda)$ is given{\footnote{For $\comi$, one can give parities for  weight spaces similar to \cite[\S6]{CLW15}.}}, say $\epsilon\in \bbz_2=\{\bz,\bo\}$, the super-structure of $\Delta(\lambda)$ is determined by the super structure of $U(\ggg_{\geq 1})= U(\ggg_{\geq 1})_\bz \oplus  U(\ggg_{\geq 1})_\bo$ together with  $\epsilon$  as follows
     $$\Delta (\lambda)=\Delta(\lambda)_\bz \oplus  \Delta(\lambda)_\bo, \mbox{ where } \Delta(\lambda)_{\delta+\epsilon}= U(\ggg_{\geq 1})_{\delta}\otimes L^0(\lambda) \mbox{ for }\delta\in\bbz_2.$$

 Obviously, $U(\ggg)$ has a $\mathbb{Z}$-grading induced by the $\mathbb{Z}$-grading of $\ggg.$  So for $\Delta(\lambda)$, we have the following decomposition as a $\ggg_{0}$-module  $$\Delta(\lambda)\cong \bigoplus_{i\geq1}U(\ggg_{\geq1})_i\otimes_{\bbc}L^0(\lambda)$$
	where $U(\ggg_{\geq1})_i$ denotes the $i$th homogeneous part of $U(\ggg_{\geq1})$.
	Because $U(\ggg_{\geq1})_i,i\geq0,$ is finite-dimensional,
	$U(\ggg_{\geq1})_i\otimes_{\bbc}L^0(\lambda)$ is a finite-dimensional $\ggg_{0}$-module. Hence, $\Delta(\lambda)$ is locally finite over $\ggg_{0}$.
	Consequently, $\Delta(\lambda)$ is an object in $\comi$.
 As to the co-standard module, we have  the following isomorphisms over $\ggg_{\leq0}$:
\begin{align*}
\nabla(\lambda)&\cong \mathrm{Hom}_{\ggg_0}(U(\ggg_{\leq0}),L^0(\lambda))\cr
&\cong \mathrm{Hom}_{\ggg_0}(U(\ggg_{\leq0}),
\mathbb{C})\otimes_{\mathbb{C}}L^0(\lambda)\cr
&\cong \mathrm{Hom}_{\bbc}(\bigwedge(\ggg_{-1}),
\mathbb{C})\otimes_{\mathbb{C}}L^0(\lambda),
\end{align*}
where $\bigwedge(\ggg_{-1})$ denotes the exterior product space on the abelian Lie (super)algebra $\ggg_{-1}$, and the last isomorphism above is due to the fact that by definition $U(\ggg_{-1})=\bigwedge(\ggg_{-1})$. Hence $\dim\nabla(\lambda)=2^n\dim L^0(\lambda)$, and $\nabla(\lambda)$ is an object in $\comi$. Especially, $\nabla(\lambda)$ admits a simple socle $L^0(\lambda)$  over $\ggg_{\leq 0}$. For the detailed description of $\nabla(\lambda),$  readers can refer to Proposition \ref{iso2} later.


Furthermore,  both $\Delta(\lambda)$ and $\nabla(\lambda)$ belong to $\comidd$ as long as  $\lfloor L^0(\lambda)\rfloor=d\geq d'$. In this case, we say that {\sl{both of them have depth $d$}}. Generally, for $M\in \comiid$, define the depth of $M$ to be the least number $t$ with $M_t\ne 0$ for the gradation $M=\sum_{i=d}^\infty M_i$. Denote by $\dpt(M)$ the depth of $M$. By  definition, $\dpt(M)\geq d$
 for $M\in \comiid$.
The following basic observation is clear.

\begin{lem}\label{sts-costa indec} Both $\Delta(\lambda)$ and $\nabla(\lambda)$ are indecomposable.
\end{lem}

Actually,  it is readily known that $\Delta(\lambda)$ (resp. $\nabla(\lambda)$) has a unique maximal submodule. Hence, $\Delta(\lambda)$ (resp. $\nabla(\lambda)$) has a unique simple quotient, which is denoted by $L(\lambda)$ (resp. $\tilde L(\lambda)$).

\begin{lem} \label{basic lem O pre}
Maintain the notations as above. Then $\{L(\lambda)\}_{(\lambda,d)\in \bfe}$  and $\{\tilde L(\lambda)\}_{(\lambda,d)\in \bfe}$ are two complete sets of pairwise non-isomorphic irreducibles in $\comi$ respectively. Hence every simple object in $\comi$ is finite-dimensional.
\end{lem}

\begin{proof}   Let $E$ be any simple object of $\comi$ and $v$ be a non-zero weight vector belonging to $E$. Consider the finite-dimensional  $U(\sfp)$-module  $U(\sfp).v.$ Obviously, $U(\sfp).v$ has a non-zero $U(\sfp)$-irreducible  submodule $E_0$. Assume that $E_0$ is isomorphic to  $L^0(\lambda)$ for some $\lambda\in \Lambda^+$, then we have a non-zero homomorphism
from $\Delta(\lambda)$ to $E$.
Hence $E$ is isomorphic to $L(\lambda)$, with the depth of $E$ equal to $\lfloor E_0\rfloor$.

On the other hand, assume that $L(\lambda)$ and $L(\mu)$ are two irreducible modules with depths $d_\lambda$ and $d_\mu$ respectively. By the construction, $L^0(\lambda)$ is the unique simple socle of $L(\lambda)$ over  $\ggg_{\leq 0}$. If $L(\lambda)$ and $L(\mu)$ are isomorphic, then $L^0(\lambda)$ and $L^0(\mu)$ must be isomorphic as  $\ggg_{\leq 0}$-modules. Hence $\lambda=\mu$. Naturally, $d_\lambda=d_\mu$. Thus, we already prove that the set $\{L(\lambda)\}_{(\lambda,d)\in \bfe}$  forms a complete set of pairwise non-isomorphic simple objects in $\comi$.

By the same arguments, one can similarly prove the statement for  $\{\tilde L(\lambda)\}_{(\lambda,d)\in \bfe}$.
Since $\nabla(\lambda)$ is finite-dimensional, any simple object of $\comi$ is finite-dimensional.
\end{proof}

\begin{rem} \label{simple module aw csd} (1) By the above lemma, we can see that for any  $\lambda\in  \Lambda^+$ (modulo the depths), there is a unique $\tilde \lambda\in \Lambda^+$
	such that $L(\lambda)\cong\tilde L(\tilde\lambda)$.
	  Thus, the correspondence sending $\lambda$ to $\tilde{\lambda}$ gives rise to  a permutation on $ \Lambda^+$. The precise description can be given in \S\ref{appendix} with aid of Proposition \ref{iso2}.

(2) For $M\in \comi$, we write $(M: L(\lambda))$  for the multiplicity of the simple object $L(\lambda)$ in $M$, i.e., the supremum of $\#\{i\mid M^i\slash M^{i-1}\cong L(\lambda)\}$ over all finite filtration $\{M=M^k\supset...\supset M^i\supset M^{i-1}\supset...\supset M^1\supset  M^0=0\mid i\in\bbz_{>0}\}$. Suppose $M=\oplus_{k\in\bbz} M_k$. Note that $\dim M_{d}<\infty$ for $d=\lfloor L(\lambda)\rfloor$. So $(M:L(\lambda))$ is finite. Especially, we will call $L(\lambda)$ a composition factor of $M$ if $(M:L(\lambda))$ is nonzero.
\end{rem}

\subsection{Some natural representations and related notations}\label{nat rep sec}
We collect some basic facts on natural representations of $\ggg=X(n)$ for $X\in\{W, \bar S, \bar H, \och\}$, which will be used later for the study of blocks of $\comi$.

Recall that for $\ggg=\sum_{i\geq -1}\ggg_i$,  the graded subspaces $\ggg_{-1}=\sum_{i=1}^n\bbc D_i$ and
\begin{equation}\label{grading 0}
\ggg_{0}=\begin{cases}
\frak{gl}(V),&\text{if}\,\, \ggg=W(n), \bar S(n);\cr
\frak{so}(V)+\bbc\sfd,&\text{if}\,\, \ggg=\bar H(n), \ochn
\end{cases}
\end{equation}
for $V=\sum_{i=1}^n\bbc\xi_i$. Especially, $\ggg_{-1}$ becomes the contragredient module $V^*$ of $V$ over $\ggg_0$ with the weight set
$$\wt(\ggg_{-1})=\begin{cases}\{-\epsilon_1,...,-\epsilon_n\}, &\mbox{ for } X(n), X\in\{W, \bar S\};\\
 \{-\epsilon_1-\delta,...,-\epsilon_r-\delta, \epsilon_1-\delta,...,\epsilon_r-\delta\},  &\mbox{ for } \bar H(2r) \mbox{ or }\overline{C{\hskip-0.1cm}H}(2r);\\
 \{-\epsilon_1-\delta,...,-\epsilon_r-\delta, \epsilon_1-\delta,...,\epsilon_r-\delta, -\delta\},  &\mbox{ for } \bar H(2r+1) \mbox{ or }\overline{C{\hskip-0.1cm}H}(2r+1).
   \end{cases}
   $$
   Furthermore, $\bigwedge^n(\ggg_{-1})$ is a one-dimensional $\ggg_0$-module generated by $D_1\wedge \cdots \wedge D_n$,   of weight $-\sum_{i=1}^n\epsilon_i$ when $\ggg=X(n)$ for $X\in\{W,\bar S\}$, or of weight $-n\delta$ when  $\ggg=\bar H(n)$.
Furthermore, $\bigwedge\left(\ggg_{-1}\right)=\sum_{i=0}^n\bigwedge^i(\ggg_{-1})$ admits a weight set
\begin{align}\label{-1 weight set WS}
\wt\left(\bigwedge(\ggg_{-1})\right)
=\{-(\epsilon_{i_1}+\cdots+\epsilon_{i_k})\mid 1\leq i_1<\cdots<i_k\leq n, k=0,1,...,n\}
\end{align}
for $X(n)$, $X\in\{W,\bar S\}$.

From now on, we  set
\begin{align} \label{weight setting H}
\epsilon_{r+k}:=-\epsilon_k \;\; (k=1,...,r)
\mbox{ for }\ggg=\bar H(n) \mbox{ or } \ochn.
\end{align}
Let $\aleph$ be $0$ or $1$ in the following.
Then we can write
 \begin{align}\label{-1 weight set H}
&\wt\left(\bigwedge(\ggg_{-1})\right)\\
=&\begin{cases}\{-(\epsilon_{i_1}+\cdots+\epsilon_{i_k})-k\delta\mid 1\leq i_1<\cdots<i_k\leq n, k=0,1,...,n\}, &\mbox{ for }\bar H(2r);\\
\{-(\epsilon_{i_1}+\cdots+\epsilon_{i_k})-(k+\aleph)\delta\mid 1\leq i_1<\cdots<i_k\leq n-1, k=0,1,...,n\}, &\mbox{ for } \bar H(2r+1).
\end{cases}
\end{align}
The above is also true for $\och(n)$.

We always set $\ggg'_0=[\ggg_0,\ggg_0]$ throughout the paper. Then $\ggg'_0$  is a semisimple Lie algebra.

\begin{lem} \label{basic nat rep} Let $\ggg=W(n)$.  The following statements hold.
\begin{itemize}
\item[(1)] Set $M:=\sum_{i=1}^n\bbc m_i\in \ggg_1$  with  $m_i=\xi_i\sfd\in\ggg_1$ for $\sfd=\sum_{j=1}^n\xi_jD_j$.
    Then  both $\ggg_{-1}$ and $M$ are not only abelian subalgebras but also $\ggg_0$-modules. Especially  $U(\ggg_{-1})=\bigwedge \ggg_{-1}$ and $U(M)=\bigwedge M$. Here and after,  $\bigwedge L$ denotes  the exterior-product space of a vector space $L$.
\item[(2)] Under the identification between $\ggg_0$ and $\mathfrak{gl}(V)$ for $V=\sum_{i=1}^n\bbc\xi_i$, the $\ggg_0$-module $M$ is isomorphic to $V$ while $\ggg_{-1}$ is isomorphic to its linear dual $V^*$ (as $\ggg_0$-modules).
\item[(3)] Consider the following tensor products $$M^-(\lambda):=\bigwedge \ggg_{-1}\otimes_\bbc L^0(\lambda)$$
    and
    $$M^+(\mu):=\bigwedge M\otimes_\bbc L^0(\mu)$$
    in the category of $\ggg_0$-modules,  where $\lambda,\mu\in \Lambda^+.$
     If $L^0(\mu)$ is a composition factor of the $\ggg_0$-module $M^-(\lambda)$, then $L^0(\lambda)$ must be a composition factor of the $\ggg_0$-module $M^+(\mu)$.
\end{itemize}
\end{lem}
\begin{proof}  By a straightforward computation, the statements in (1) and (2) can be easily verified.

(3) Note that $\ggg_0\cong \mathfrak{gl}(V)$. 
The statement follows from  (2) and  the following isomorphism
\begin{align}\label{tensor formu}
\Hom_{\ggg_0}(L^0(\mu), \bigwedge V^*\otimes_\bbc L^0(\lambda))\cong \Hom_{\ggg_0}(L^0(\mu)\otimes_\bbc \bigwedge V, L^0(\lambda)).
\end{align}
\end{proof}

In the following, we will  generalize Lemma \ref{basic nat rep}(3) to the situation when $\ggg=\ochn$.
 Set $L:=\ggg_{\geq 1}\subseteq U(\ggg)$.
 Consider the $\ggg_{0}$-modules  $M^+(\mu):=L\otimes_{\bbc}L^0(\mu)$ and $M^-(\lambda):=\sum_{i=0}^n M^-(\lambda)_{-i}$ with $M^-(\lambda)_{-i}=\bigwedge^i (\ggg_{-1})\otimes_\bbc L^0(\lambda)$.
 The following lemma is somewhat a bridge to understand the block structure of $\comi$ for the case $\ochn$ (see Proposition \ref{I prop}).

\begin{lem}\label{basic nat rep ch}
Let $\ggg=\ochn$ and
 $L^0(\mu)$ be an irreducible composition factor of the $\ggg_0$-module $M^-(\lambda)_{i},$ $\lambda,\mu\in\Lambda^+.$ The following statements hold.

\begin{itemize}
\item[(1)]
If $i\leq -3$, then  $L^0(\lambda-2\delta)$ is a composition factor of the $\ggg_0$-module $M^+(\mu)$.
\item[(2)] If $i=-1$, then $L^0(\lambda-2\delta)$ is a composition factor of the $\ggg_0$-module $M^+(\mu-(n-2)\delta)$.
\item[(3)] If $i=-2$, then $L^0(\lambda-2\delta)$ is a composition factor of the $\ggg_0$-module $M^+(\mu-(n-4)\delta)$.

\end{itemize}
\end{lem}

\begin{proof} (1) Recall that $\ggg=\sum_{i=-1}^{n-2}\ggg_i$ with $\ggg_i=\wn_i\cap\ggg$ for $\ggg=\ochn$. We still set $V=\sum_{i=1}^n\bbc\xi_i$. Then we can identify $\ggg'_0$ with  $\mathfrak{so}(V)$, which admits a natural representation on $V$. The $\ggg'_0$-module $\ggg_{-1}=\sum_{i=1}^n\bbc D_i$ is  isomorphic to the contragredient $\ggg'_0$-module $V^*$ of $V$. Furthermore, for $i\in\{1,2,\cdots, n-2\}$, $\ggg_i$ is isomorphic to $\bigwedge^{i+2}(V)$ and admits eigenvalue $i$ for the action of $\sfd$. Actually, we can identify $\ggg_i$ with the space spanned by $D_H(\xi_{j_1}\cdots\xi_{j_{i+2}})$ for all $(j_1,...,j_{i+2})$ satisfying $1\leq j_1<\cdots<j_{i+2}\leq n$, the latter of which is isomorphic to $\bigwedge^{i+2} V$ as vector spaces. We can further say that $\ggg_i$ is isomorphic to $\bigwedge^{i+2}V$ as $\frak{so}(V)$-modules. This is ensured by  the definition of $D_H$ and the fact that for
the basis elements  $X=D_H(\xi_s\xi_t)\in \ggg'_0$ ($1\leq s<t\leq n$), 	
the following identity holds.
 $$\textsf{ad}X. D_H(\xi_{j_1}\cdots\xi_{j_{i+2}})=\sum_{k=1}^{i+2} D_H(\xi_{j_1}\cdots \xi_{j_{k-1}}\cdot X(\xi_{j_k})\cdot \xi_{j_{k+1}}\cdots \xi_{j_{i+2}}).$$

We continue to apply the isomorphism presented in (\ref{tensor formu}) for $\ggg'_0$-modules in the current case. For $i\in\{1,...,n-2\}$, we further have the following identity
$$\Hom_{\ggg'_0}(L^0(\mu), \bigwedge^{i+2} V^*\otimes_\bbc L^0(\lambda))\cong \Hom_{\ggg'_0}(L^0(\mu)\otimes_\bbc \ggg_i , L^0(\lambda)).
$$
Or to say, for $i\in \{1,..., n-2\},$
$$\Hom_{\ggg'_0}(L^0(\mu), M^-(\lambda)_{-(i+2)})\cong \Hom_{\ggg'_0}(L^0(\mu)\otimes_\bbc \ggg_{i} , L^0(\lambda)).
$$
Taking the eigenvalues of $\sfd$ into account, we get the first statement.

(2)
Recall that as $\ggg'_0$-modules,  $\bigwedge^{n-1} V^*\cong V^*$ and $\bigwedge^{n-2}V^*\cong \bigwedge^{2}V^*$.  Taking the eigenvalues of $\sfd$ into account, the second and the third statements follow from the first one.
\end{proof}

\begin{rem}\label{g0 regarding} The $\ggg_0$-module $M^+(\mu)$ in the arguments of Lemmas   \ref{basic nat rep} and \ref{basic nat rep ch} can be regarded as a $U(\ggg_0)$-submodule in $U(\ggg_{\geq 0})\otimes_{U(\ggg_0)}L^0(\mu)$. In general, for a $\ggg_0$-submodule $L$ of $U(\ggg_{\geq 0})$ by adjoint action, the tensor product module $M^+(\mu)=L\otimes_\bbc L^0(\mu)$  can  be regarded as $LU(\ggg_0)\otimes_{U(\ggg_0)}L^0(\mu)$, the latter of which is a $\ggg_0$-submodule of the induced module $U(\ggg_{\geq 0})\otimes_{U(\ggg_0)}L^0(\mu)$. Similarly, $M^-(\lambda)$ can be regarded as a $\ggg_0$-submodule of the induced module $U(\sfp)\otimes_{U(\ggg_0)}L^0(\lambda)$.
\end{rem}

\section{Projective covers}\label{proj and block}
Keep the notations as the previous sections.
\subsection{Projective covers in $\comi$}\label{section proj}
By Lemma \ref{basic lem O pre}, $\{L(\lambda)\}_{(\lambda,d)\in \bfe}$ form a complete set of pairwise non-isomorphic simple objects in $\comi$. By abuse of notations, we don't distinguish  $\bfe$ and  the set of iso-classes of irreducible modules in $\comi$ from now on. Especially, we make an appointment that  the simple object $L(\lambda)$ with depth $d$ will be written as
$ L(\lambda)=L(\lambda)_d.$
We first have the following basic observations. 

\begin{lem}\label{pre proj}
\begin{itemize}
 \item[(1)] Suppose that $M$ is  a $\bar \hhh$-semisimple and  locally finite $U(\ggg_0)$-module. Then $M$ is semisimple over $\ggg_0$. 
     \item[(2)] Suppose that $M$  is a finite-dimensional $U(\sfp)$-module generated by a maximal $\lambda$-weighted vector $v$.  Then $M$ admits a unique  irreducible quotient module,  which is isomorphic to  $L^0(\lambda)$ as a $\ggg_0$-module, endowed with trivial $\ggg_{-1}$-action.

         \item[(3)] Denote by $\cof$ the category of $\bar\hhh$-semisimple and locally finite $\ggg_0$-modules. If $V\in\cof$  is a highest weight module, i.e., generated by a maximal vector of weight $\lambda$, then $V\cong L^0(\lambda)$.

             \item[(4)] Any finite-dimensional irreducible $\ggg_0$-module $L^0(\lambda)$ for $\lambda\in \Lambda^+$ is  projective in $\cof$.
     \end{itemize}
\end{lem}
\begin{proof} (1) 
For any nonzero $v\in M$, $V:=U(\ggg_0)v$ is finite-dimensional. 
As $M$ is $\bar\hhh$-semisimple, so is $V$. We write $V=\sum_{\lambda\in \bar\hhh^*}V_\lambda$.
The finite-dimensionality of $V$ entails,  by some routine arguments, that $V$ can be decomposed into a direct sum of  irreducible $\ggg_0$-modules which are generated by  maximal (weighted-) vectors in $V=\sum V_\lambda$. Therefore, $M$ is semisimple over $\ggg_0$.

(2)  Recall for $\mu,\tau\in \bar\hhh^*$, $\mu\succeq \tau$ means that $\mu-\tau$ lies in $\bbz_{\geq 0}$-span of $\Phig\cup \Phi_0^+$. Clearly $M$ admits one-dimensional weight space $M'_\lambda$ of the highest weight $\lambda$. Furthermore, any proper submodule of $M$ admits weight spaces less than $\lambda$. Hence $M$ admits a unique maximal submodule, thereby $M$ as a $U(\sfp)$-module, has a quotient isomorphic to $L^0(\lambda)$, which can be viewed as an irreducible $\ggg_0$-module, endowed with a trivial $\ggg_{-1}$-action.

(3) This is a direct consequence of (1). Otherwise, $V=V_1\oplus V_2 \oplus\cdots\oplus V_s,s\geq2,$ and $V_i's$ are all finite-dimensional simple $\ggg_{0}$-modules. Then $V$ can not be generated by a single maximal vector of weight $\lambda.$

(4) It follows from the statements (1).
\end{proof}

The following result asserts the existence of  projective covers of simple modules in $\comi$.

\begin{thm} \label{projective thm}
Each simple object $L(\lambda)$ in $\comi$ has a projective cover $P(\lambda)$. Furthermore, $P(\lambda)$ admits a flag of standard modules, i.e., there is a sequence of submodules of $P(\lambda)$ $$P(\lambda)=P_0\supset P_1\supset \cdots\supset P_l\supset P_{l+1}=0$$
such that $P_{i}\slash P_{i+1}\cong \Delta(\lambda_i)$ for some $\lambda_i$, $i=0,1,\cdots, l$.
\end{thm}

\begin{proof}  Set $I(\lambda)=U(\ggg)\otimes_{U(\ggg_0)}L^0(\lambda)$. Then $I(\lambda)$ lies in $\comi$ (see Definition \ref{defn}). 
Our arguments are divided into different steps.

(i) We first claim that $I(\lambda)$ is a projective object  in $\comi$.

Indeed, thanks to Lemma \ref{pre proj}, $L^0(\lambda)$ is a projective $\ggg_0$-module in $\cof$. Note that the induction functor $U(\ggg)\otimes_{U(\ggg_0)}-$ is left adjoint to the restriction functor. The claim follows.

(ii) We next show that $I(\lambda)$ has a finite filtration such that each sub-quotient is isomorphic to a standard module.

Note that $I(\lambda)=U(\ggg)\otimes_{U(\sfp)}(U(\sfp)\otimes_{U(\ggg_0)}L^0(\lambda))$.
Now we consider the $U(\sfp)$-module $U(\sfp)\otimes_{U(\ggg_0)}L^0(\lambda).$
As a vector space, $U(\sfp)\otimes_{U(\ggg_0)}L^0(\lambda)\cong \bigwedge(\ggg_{-1})\otimes_\bbc L^0(\lambda).$ Denote
$\mathscr{L}^j(\lambda):=\bigwedge\nolimits^j \ggg_{-1}\otimes_\bbc L^0(\lambda)$
and $\mathscr{L}^{\geq j}(\lambda):=\bigoplus_{i=j}^{n}\mathscr{L}^i(\lambda),0\leq j \leq n.$ By a simple calculation, we can check that each $\mathscr{L}^{\geq j}(\lambda),0\leq j \leq n,$ is a $U(\sfp)$-submodule of $U(\sfp)\otimes_{U(\ggg_0)}L^0(\lambda)$. In particular, $\mathscr{L}^{\geq 0}(\lambda)=U(\sfp)\otimes_{U(\ggg_0)}L^0(\lambda).$ Then we have the following subsequence of $U(\sfp)$-modules.
\begin{align}\label{subquo1}
U(\sfp)\otimes_{U(\ggg_0)}L^0(\lambda)=\mathscr{L}^{\geq 0}(\lambda)\supseteq\mathscr{L}^{\geq 1}(\lambda)\supseteq\cdots\supseteq \mathscr{L}^{\geq (n-1)}(\lambda)\supseteq \mathscr{L}^{\geq n}(\lambda)\supseteq0,
\end{align}
which satisfies that $\mathscr{L}^{\geq i}(\lambda)/\mathscr{L}^{\geq (i+1)}(\lambda)\cong \mathscr{L}^{ i}(\lambda), 0\leq i \leq n-1.$ Here the subquotient $\mathscr{L}^{ i}(\lambda)$ has trivial $\ggg_{-1}$-action and is finite-dimensional.

 Since $\ggg_0$ is isomorphic to $\mathfrak{gl}(n)$ (resp.  $\mathfrak{sl}(n)+\bbc \sfd$ or $\mathfrak{so}(n)+\bbc\sfd$) for $\ggg$ being of type $W$ (resp. $\bar S$ or $\bar H$)
and $\sfd$ acts on $\mathscr{L}^{ i}(\lambda)$ as a scalar $\lambda(\sfd)-i$,  Weyl's completely reducible theorem is available to  $\mathscr{L}^j(\lambda)$, which means that  $\mathscr{L}^j(\lambda)$ can be certainly  decomposed into the following sum of irreducible $\ggg_0$-modules:
\begin{equation}\label{decom of i-th}
\mathscr{L}^j(\lambda)=\bigoplus\limits_{k=1}^{n_j} L^0\left(\eta_k^{(j)}\right),
\end{equation}
where  $\eta_k^{(j)}\in \Lambda^+$ satisfies $\Hom_{\ggg_0}\left(L^0\left(\eta_k^{(j)}\right),\bigwedge^j{\ggg_{-1}}\otimes L^0(\lambda)\right)\ne 0$.

So as a $U(\sfp)$-module, there is a filtration of $\mathscr{L}^{\geq j}(\lambda)$

\begin{align}\label{subquo2}
\mathscr{L}^{\geq j}(\lambda)=:\mathscr{L}^{\geq j}_1(\lambda)\supseteq \mathscr{L}^{\geq j}_2(\lambda)\supseteq \cdots \supseteq
\mathscr{L}^{\geq j}_{n_j}(\lambda)\supseteq \mathscr{L}^{\geq j+1}(\lambda)=\mathscr{L}^{\geq j+1}_1(\lambda)
\end{align}
such that $  \mathscr{L}^{\geq j}_{k}(\lambda) \slash\mathscr{L}^{\geq j}_{k+1}(\lambda)\cong L^0\left(\eta_k^{(j)}\right)$ and  $  \mathscr{L}^{\geq j}_{n_j} (\lambda)\slash\mathscr{L}^{\geq j+1}_{1}(\lambda)\cong L^0\left(\eta_{n_j}^{(j)}\right).$
From (\ref{subquo1}) and (\ref{subquo2}), we then get the following $U(\sfp)$-module filtration,
\begin{align}\label{subquo3}
\mathscr{L}^{\geq 0}(\lambda)\supseteq\mathscr{L}^{\geq 1}(\lambda)\supseteq\cdots&\supseteq  \mathscr{L}^{\geq j}_1(\lambda)\supseteq \mathscr{L}^{\geq j}_2(\lambda)\supseteq \cdots \supseteq\mathscr{L}^{\geq j}_{n_j}(\lambda)\supseteq\cr
&\supseteq
 \mathscr{L}^{\geq j+1}_1(\lambda)\supseteq \mathscr{L}^{\geq j+1}_2(\lambda)\supseteq \cdots \supseteq \mathscr{L}^{\geq j+1}_{n_{j+1}}(\lambda)\supseteq\cdots \cr
 \cdots&\supseteq \mathscr{L}^{\geq n}(\lambda)\supseteq0,
\end{align}
such that $  \mathscr{L}^{\geq j}_{k}(\lambda) \slash\mathscr{L}^{\geq j}_{k+1}(\lambda)\cong L^0\left(\eta_k^{(j)}\right)$ and  $  \mathscr{L}^{\geq j}_{n_j} (\lambda)\slash\mathscr{L}^{\geq j+1}_{1}(\lambda)\cong L^0\left(\eta_{n_j}^{(j)}\right)$ for $j=0,1,\cdots,n.$
Now set $I^{\geq j}_{k}(\lambda)=U(\ggg)\otimes_{U(P)}\mathscr{L}^{\geq j}_{k}(\lambda).$ Then we have the following $U(\ggg)$-module filtration,
\begin{equation}\label{subquo4}
I^{\geq 0}(\lambda)\supseteq I^{\geq 1}_1(\lambda)\supseteq\cdots\supseteq  I^{\geq j}_1(\lambda)\supseteq I^{\geq j}_2(\lambda)\supseteq \cdots \supseteq
I^{\geq j}_{n_j}(\lambda)\supseteq I^{\geq j+1}_1(\lambda) \supseteq\cdots\supseteq I^{\geq n}(\lambda)\supseteq0.
\end{equation}
By the construction, 
$I_k^{\geq j}(\lambda)\slash I_{k+1}^{\geq j}(\lambda)$ is isomorphic to $\Delta\left(\eta_k^{(j)}\right)$ for $1\leq k<n_j$, and $I_{n_j}^{\geq j}\slash I_{1}^{\geq j+1}$ is isomorphic to $\Delta\left(\eta_{_{n_j}}^{(j)}\right)$.

(iii) Thirdly, we prove that any direct summand of $I(\lambda)$ admits a $\Delta$-flag.

 By the construction in (ii), we have got that $I(\lambda)$ admits a $\Delta$-flag of finite length, in which the bottom one is a submodule $\Delta(\gamma)$ with $\gamma=\lambda-\sum_{i=1}^n\epsilon_i$ for $\ggg=\wn$ or $\bsn$, and $\gamma=\lambda-n\delta$ for $\ggg=\bhn$ or $\ochn$. This means that $\gamma\in \Lambda^+$ is the minimal one in $\wt(I(\lambda))\cap \Lambda^+$ (the set of the dominant and integral weights of $I(\lambda)$  is in the same sense as in the proof of the above lemma). 
 Actually, one can prove the general result that if $V\in \comi$ admits a $\Delta$-flag of finite length with the bottom standard module factor $\Delta(\gamma)$ satisfying that $\gamma$ is minimal in $\wt(V)\cap \Lambda^+$, then any direct summand of $V$ admits a $\Delta$-flag. This can be done by some standard inductive arguments on the lengths of $\Delta$-flags (see \cite[\S3.7]{Hum08}).

(iv) Fourthly, we prove that there exists an indecomposable projective module $J_0$ such that $J_0\rightarrow L(\lambda)$ is an epimorphism as $U(\ggg)$-modules.

From the arguments in (ii), we know that  as $U(\ggg)$-modules,
$$I(\lambda)\slash I_1^{\geq 1}(\lambda)\cong \Delta(\lambda).$$
So there are natural surjective morphisms
$I(\lambda){\overset {\pi_2}\longrightarrow} \Delta(\lambda){\overset {\pi_1}\longrightarrow} L(\lambda)$.
Denote $\pi:=\pi_1\circ \pi_2$. So we have
\begin{align}\label{irr proj of I}
\pi: I(\lambda)\twoheadrightarrow L(\lambda).
\end{align}

Assume $I(\lambda)=\bigoplus_{i=0}^{k}J_i$ (the finiteness of $k$ is ensured by (ii) and (iii)). Then there is a summand of $I(\lambda)$, written as $J_0$ without loss of generality,  such that $\pi|_{J_0}$ is non-zero. We denote $\pi|_{J_0}$ by $\pi_0.$
Because $ \Delta(\lambda){\overset {\pi_1}\longrightarrow} L(\lambda)$ is surjective,
the projective property of $J_0$ entails that $\pi_0$ can be lifted to a morphism $\bar\pi_0:J_0\longrightarrow \Delta(\lambda).$


(v) We claim that $J_0$ is the projective cover of both $\Delta(\lambda)$ and $L(\lambda)$.

 By the above argument, we already have the following commutative diagram:
\begin{equation*}
  \begin{array}{c}
\xymatrix{
  & J_0\ar[dl]_{\bar\pi_0}\ar[d]^{\pi_0} \\
  \Delta(\lambda) \ar[r]^{\pi_1} & L(\lambda).}
\end{array}
\end{equation*}
In fact, $\bar\pi_0$ is surjective. Otherwise, the image of $\bar\pi_0$ will be contained in the maximal submodule of $\Delta(\lambda),$ so $\pi_1\circ\bar\pi_0(J_0)=0\neq \pi_0(J_0),$  which contradicts to the above commutative diagram.

What remains is to prove that $\pi_0$ is essential. Consider $A:=\Hom_\comi(I(\lambda),I(\lambda))$. Then we have an isomorphism of vector spaces: $A\cong \Hom_{U(\ggg_0)}(L^0(\lambda), I(\lambda)|_{U(\ggg_0)})$.  Because $L^0(\lambda)$
is generated by $v_{\lambda}$ and $I(\lambda)_{\lambda}$ is finite-dimensional,	
 $\dim A<\infty$. Hence, as a subalgebra of $A$, $A_0:=\Hom_\comi(J_0,J_0)$ is finite-dimensional. Then, by some standard arguments on Fitting decomposition
we can prove that $\pi_0$ is indeed essential.

We can further have  that $J_0$ is also the projective cover of $\Delta(\lambda)$. This is because  the essential property of $\bar \pi_0$ can be ensured by that of $\pi_0$.

As $I(\lambda)$ admits a unique factor $\Delta(\lambda)$ in its $\Delta$-flag, it is easy to deduce that the choice of $J_0$ is unique among  all indecomposable direct summands of $I(\lambda)$.
\end{proof}

\begin{rem}\label{proj remark}
		
{\rm(1)}  We can precisely construct such a $P(\lambda)$ (=$J_0$) as below.
 From $\pi_0=\pi|_{J_0}$ and the definition of the category $\comi$,  it follows that $J_0$ contains a vector $v_0$ of the form like
 $$1\otimes v^0_\lambda+\sum_{i}u_i\otimes v_i$$
  where $u_i\in U(\ggg_{\geq 1})\ggg_{\geq 1}U(\ggg_{-1})\ggg_{-1}$ with the weight of all $u_i\otimes v_i$ being $\lambda$. Set
  $$\tilde J_0:=U(\ggg)v_0.$$
   By the arguments as above, we actually have the following commutative diagram
    \begin{equation*}
  \begin{array}{c}
\xymatrix{
  & \tilde J_0\ar[dl]_{\bar\pi_0|_{\tilde J_0}}\ar[d]^{\pi_0|_{\tilde J_0}} \\
  \Delta(\lambda) \ar[r]^{\pi_1} & L(\lambda).}
\end{array}
\end{equation*}
   The essential property of $\pi_0$ entails that $J_0=\tilde J_0$.

{\rm(2)} From the proof (v) of Theorem \ref{projective thm}, we know that $J_0$ admits a unique maximal submodule,  which is exactly $\ker(\pi_0)$. So an irreducible module in $\comi$ is naturally the unique irreducible quotient of its projective cover.

{\rm(3)} Let $\lambda\in \Lambda^+$.  Set
   \begin{align}\label{upsilon}
   \Upsilon(\lambda):=\{\mu\in \Lambda^+\mid &\;(\bigwedge{\ggg_{-1}}\otimes L^0(\lambda): L^0(\mu))_{\ggg_0}\ne 0 \},
   \end{align}
  where $(L:L^0(\mu))_{\ggg_0}$ denotes the multiplicity of $L^0(\mu)$ in the composition series of the finite-dimensional $\ggg_0$-module $L$.
   As in the proof (ii) of Theorem \ref{projective thm}, we have the following decomposition of as $\ggg'_0$-modules:
$$\bigwedge\ggg_{-1}\otimes_\bbc L^0(\lambda)=\bigoplus\limits_{\mu\in\Upsilon(\lambda)} n_{\lambda,\mu}L^0(\mu).$$
Moreover, the following statements hold.
\begin{itemize}
\item[($1^\circ$)] $n_{\lambda, \mu}=0$ for any $\mu\notin\Upsilon(\lambda)$.

\item[($2^\circ$)] For a projective object $Q\in \comi$, denote by $[Q:\Delta(\mu)]$ the multiplicity of $\Delta(\mu)$ in its $\Delta$-flag. Then
$[I(\lambda):\Delta(\mu)]=n_{\lambda, \mu}$ for any $\mu\in\Lambda^+$. In particular, $[I(\lambda):\Delta(\lambda)]=1$.
\item[($3^\circ$)] Suppose $\lambda-\sum_{i=k}^n\epsilon_i\in \Lambda^+$. Then $[I(\lambda):\Delta(\lambda-\sum_{i=k}^n\epsilon_i)]\ne 0,1\leq k\leq n,$ for $\ggg=W(n),\bar S(n)$, In particular, $[I(\lambda):\Delta(\lambda-\sum_{i=1}^n\epsilon_i)]\ne 0$ for $\ggg=W(n),\bar S(n)$. 
\item[($4^\circ$)] 
    %

$[I(\lambda):\Delta(\lambda+\sum_{i=1}^{k}\epsilon_i-(n-k)\delta)]\ne 0$ ($k=0,1,...,r$) for $\ggg=\bar H(n),\ochn$. In particular, $[I(\lambda):\Delta(\lambda-n\delta)]\ne 0$ for $\ggg=\bar H(n),\ochn$.
\end{itemize}
These statements  {\rm($1^\circ$)-($4^\circ$)} will be used in the sequel.
The statements  {\rm($1^\circ$)-($2^\circ$)} are direct consequences of the theorem.
For  {\rm($3^\circ$)}, we remind that  $D_{k}\wedge D_{k+1}\wedge \cdots\wedge D_n\otimes v_{\lambda}^0$ with $1\leq k \leq n$ is a maximal weight vector of $\bigwedge^{n-k+1}\ggg_{-1}\otimes_\bbc L^0(\lambda)$ for
 $\ggg=W(n)$, $\bar S(n)$.
As for  {\rm($4^\circ$)}, we can check that $D_{k+1}\wedge \cdots\wedge D_r\wedge D_{r+1}\wedge\cdots\wedge D_n\otimes v^0_{\lambda}$ is a  $\ggg_{0}$-maximal weight vector.
Now the results in   {\rm($3^\circ$)} and  {\rm($4^\circ$)} hold
by (\ref{-1 weight set WS}), (\ref{-1 weight set H}) and the formula (\ref{decom of i-th}) in the proof of Theorem \ref{projective thm}.
\end{rem}

\subsection{The category $\comif$}\label{delete}
Denote by $\comif$ the full subcategory of $\comi$ whose objects are finitely-generated $U(\ggg)$-modules in $\comi$.
Then we have the following consequence based on Theorem \ref{projective thm}.

\begin{thm}\label{projective thm 2}  The category $\comif$ has enough projective objects, this is to say, for each $M\in\comif$, there is a projective object $P\in\comif$ and an epimorphism $P\twoheadrightarrow M$.    \end{thm}

\begin{proof}  Note that $P(\lambda)$, $\Delta(\lambda)$, $\nabla(\lambda)$ and $L(\lambda)$ are all in $\comif$. And it is still true that $P(\lambda)$ is a projective cover of $L(\lambda)$ in $\comif$.
For any nonzero object $M\in \comif$, $M$ admits a filtration of finite length
\begin{align}\label{standard fil 0}
M=M^0\supset M^1\supset M^2 \supset\cdots\supset M^{t-1}\supset M^{t}=0
\end{align}
 such that $M^{i-1}\slash M^{i}$ is isomorphic to a non-zero quotient of  $\Delta(\lambda_i)$ for some $L(\lambda_i)\in \bfe$, $i=1,\cdots, t$. The least number $t$ in  all possible filtrations as in (\ref{standard fil 0}) is called the standard length of $M$, denoted by $l(M)$.

Set  $P=\bigoplus_{i=1}^t P(\lambda_i)$. Then by induction on $t$, there will be a covering morphism from $P$  onto $M$.  The proof is completed.
\end{proof}

\begin{prop} \label{ind proj prop} In $\comif$,  every indecomposable projective module is isomorphic to the projective cover $P(\lambda)$ of some irreducible module $L(\lambda)$.
\end{prop}

\begin{proof}
		
Suppose $P$ is an indecomposable projective module in $\comif$. By the definition of $\comif$,  $P$ has an irreducible quotient $L(\lambda)$, which defines an epimorphism $\phi:P\rightarrow L(\lambda)$. The projective property of $P$ yields the following commutative diagram
	\begin{equation*}
	\begin{array}{c}
	\xymatrix{
		& P\ar[dl]_{\bar\phi}\ar[d]^{\phi} \\
		P(\lambda) \ar[r]^{\pi_0} & L(\lambda).}
	\end{array}
	\end{equation*}
	Because $\pi_0$ is essential, $\bar \phi$ must be surjective. 
	Hence $P(\lambda)$ is isomorphic to a direct summand of $P$.
	The assumption of indecomposability of $P$ entails that $P$ is isomorphic to $P(\lambda)$.		
\end{proof}

\section{Degenerate BGG reciprocity and typical functor}\label{degen}

Maintain the previous notations and assumptions.
\subsection{} Thanks to Lemma \ref{pre proj}, Brundan's arguments in \cite{Brun04} {\color{red}are} available to $\comi$.

\begin{thm}\label{basic proj citedthm} (\cite[Theorem 4.4]{Brun04} and \cite[Theorem 3.2]{Soe98}) Every simple object $L(\lambda)=L(\lambda)_d$ in $\mathcal{O}^{\text{min}}_{\geq d'}$ admits a projective cover $P_{\geq d'}(\lambda)$ with $d\geq d'$, the head of $P_{\geq d'}(\lambda)$ is isomorphic to $L(\lambda)=L(\lambda)_{d}$. Moreover,
\begin{itemize}
\item[(1)] $P_{\geq d'}(\lambda)$ admits a finite $\Delta$-flag with $\Delta(\lambda)$  at the top.

\item[(2)] For $m < l$, the kernel of any surjection $P_{\geq m}(\lambda)
\rightarrow P_{\geq l}(\lambda)$  admits a finite $\Delta$-flag with subquotients of the form $\Delta(\mu)$ for $m\leq \lfloor L^0(\mu)\rfloor< l$.

\item[(3)]  $L(\lambda)$ admits a projective cover in $\comi$  if and only if there exists $l\ll 0$  with $P_{\geq l}(\lambda)=P_{\geq l-1}(\lambda)=P_{\geq l-2}(\lambda)=\cdots$, in which case $P(\lambda)=P_{\geq l}(\lambda)$.
\end{itemize}
\end{thm}

 In our case we have a stronger result (Theorem \ref{projective thm}). This is to say, the projective covers of $L(\lambda)$ in $\comid$ and $\comis$ exist. But the above theorem can help us to give more information on $P(\lambda)$ in the next subsection.

\subsection{} By Theorem \ref{basic proj citedthm}, every simple object $L(\lambda)=L(\lambda)_{d'}$ in $\comid$ admits a projective cover $P_{\geq d}(\lambda)$ with $d'\geq d$, the head of $P_{\geq d}(\lambda)$ is isomorphic to $L(\lambda)=L(\lambda)_{d'}$. Theorem \ref{basic proj citedthm} along with Theorem \ref{projective thm} implies that there exists $l\ll 0$  with $P_{\geq l}(\lambda)=P_{\geq l-1}(\lambda)=P_{\geq l-2}(\lambda)=\cdots$, and  $P(\lambda)=P_{\geq l}(\lambda)$.
Furthermore, by Theorem \ref{basic proj citedthm}, any $P_{\geq l}(\lambda)$ admits a $\Delta$-flag. Denote by  $[P_{\geq l}(\lambda):\Delta(\mu)]$ the multiplicity of $\Delta(\mu)$ in the $\Delta$-flag of $P_{\geq l}(\lambda)$.
By \cite[\S4]{Soe98} or \cite[Lemma 4.5]{Brun04}, we have the following result.

\begin{lem}\label{Soergel lem}
$[P_{\geq l}(\lambda):\Delta(\mu)]=(\nabla(\mu):L(\lambda))$ for all $L(\lambda)$ and $L(\mu)\in \bfe$ as long as $\dpt(L(\lambda))\geq l$ and $\dpt(L(\mu))\geq l$.
\end{lem}

\subsection{Degenerate BGG reciprocity}
\begin{thm}\label{deg bggthm} In the category $\comis$, the following statement holds
$$ [P(\lambda):\Delta(\mu)]=(\nabla(\mu):L(\lambda))$$
for all $L(\lambda),L(\mu)\in \bfe$.
\end{thm}

\begin{proof} For any given $L(\lambda)\in \bfe$, assume $\dpt(L(\lambda))=d$. By Theorem \ref{basic proj citedthm}(3), there exists some $l\ll 0$ such that $P(\lambda)$ in $\comi$ and $P(\lambda)=P_{\geq l}(\lambda)=P_{\geq l-i}(\lambda)$ for all $i\in \bbz_{\geq 0}$ (certainly, $l<d$). For any $L(\mu)\in \bfe$, there exists some $i_0\in \bbz_{\geq 0}$ such that $\dpt(L(\mu))\geq l-i_0$.
Since $l<d$ and  $i_0\in \bbz_{\geq 0}$,  we have $P(\lambda)=P_{\geq l-i_0}(\lambda)$ and
$\dpt(L(\lambda))\geq l-i_0.$
Now applying Lemma \ref{Soergel lem} to $P_{\geq l-i_0}(\lambda)$, we have $ [P(\lambda):\Delta(\mu)]=(\nabla(\mu):L(\lambda))$.
\end{proof}

\subsection{The Kac-module realization of co-standard modules}\label{kac}
Set  $\ggg^+:=\oplus_{i\geq0}\ggg_{i}$. The following module is the so-called Kac-module
$$K(\lambda)=U(\ggg)\otimes_{U(\ggg^+)}L^0(\lambda),$$
where $L^0(\lambda)$ is regarded as a $\ggg^+$-module with trivial $\ggg_{\geq1}$-action. One can check that $K(\lambda)$ has a simple head, denoted by $\overline{L}(\lambda)$.

Following \cite{Ser05}, we introduce the set $\Omega$ of the so-called Serganova atypical weights as follows.

If $\ggg=W(n),$ set
$$\Omega=\{a\epsilon_i+\epsilon_{i+1} +\cdots+\epsilon_n\mid a\in\mathbb{C},1\leq i\leq n \}.$$

If $\ggg=\bar{S}(n),$ set
$$\Omega=\{a\epsilon_1+\cdots+a\epsilon_{i-1}+b\epsilon_{i}+ (a+1)\epsilon_{i+1} +\cdots+(a+1)\epsilon_n\mid a,b\in\mathbb{C},1\leq i\leq n \}.$$

If $\ggg=\bar{H}(n),$ set
$$\Omega=\{-\epsilon_1-\cdots-\epsilon_{i-1}+b\epsilon_{i}+a\delta  \mid a,b\in\mathbb{C},1\leq i\leq r \}.$$
\begin{defn}
	A weight $\lambda$ is called Serganova atypical if $\lambda$ belongs to $\Omega.$ Otherwise, $\lambda$ is called Serganova typical.
\end{defn}

 Keep it in mind that the notation  $\Xi=\epsilon_{1}+\epsilon_{2}+\cdots+\epsilon_{n}$ for $\ggg=W(n)$ or $\bar{S}(n).$

\begin{prop}\label{iso2}
	Let $\ggg=X(n), X\in\{W,\bar{S},\bar{H}\}$.
	\begin{itemize}
		\item[(1)] If $\ggg=W(n),\bar{S}(n)$, then $\nabla(\lambda)\cong K(\lambda+ \Xi)$.
		\item[(2)] If $\ggg=\bar{H}(n)$, then $\nabla(\lambda)\cong K(\lambda+ n\delta)$.
		
		\item[(3)] The Kac-module $K(\lambda)$ is irreducible if and only if $\lambda$ is Serganova typical.
		
	\end{itemize}
\end{prop}

\begin{proof} The third statement follows from  \cite[Theorem 6.3]{Ser05}. We proceed to prove the first two statements.
	
	Note that $\ggg^+$ is a subalgebra of $\ggg$ with codimension
	$n$ and $\ggg_{\bar{0}}\subseteq\ggg^+.$
	Let $f:\ggg^+\rightarrow\mathfrak{gl}(\ggg/\ggg^+)=\mathfrak{gl}(\ggg_{-1})$ be the map defined by $f(a)(b+\ggg^+)=[a,b]+\ggg^+.$ Then it follows from \cite[Theorem 2.2]{BF93} that $U(\ggg):U(\ggg^+)$ is a free $\theta$-Frobenius
	extension, where $\theta$ is the unique automorphism of $U(\ggg^+)$ defined by
	\[
	\theta(a)=\left\{
	\begin{array}{ll}
	a+\mu(a)\cdot 1,& \mbox{if}~a~\in\ggg_{\bar{0}}^+,\\
	(-1)^na,& \mbox{if}~a~\in\ggg_{\bar{1}}^+,
	\end{array}
	\right.
	\]
	and $\mu:\ggg^+\rightarrow\mathbb{C}$ is defined by $\mu(a)=\mathrm{tr}f(a).$
	Thus by \cite[\S3]{NT60}, we have
	
	\begin{equation}\label{formu}
	\mathrm{Hom}_{\ggg^+}\left(U(\ggg),~_{\theta}L^0(\lambda)\right)\cong U(\ggg)\otimes_{U(\ggg^+)}L^0(\lambda)=K(\lambda),
	\end{equation}
	where $_{\theta}L^0(\lambda)$ is a $\ggg^+$-module with action twisted by $\theta$, i.e.,  $s*v:=\theta(s)v$ for any $s\in\ggg^+$ and $v\in L^0(\lambda)$.
	
	Now let $v_{\lambda}^0$ be a maximal vector of $L^0(\lambda)$ corresponding to the standard Borel subalgebra $\bbb^+_0.$  Since
	\[
	\theta(x)=\left\{
	\begin{array}{ll}
	x,& \mbox{if}~x\in\ggg_{\geq 1}\cap\ggg_{\bar{0}},\\
	(-1)^nx,& \mbox{if}~x\in\ggg_{\geq 1}\cap\ggg_{\bar{1}},
	\end{array}
	\right.
	\]
	$_{\theta}L^0(\lambda)$ is still an irreducible $\ggg^+$-module with trivial $\ggg_{\geq 1}$-action. Because $\mu(x)=0$ for $x\in\nnn^+$, $v_{\lambda}^0$ is still a maximal vector of $_{\theta}L^0(\lambda)$. Let $h\in \bar\hhh$.
	
	Case (i): $\ggg=W(n)$ or $\bar{S}(n)$.
	
	Since $f(h)(D_i+\ggg^+)=[h,D_i]+\ggg^+=-\epsilon_i(h)(D_i)+\ggg^+$, it follows that $\mu(h)=\mathrm{tr}f(h)=-\epsilon_1(h)-\epsilon_2(h)-\cdots-\epsilon_n(h)=-\Xi(h).$ Consequently,
	$$h*v_{\lambda}=\theta(h)v_{\lambda}=(h+\mu(h).1)v_{\lambda}=(\lambda- \Xi)(h)v_{\lambda}.$$
	Hence by (\ref{formu}) we get that
	$$\nabla(\lambda- \Xi)\cong K(\lambda)\mbox{~for~any~}\lambda\in\Lambda^+.$$
	Equivalently,
	\begin{equation}
	\nabla(\lambda)\cong K(\lambda+ \Xi)\mbox{~for~any~}\lambda\in\Lambda^+.
	\end{equation}
	
	Case (ii): $\ggg=\bar H(n)$.
	
	Subcase (ii-1): $n=2r$.
	
	In this subcase,
	$$f(h)(D_i+\ggg^+)=(-\epsilon_i-\delta)(h)(D_i+\ggg^+)\mbox{~for~}1\leq i\leq r,$$
	$$f(h)(D_i+\ggg^+)=(\epsilon_{i'}-\delta)(h)(D_i+\ggg^+)\mbox{~for~}r+1\leq i\leq 2r.$$
	
	Subcase (ii-2): $n=2r+1$.
	
	In this subcase,
	$$f(h)(D_i+\ggg^+)=({-\epsilon_i-\delta})(h)(D_i+\ggg^+)\mbox{~for~}1\leq i\leq r,$$
	$$f(h)(D_i+\ggg^+)=({\epsilon_{i'}-\delta})(h)(D_i+\ggg^+)\mbox{~for~}r+1\leq i\leq 2r,$$
	$$f(h)(D_n+\ggg^+)=-\delta(h)(D_n+\ggg^+).$$
	
	It follows that $_{\theta}L^0(\lambda)\cong L^0(\lambda-n\delta)$. Hence, by (\ref{formu}), we get
	$$\nabla(\lambda-n\delta)\cong K(\lambda),\,\,\forall\,\lambda\in\Lambda^+.$$
	Equivalently,
	$$\nabla(\lambda)\cong K(\lambda+n\delta),\,\,\forall\,\lambda\in\Lambda^+.$$	
\end{proof}

\begin{thm}\label{iso-2}
	Let $\lambda,\mu\in \bfe$. Then the following statements hold.
	\begin{itemize}
		\item[(1)] If $\ggg=W(n)$ or $\ggg=\bar{S}(n)$, then
		$$ [P(\lambda):\Delta(\mu)]=(K(\mu+\Xi):L(\lambda)).$$
		\item[(2)] If $\ggg=\bar{H}(n)$, then
		$$ [P(\lambda):\Delta(\mu)]=(K(\mu+n\delta):L(\lambda)).$$
	\end{itemize}
\end{thm}
\pf Theorem \ref{deg bggthm} and Proposition \ref{iso2} can be applied to get these results.\qed

\subsection{Typical blocks and the typical  functor}\label{typical fun}
 We begin this subsection with the
 following  consequence of indecomposable projective modules in $\comi$, which is well known  for Noetherian categories.

\begin{lem}\label{simult proj} Suppose $M\in\comi$. Then the following statements hold.
\begin{itemize}
\item[(1)] For any $L(\lambda)\in \bfe$,
 $$(M: L(\lambda))=\dim\Hom_{\comi}(P(\lambda),M).$$

 \item[(2)] If there exists a nonzero vector $v\in M$ of weight $\lambda$, which is annihilated by $\ggg_{-1}+\nnn^+$, then $(M:L(\lambda))\ne 0$.
 \end{itemize}
 \end{lem}

\begin{proof}
  (1) Suppose $\dpt( L(\lambda))= t$.
   If $(M: L(\lambda))\ne 0$, then the multiplicity is less than the dimension of $M_t$. By the definition of $\comi$, $\dim M_t< \infty$. Thus, it is a routine way to prove the lemma by induction on $(M:L(\lambda))<\infty$.

   (2) Consider the submodule $N$ generated by $v$ in $M$. Note that by assumption the module $U(\ggg_0)v$ is a finite-dimensional highest weight module over $U(\ggg_0)$,  generated by the maximal vector $v$. Therefore $U(\ggg_0)v$ is isomorphic to $L^0(\lambda)$. Furthermore,  the assumption of $\ggg_{-1}$-annihilation of $v$ implies that  as a $U(\ggg)$-module, $N$ is a homomorphism image of $\Delta(\lambda)$. Hence $(M:L(\lambda))\ne 0$.
   \end{proof}

In general, define $A_\lambda:=\Hom_\comi(P(\lambda),P(\lambda))$. Then $A_\lambda$ is a finite-dimensional $\mathbb{C}$-algebra, whose dimension is exactly $(P(\lambda): L(\lambda))$ by Lemma \ref{simult proj}(1).
Let ${\mathcal{O}}^{\text{min}}_\lambda$ stand for the block in which $L(\lambda)$ lies.
In general, we can define a functor
 $$\sfS_\lambda:=\Hom_\comi(P(\lambda),-).$$
  By Lemma \ref{simult proj}(1) again, $\sfS_\lambda$ gives rise to a functor from ${\mathcal{O}}^{\text{min}}_\lambda$  to the category  of finite-dimensional $A_\lambda$-modules, the latter of which is denoted by $\almf$.


Denote by $\Lambda_{st}$ the set of all Serganova typical weights. All dominant Serganova  typical weights  can be clearly described. For example, if $\ggg=W(n)$, then
$\Lambda^+_{st}=\{\lambda=\sum_{i=1}^n a_i\epsilon_i\mid a_i-a_{i+1}\in \bbz_{\geq 0}\}\backslash \Omega^+$ with $\Omega^+:=\Omega\cap \Lambda^+$.

Set $\Lambda_\sft:=\{\lambda-\Xi\mid \lambda\in \Lambda_{st}\}$ if $\ggg=W(n),\bar{S}(n)$, and  $\Lambda_\sft:=\{\lambda-n\delta\mid \lambda\in \Lambda_{st}\}$ if $\ggg=\bar{H}(n)$. All weights lying in $\Lambda^+_\sft:=\Lambda^+\bigcap\Lambda_\sft $ are called typical. According to Theorem \ref{deg bggthm} and Proposition \ref{iso2} (or Theorem \ref{char form P}), we have that for $\lambda\in \Lambda^+_\sft$,
\begin{align}\label{trip equa}
P(\lambda)=\Delta(\lambda).
\end{align}
So when $\lambda$ is typical, $A_\lambda=\End_\comi(\Delta(\lambda))$, which is one-dimensional. The functor $\sfS_\lambda$ is degenerated.

\begin{prop} Let $\lambda$ be a typical weight and $M$ be an object of $\comitl$. Then the functor $\sfS_\lambda$ measures the multiplicity  of $L(\lambda)$ in $M$. This is to say, if
	$(M:L(\lambda))=m,$ then $\sfS_\lambda(M)=\mathbb{C}^m$, the unique   $m$-dimensional $A_\lambda$-module up to isomorphisms.
\end{prop}
\begin{proof}  Note that $A_\lambda$ is a one-dimensional algebra over $\mathbb{C}$, which is isomorphic to $\mathbb{C}$. The isomorphism class of an object in $\almf$ is only dependent on the dimension. So the statement is a direct consequence of Lemma \ref{simult proj}(1).
\end{proof}

\section{Blocks  of $\comi$}\label{block sec}

\subsection{Definition}\label{def of block} Due to  Theorem \ref{projective thm}, we  define an equivalent relation $\sim$ in $\mathbf{E}$.
For any simple objects $L(\lambda_1), L(\lambda_2)$ in $\mathbf{E}$, we say that $L(\lambda_1)$ and $L(\lambda_2)$ are linked (or $\lambda_1$ and $\lambda_2$ are linked) if there exists $L(\mu)\in \bfe$ such that $(P(\mu):L(\lambda_i))\ne 0$ for $i=1,2$. We say that $L(\lambda)\sim L(\mu)$ (or $\lambda\sim\mu$)  if there  exist  a sequence $L(\lambda)=L(\lambda_1)$, $L(\lambda_2)$,......,$L(\lambda_k)=L(\mu)$ in $\mathbf{E}$ such that
$L(\lambda_i)$ and $L(\lambda_{i+1})$ are linked (or $\lambda_i$ and $\lambda_{i+1}$ are linked) for every $i=1,...,k-1$.

For a given element $\theta\in \bfe\slash{\hskip-2pt\sim}$, we define a full subcategory $\comit$ of $\comi$ whose objects are those modules $M$ only admitting composition factors from $\theta$. We call $\comit$ a block corresponding to $\theta$.

\begin{lem}\label{indecom blo 1} Any indecomposable object in $\comif$ must belong to a certain $\comit$.
\end{lem}

\pf Suppose that $M$ is a nonzero indecomposable module belonging to $\comif$ .
As in the proof of Theorem \ref{projective thm 2}, there is a projective module
	$P:=\oplus_{i=1}^t P(\lambda_i)$ and an epimorphism $\pi: P\twoheadrightarrow M.$ So
	\begin{align}\label{covering morp}
	M=\pi(P(\lambda_1))+\pi(P(\lambda_2))+\cdots+\pi(P(\lambda_t)).
	\end{align}
This ensures that we can define a non-zero submodule $M_\theta$ of $M$, which is a sum of submodules belonging to $\comit$.

If $M_\theta$ coincides with $M$, then we are done.	
	Otherwise, we have a non-zero submodule $M_\theta'$ of $M$, which is the sum of all submodules belonging to the blocks outside $\comit$.	
	Then $M=M_\theta + M_\theta'$ by (\ref{covering morp}). Furthermore, $M_\theta+ M_\theta'$ is a direct sum through the definition of blocks. This contradicts to the indecomposability of $M$. The proof is completed.\qed

Recall that all standard modules $\Delta(\lambda)$ and costandard  modules $\nabla(\lambda)$ are indecomposable and finitely generated. In addition, we have the following stronger results.
\begin{lem}\label{indecom blo}
	Let $(\lambda,d)\in \mathbf{E}= \Lambda^+\times \bbz$. Then
$\Delta(\lambda)$ and $\nabla(\lambda)$ are in the same block.
\end{lem}
\begin{proof} One can give a  proof following \cite[Lemma 3.5]{Brun04}. Here we give another one.  By the arguments  in the proof of Proposition \ref{iso2},  as a vector space $\nabla(\lambda)$ can be identified with $\bigwedge \ggg_{-1}\otimes {_{-\theta}}L^0(\lambda)$. Take  a maximal vector $v_0$ of $L^0(\lambda)$, and set $v=\bigwedge_{i=1}^nD_i\otimes v_0$. By definition, $v$ has weight $\lambda$. Furthermore, $v$ is annihilated by $\ggg_{-1}+\nnn^+$. Hence by Lemma \ref{simult proj}(2), $\nabla(\lambda)$ shares the same composition factor $L(\lambda)$ with $\Delta(\lambda)$.  So this lemma is a direct consequence of Lemmas \ref{sts-costa indec} and \ref{indecom blo 1}.
\end{proof}

\begin{rem} In \cite{Brun04}, the definition of blocks was introduced via  standard modules  and  co-standard modules because of the loss of projective covers of simple objects. Lemma \ref{indecom blo} shows that our definition of blocks is compatible with the one introduced therein.
\end{rem}

\subsection{} In the following, we discuss some block properties through investigating standard modules. Recall that $\ggg$ admits a $\bbz$-gradation which gives rise to the $\bbz$-gradation $U(\ggg)=\sum_{i\in \bbz}U(\ggg)_i$.  Similarly, we can talk about the gradation of $U(\ggg_{\geq 1})=\sum_{i\geq 0}U(\ggg_{\geq 1})_i$.

Consider $\Delta(\mu)=\sum_{i\geq 0}\Delta(\mu)_i$ for $\Delta(\mu)_i=U(\ggg_{\geq 1})_i\otimes L^0(\lambda)$. Set $\Delta(\mu)^{(j)}=\sum_{i\geq j}\Delta(\mu)_i$ for $j\in \bbn$.
Then as a $\ggg_{\geq 0}$-module,  $\Delta(\mu)$ has the natural descending  filtration  $\{\Delta(\mu)^{(j)}\}_{j\in \bbn}$.

\begin{lem}\label{comp}
	Let $\lambda,\mu\in\Lambda^+$. 	If $v_{\lambda}$ is a nonzero $\lambda$-weighted vector of $\Delta(\mu)$ annihilated by $\nnn^+$, then $\lambda\sim \mu$.
\end{lem}

\begin{proof}
 By definition, $\Delta(\mu)=U(\ggg_{\geq 1})\otimes L^0(\mu)$ as a vector space. For any $\nnn^+$-annihilating vector $v_{\lambda}\in \Delta(\mu)$ of weight $\lambda$,
if $v_\lambda$ lies in $1\otimes L^0(\mu)$, then $\lambda$ coincides with $\mu$, and the statement of the lemma is obvious.  In the following,  we suppose $v_\lambda\in \Delta(\mu)^{j}\backslash\Delta(\mu)^{j-1}$ for some $j>0$. Still set $\ggg^+=\ggg_{\geq0}$. Consider the  $U(\ggg^+)$-submodule generated by $v_\lambda$ in $\Delta(\mu)$, denoted by $\calm$. Clearly, $\calm$ has a proper submodule $\caln:=U(\ggg_{\geq1})\ggg_{\geq1}U(\ggg_0)v_\lambda$.  So we have a $U(\ggg^+)$-module $\overline{\calm}:=\calm\slash \caln$. This $\overline{\calm}$ is generated by the image of $v_\lambda$ in $\overline{\calm}$, denoted by $\bar v_\lambda$ which has weight $\lambda$ and is annihilated by $\nnn^+\oplus\ggg_{\geq 1}$. So we have surjective morphisms
$$\calm{\twoheadrightarrow} \overline\calm{\twoheadrightarrow} L^0(\lambda),$$
where $L^0(\lambda)$ is an irreducible $\ggg^+$-module with highest weigh $\lambda$ and  trivial $\ggg_{\geq1}$-action. Consider the functor $\Gamma=\Hom_{\ggg_{\geq 0}}(U(\ggg),-)$ from the category of $U(\ggg^+)$-modules to the one of $U(\ggg)$-modules. Then $\Gamma(L^0(\lambda))=\nabla(\lambda)$.

In the following we focus on the subcategory $\calc_{\ggg^+}$ of $U(\ggg^+)$-module category which consists of  objects $C$ satisfying: (i) it has $\bbz$-gradation, and finitely generated over $U(\ggg^+)$,  (ii) $C$ is locally finite over $\ggg_0$, i.e. for any $v\in C$ the  $U(\ggg_0)$-submodule generated by $v$ is finite-dimensional. Then all irreducible objects in $\calc_{\ggg^+}$ are finite-dimensional, and the isomorphism classes of irreducible objects in $\calc_{\ggg^+}$ coincide with $\{L^0(\lambda)\mid \lambda\in \Lambda^+\}$ (see the forthcoming Lemma \ref{lem: calc irr}).
The functor $\Gamma$ is regarded as a functor from $\calc_{\ggg^+}$ to the $\bbz$-graded $U(\ggg)$-module category. Furthermore, by the same arguments as in the proof of Proposition \ref{iso2},  $\Gamma(M)$ for any $M\in \calc_{\ggg^+}$, can be identified with $\bigwedge \ggg_{-1}\otimes {_{-\theta}}M$
where the meaning of ${_{-\theta}}M$ are the same as in the paragraph around (\ref{formu}).

Note that $\Delta(\mu)$ belongs to $\calc_{\ggg^+}$, and is still  an indecomposable $U(\ggg^+)$-module.
The irreducible $U(\ggg^+)$-module $L^0(\lambda)$ is already known as a composition factor of $\Delta(\mu)$. Hence, there is a series of irreducible $U(\ggg^+)$-modules $L^0(\lambda_i)$, $i=0,1,\ldots,s$ for $\lambda_i\in\Lambda^+$ such that $\lambda_0=\lambda$ and $\lambda_s=\mu$ with
$\text{Ext}_{\calc_{\ggg^+}}^1(L^0(\lambda_{i-1}), L^0(\lambda_{i}))\ne 0$ or   $\text{Ext}_{\calc_{\ggg^+}}^1(L^0(\lambda_{i}), L^0(\lambda_{i-1}))\ne 0$ for $i=1,\ldots, s$ (see the forthcoming Lemma \ref{lem: conExt quiver}).
 Note that $\Gamma$ is an exact functor.
 Under the former situation, for example, we claim that
\begin{align}\label{eq: non-split ext wanted}
 \text{Ext}^1_{U(\ggg)}(\nabla(\lambda_{i-1}),\nabla(\lambda_i))\ne 0.
 \end{align}
Actually, taking in $\calc_{\ggg^+}$ a non-split extension
\begin{align}\label{eq: non-split ext}
0\longrightarrow L^0(\lambda_i){\overset{\varphi}{\longrightarrow }} N{\overset{\psi}{\longrightarrow }} L^0(\lambda_{i-1})\longrightarrow 0,
\end{align}
one has {\color{red}a}  short exact sequence over $U(\ggg)$:
$$0\longrightarrow \nabla(\lambda_i){\overset{\Gamma(\varphi)}{\longrightarrow}}  \Gamma(N){\overset{\Gamma(\psi)}{\longrightarrow}} \nabla(\lambda_{i-1})\longrightarrow 0. $$
If this one is split, i.e. there exists a $U(\ggg)$-module homomorphism $\pi:\nabla(\lambda_{i-1})\longrightarrow \Gamma(N)$ such that $\Gamma(\psi)\circ\pi=\id_{\nabla(\lambda_{i-1})}$, then one in particular has
$\Gamma(\psi)\circ\pi|_{1\otimes L^0(\lambda_{i-1})}=\id_{1\otimes L^0(\lambda_{i-1})}$.
Notice that $\Gamma(\psi)^{-1}(1\otimes L^0(\lambda_{i-1}))=1\otimes N$.
  Hence $\pi$ maps $1\otimes L^0(\lambda_{i-1})$ to $1\otimes N$. This implies that  the extension (\ref{eq: non-split ext}) is split, which contradicts to the assumption. The claim (\ref{eq: non-split ext wanted}) is proven.

Hence as $U(\ggg)$-modules, the indecomposable module $\nabla(\lambda)$ must lie in the same block as  the indecomposable module $\nabla(\mu)$ does.

Thanks to Lemma \ref{indecom blo 1}, it follows that $\lambda\sim\mu$. The proof is completed.
\end{proof}

\begin{lem}\label{lem: calc irr} Let $M$ be an irreducible module in the category $\calc_{\ggg^+}$  which is defined in the proof above.  Then $M$ is finite-dimensional, which is actually an irreducible $\ggg_0$-module annihilated by $\ggg_{\geq 1}$.
\end{lem}

\begin{proof} At first, $U(\ggg^+)$ has the $\bbz$-graded structure defined by the $\bbz$-gradation of $\ggg^+$. That is, $U(\ggg^+)=\bigoplus_{i\in \bbz}U(\ggg^+)_i$ with all $U(\ggg^+)_i$ being $\ggg_0$-modules. Furthermore, $U(\ggg^+)_{\geq k}:=\bigoplus_{i\geq k}U(\ggg^+)_i$ is a regular $U(\ggg^+)$-module.

 For $M=\sum_{i\in \bbz}M_i$, we suppose that $M_0$ is nonzero without loss of generality.  Take a nonzero vector $v\in M_0$. By assumption, $V_0=U(\ggg_0)v$ is a finite-dimensional subspace in $M_0$. By the irreducibility of $M$, we have $M=\sum_{i\geq 0}V_i$ where $V_i=U(\ggg^+)_iV_0$. Furthermore, set
 $$M^{(k)}:= \sum_{i\geq k}V_{i}.$$
 Then all $M^{(k)}$ are $U(\ggg^+)$-submodules of $M$ and $M\slash M^{(1)}$ is finite-dimensional. The irreducibility of $M$ yields that for any $k$,  $M^{(k)}$ either coincides with $M$ itself or equals to zero. Combining with the filtration $M=M^{(0)}\supset M^{(1)}\supset M^{(2)}\supset\cdots$ along with the fact that $\ggg_{\geq1}$ is nilpotent, we have that if $M=M^{(1)}$, by Nakayama Lemma $M=0$. It's a contradiction. So it must happen that $M^{(1)}=0$. Note that
 $M=M\slash M^{(1)}$ is finite-dimensional, which actually coincides with $V_0$.
  Consequently,  $M$ is irreducible over $\ggg_0$, annihilated by $\ggg_{\geq1}$.
\end{proof}

\begin{lem}\label{lem: conExt quiver} In $\calc_{\ggg^+}$, any two composition factors of $\Delta(\mu)$ lie in a connected Ext-quiver. This is to say, if $L^0(\lambda), L^0(\lambda')$ are two composition factors
of $\Delta(\mu)$, then there are a series of different $\lambda_i$, $i=0,1,\ldots, s$ with $\lambda_0=\lambda$ and $\lambda_s=\lambda'$ such that $\text{Ext}_{\calc_{\ggg^+}}^{\hskip2pt1}(\lambda_{i-1},\lambda_i)\ne 0$ or $\text{Ext}_{\calc_{\ggg^+}}^{\hskip2pt1}(\lambda_{i},\lambda_{i-1})\ne 0$
for all $i=1,\ldots,s$.
\end{lem}
\begin{proof} We only need to show the lemma for the fixed  $\lambda'=\mu$ because $\Delta(\mu)$ has a simple head isomorphic to $L^0(\mu)$ in $\calc_{\ggg^+}$. For this we write $\Delta(\mu)=\bigoplus_{i\geq 0} \Delta(\mu)_i$ which has a natural $\bbz$-grading arising from the gradation of $\ggg^+=\sum_{i\geq 0}\ggg_{i}$, furthermore $\Delta(\mu)$ admits a $U(\ggg^+)$-module filtration $\{\Delta(\mu)^{(k)}:=\bigoplus_{i\geq k}\Delta(\mu)_i\mid k\in \bbz_{\geq0}\}$.

By construction, there is $k\geq 1$ such that $L^0(\lambda)$ is a subquotient of $\Delta(\mu)^{(k)}$. We further suppose without loss of generality, that $L^0(\lambda)\cong M\slash N$ for
$M, N\in\calc_{\ggg^+}$ satisfying  $M, N\subset\Delta(\mu)^{(k)}$.

 Consider $\overline{\Delta(\mu)}:=\Delta(\mu)\slash \Delta(\mu)^{(k+1)}$ which is a finite-dimensional and indecomposable object in $\calc_{\ggg^+}$. Clearly $\overline{\Delta(\mu)}$ also has a head isomorphic to $L^0(\mu)$.
Set $\phi: \Delta(\mu)\rightarrow \overline{\Delta(\mu)}$ to be the canonical surjective homomorphism in $\calc_{\ggg^+}$. Then $L^0(\lambda)\cong \overline{\phi^{-1}(\phi(M))}\slash \overline{\phi^{-1}(\phi(N))}$ which is still a composition factor of $\overline{\Delta(\mu)}$. Here $\phi^{-1}(\bullet)$ stands for the preimage in $\Delta(\mu)$ of $\bullet$.

%
 By the same arguments as in the finite-dimensional module category, it can be shown that $L^0(\lambda)$ and $L^0(\mu)$ lie in a connected Ext-quiver in $\calc_{\ggg^+}$.

 The proof is completed.
\end{proof}

\subsection{}
Recall  $\Xi=\epsilon_{1}+\epsilon_{2}+\cdots+\epsilon_n.$
We have the following elementary observation, the proof of which follows directly from the forthcoming Lemma \ref{mults} in Appendix B.
\begin{lem}\label{lemm for block S(n)}
	Let $\ggg=\bar{S}(n).$ Then the following statements hold.
	\begin{itemize}
		\item[(1)] For $l\in\mathbb{C},$ we have  $l\Xi\sim (l+\mathbb{Z})\Xi$.
		\item[(2)] For $\lambda=b\Xi+c\epsilon_{n},b\in\mathbb{C}, c\in\mathbb{Z}_{\leq0},$ we have  $\lambda\sim b\Xi$.
		\item[(3)] For $\lambda=a\epsilon_{1}+b\Xi,a\in\mathbb{Z}_{\geq0},b\in\mathbb{C},$ we have  $\lambda\sim b\Xi$.
	\end{itemize}
\end{lem}

The following result is crucial for determining the blocks for the Lie superalgebra $\bar{S}(n)$ of special type.

\begin{prop}\label{prop for bolck S(n)}
Let $\ggg=\bar{S}(n)$ and $\lambda=\lambda_1\epsilon_1+\lambda_2\epsilon_2+\cdots+\lambda_n\epsilon_n\in\Lambda^+.$ Then $\lambda\sim\lambda_1\Xi$, i.e., $L(\lambda)$ belongs to the same block as that $L(\lambda_1\Xi)$ lies in.
\end{prop}

\begin{proof}
We begin with the following Claim.

\textsc{Claim}: if there exists some $i$ with $2\leq i\leq n-2$ such that
$\lambda_i>\lambda_{i+1}$ and $\lambda_{n-1}\geq \lambda_{n}+1$, then
\begin{equation}\label{Sn-eq-1}
\lambda\sim \lambda+\epsilon_{i+1}.
\end{equation}

Indeed, for any $j$ with $2\leq j\leq n-1$, $\xi_1\xi_2\cdots\xi_jD_n\otimes v_{\lambda}^0$ is an $\nnn^+$-maximal weight vector of weight $ \lambda+\epsilon_1+\epsilon_2+\cdots+\epsilon_j-\epsilon_n$.
By Lemma \ref{comp}, we know that $L(\lambda)$ and $L(\lambda+\epsilon_1+\epsilon_2+\cdots+\epsilon_j-\epsilon_n)$ {\color{red}lie} in the same block. i.e.,
$\lambda\sim \lambda+\epsilon_1+\epsilon_2+\cdots+\epsilon_j-\epsilon_n$. In particular,
\begin{equation}\label{Sn-eq-2}
\lambda\sim \lambda+\epsilon_1+\epsilon_2+\cdots+\epsilon_{n-1}-\epsilon_n.
\end{equation}
By the condition of the claim, $\lambda-\epsilon_1-\epsilon_2-\cdots-\epsilon_i+\epsilon_n\in\Lambda^+$, so $\lambda\sim \lambda-\epsilon_1-\epsilon_2-\cdots-\epsilon_i+\epsilon_n$ and $\lambda-\epsilon_1-\epsilon_2-\cdots-\epsilon_i+\epsilon_n\sim \lambda+\epsilon_{i+1}$, it follows that $\lambda\sim \lambda+\epsilon_{i+1}$. The claim follows.

With the above claim, we carry on the proof by taking all possibilities of $\lambda_1$ into the arguments.

\textsc{Case 1}: $\lambda_1=\lambda_2$.

In this case, set $\mu=\lambda+\epsilon_1+\epsilon_2+\cdots+\epsilon_{n-1}-\epsilon_n=(\lambda_1+1)\epsilon_1+(\lambda_2+1)\epsilon_2+\cdots+
(\lambda_{n-1}+1)\epsilon_{n-1}+(\lambda_n-1)\epsilon_n$. Then $\lambda\sim\mu$ by (\ref{Sn-eq-2}). Moreover, we can use (\ref{Sn-eq-1}) successively to obtain $\mu\sim (\lambda_1+1)(\epsilon_1+\epsilon_2+\cdots+\epsilon_{n-1})+(\lambda_n-1)\epsilon_n$. Hence $\lambda\sim\mu \sim (\lambda_1+1)\Xi\sim \lambda_1\Xi$ by Lemma \ref{lemm for block S(n)}(1) and (2), as desired.

\textsc{Case 2}: $\lambda_1\neq \lambda_2$.

By using similar arguments as in \textsc{Case 1}, without loss of generality,  we can assume $$\lambda_1>\lambda_2=\lambda_3=\cdots=\lambda_{n-1}\gg\lambda_n.$$

\textsc{Subcase (i):} $\lambda_1-\lambda_2$ is even.

Recall that $\Delta(\lambda)$ contains a $\ggg_0$-submodule $\bar{S}(n)_1\otimes_{\bbc}L^0(\lambda)$, and $\bar{S}(n)_1\cong L^0(\epsilon_1+\epsilon_2-\epsilon_n)$ as $\ggg_0$-modules. Take $w_1=(1n)\in\mathfrak{S}_n$ (the symmetric group on $n$ letters), which is the Weyl group of $\ggg_0$. Set
\begin{align*}
\nu_1&:=\lambda+\frac{1}{2}(\lambda_1-\lambda_2)w_1(\epsilon_1+\epsilon_2-\epsilon_n)\cr
&=\frac{1}{2}(\lambda_1+\lambda_2)\epsilon_1+\frac{1}{2}(\lambda_1+\lambda_2)\epsilon_2+\lambda_3\epsilon_3+
\cdots+\lambda_{n-1}\epsilon_{n-1}+\left(\lambda_n+\frac{1}{2}(\lambda_1-\lambda_2)\right)\epsilon_n\in\Lambda^+.
\end{align*}
It follows from \cite[Theorem 2.10]{Kum88} and Lemma \ref{comp} that $\lambda\sim\nu_1$. Furthermore, $\nu_1\sim \frac{1}{2}(\lambda_1+\lambda_2)\Xi\sim\lambda_1\Xi$ by the claim in \textsc{Case 1} and Lemma \ref{lemm for block S(n)}(1). Consequently, $\lambda\sim\lambda_1\Xi$.

\textsc{Subcase (ii):} $\lambda_1-\lambda_2$ is odd.

In this case, take $w_2=(13)(2n)\in\mathfrak{S}_n$. Set
\begin{align*}
\nu_2&:=(\lambda+2(\epsilon_1+\epsilon_2-\epsilon_n))+w_2(\epsilon_1+\epsilon_2-\epsilon_n)\cr
&=(\lambda_1+2)\epsilon_1+(\lambda_2+1)\epsilon_2+(\lambda_3+1)\epsilon_3+\lambda_4\epsilon_4
\cdots+\lambda_{n-1}\epsilon_{n-1}+(\lambda_n-1)\epsilon_n\in\Lambda^+.
\end{align*}
It follows from \cite[Theorem 2.10]{Kum88} and Lemma \ref{comp} that $\lambda\sim\nu_2$. Now, $(\lambda_1+2)-(\lambda_2+1)$ is even in $\nu_2$.  The claim in \textsc{Subcase (i)} implies that $\nu_2\sim (\lambda_1+2)\Xi\sim \lambda_1\Xi$. Hence, we also have $\lambda\sim\lambda_1\Xi$, as desired. We complete the proof.
\end{proof}

\subsection{}\label{for I}
Continue to investigate  the standard modules. Denote by $\Delta(\lambda)_X$ the standard $X(n)$-module for $X\in \{W,\bar S, \bar H, \och\}$. i.e. $\Delta(\lambda)_X=U(X(n))\otimes_{U(P)}L^0(\lambda)$. Similarly, we can define $\comx$, $I(\lambda)_X$, $L(\lambda)_X$,  $\bfe_X$ and $\Upsilon(\lambda)_X$.
 In this subsection, we establish some relation between standard modules for  $\ochn$ and $\bhn$. The following preliminary result is important for us.

\begin{lem}\label{key lem for blocks of H}
Let  $\phi\in\Hom_{\bhn}(\Delta(\lambda)_{\ochn}, \Delta(\lambda)_{\ochn})$ with $\phi^2=\phi$.  If $\phi|_{\Delta(\lambda)_{\bhn}}=0$, then
$\phi=0$.
\end{lem}

\begin{proof}
Recall that $\ochn=\bhn\oplus\mathbb{C}D_H(\xi_1\cdots\xi_n)$. It suffices to show that
\begin{equation}\label{equiv assertion}
\phi((D_H(\xi_1\cdots\xi_n))^k\otimes v)=0,\,\,\forall\, k\in\mathbb{N}^+, v\in L^0(\lambda).
\end{equation}
We use induction on $k$ to show (\ref{equiv assertion}).

Since $\phi$ keeps the grading and weight spaces invariant, we can assume
$$\phi(D_H(\xi_1\cdots\xi_n)\otimes v^0_{\lambda})=cD_H(\xi_1\cdots\xi_n)\otimes v^0_{\lambda}+\sum\limits_{i=1}^s u_i\otimes v_i,$$
where $c\in\mathbb{C}$, $u_i\in U(\bhn_{\geq 1}), v_i\in L^0(\lambda), 1\leq i\leq s,$  and all $ v_i's$ are linearly independent.
On one hand,
\begin{align*}
\phi^2(D_H(\xi_1\cdots\xi_n)\otimes v^0_{\lambda})&=\phi(cD_H(\xi_1\cdots\xi_n)\otimes v^0_{\lambda}+\sum\limits_{i=1}^s u_i\otimes v_i)\cr
&=c\phi(D_H(\xi_1\cdots\xi_n)\otimes v^0_{\lambda})\cr
&=c^2D_H(\xi_1\cdots\xi_n)\otimes v^0_{\lambda}+\sum\limits_{i=1}^s cu_i\otimes v_i.
\end{align*}
On the other hand, we have
$$\phi^2(D_H(\xi_1\cdots\xi_n)\otimes v^0_{\lambda})=\phi(D_H(\xi_1\cdots\xi_n)\otimes v^0_{\lambda})=cD_H(\xi_1\cdots\xi_n)\otimes v^0_{\lambda}+\sum\limits_{i=1}^s u_i\otimes v_i.$$
Hence, $c=1$, or $c=0$ and $\sum\limits_{i=1}^s u_i\otimes v_i=0$. We claim that the latter happens. Indeed, if $c=1$, then for any $1\leq j\leq n$, we have
$$\phi(D_j(D_H(\xi_1\cdots\xi_n)\otimes v^0_{\lambda}))=D_j\phi(D_H(\xi_1\cdots\xi_n)\otimes v^0_{\lambda})=(-1)^{j-1}D_H(\xi_1\cdots \hat{\xi_j}\cdots\xi_n)\otimes v^0_{\lambda}+\sum\limits_{i=1}^s[D_j, u_i]\otimes v_i.$$
However,
$$\phi(D_j(D_H(\xi_1\cdots\xi_n)\otimes v^0_{\lambda}))=(-1)^{j-1}\phi(D_H(\xi_1\cdots \hat{\xi_j}\cdots\xi_n)\otimes v^0_{\lambda})=0.$$
We get a contradiction. Hence, $c=0$ and $\sum\limits_{i=1}^s u_i\otimes v_i=0$, i.e., $\phi(D_H(\xi_1\cdots\xi_n)\otimes v^0_{\lambda})=0$. Since $L^0(\lambda)=U(\nnn^-)v^0_{\lambda}$, it follows that
\begin{align*}
&\phi(D_H(\xi_1\cdots\xi_n)\otimes L^0(\lambda))\cr
=&\phi(D_H(\xi_1\cdots\xi_n)U(\nnn^-)\otimes v^0_{\lambda})\cr
=&\phi(U(\nnn^-)D_H(\xi_1\cdots\xi_n)\otimes v^0_{\lambda})\cr
=&U(\nnn^-)\phi(D_H(\xi_1\cdots\xi_n)\otimes v^0_{\lambda})\cr
=&0,
\end{align*}
i.e., (\ref{equiv assertion}) holds for $k=1$.

Now suppose $\phi((D_H(\xi_1\cdots\xi_n)^{l}\otimes L^0(\lambda))=0$ for $l<k$. We need to show that $$\phi((D_H(\xi_1\cdots\xi_n)^k\otimes L^0(\lambda))=0.$$
Since $\phi$ keeps the grading and weight spaces invariant, we can assume
$$\phi((D_H(\xi_1\cdots\xi_n))^k\otimes v^0_{\lambda})=a(D_H(\xi_1\cdots\xi_n))^k\otimes v^0_{\lambda}+\sum\limits_{i=1}^t w_i\otimes \nu_i,$$
where $a\in\mathbb{C}$, $w_i\in U(\bhn_{\geq 1})\mathcal{B}, \nu_i\in L^0(\lambda), 1\leq i\leq t$, $\mathcal{B}=\text{span}_{\mathbb{C}}\{(D_H(\xi_1\cdots\xi_n))^{i}\mid 0\leq i\leq k-1\}$, and all $\nu_i's, 1\leq i\leq t$, are linearly independent.
On one hand,
\begin{align*}
\phi^2((D_H(\xi_1\cdots\xi_n))^k\otimes v^0_{\lambda})&=\phi(a(D_H(\xi_1\cdots\xi_n))^k\otimes v^0_{\lambda}+\sum\limits_{i=1}^t w_i\otimes \nu_i)\cr
&=a\phi((D_H(\xi_1\cdots\xi_n))^k\otimes v^0_{\lambda})\cr
&=a^2(D_H(\xi_1\cdots\xi_n))^k\otimes v^0_{\lambda}+\sum\limits_{i=1}^t aw_i\otimes \nu_i.
\end{align*}
On the other hand, we have
$$\phi^2((D_H(\xi_1\cdots\xi_n))^k\otimes v^0_{\lambda})=\phi((D_H(\xi_1\cdots\xi_n))^k\otimes v^0_{\lambda})=a(D_H(\xi_1\cdots\xi_n))^k\otimes v^0_{\lambda}+\sum\limits_{i=1}^t w_i\otimes \nu_i.$$
Similar arguments as in the case $k=1$ yield that $a=0$ and $\sum\limits_{i=1}^t w_i\otimes \nu_i=0$, and furthermore $\phi((D_H(\xi_1\cdots\xi_n))^k\otimes v)=0$, i.e., (\ref{equiv assertion}) holds for $k$. Consequently, $\phi=0$, as desired.
\end{proof}
As a consequence of Lemma \ref{key lem for blocks of H}, we have the following result.
\begin{coro}\label{coro for H}
As an $\bhn$-module,  $\Delta(\lambda)_\ochn$ is indecomposable.
\end{coro}

\begin{proof}
Let  $f$ be an element of $\Hom_{\bhn}(\Delta(\lambda)_{\ochn}, \Delta(\lambda)_{\ochn})$ with $f^2=f$. Then $f|_{\Delta(\lambda)_{\bhn}}=0$ or $1$ due to the indecomposability of $\Delta(\lambda)_{\bhn}$ as an $\bhn$-module. If $f|_{\Delta(\lambda)_{\bhn}}=0$, then $f=0$ by Lemma \ref{key lem for blocks of H}. If $f|_{\Delta(\lambda)_{\bhn}}=1$, then $(\textbf{id}-f)|_{\Delta(\lambda)_{\bhn}}=0$ and $(\textbf{id}-f)^2=\textbf{id}-2f+f^2=(\textbf{id}-f)$. It also follows from Lemma \ref{key lem for blocks of H} that $(\textbf{id}-f)|_{\Delta(\lambda)_{\ochn}}=0$, i.e., $f=\textbf{id}$. This implies that $0$ and $\textbf{id}$ are the only two idempotents in $\Hom_{\bhn}(\Delta(\lambda)_{\ochn}, \Delta(\lambda)_{\ochn})$. Then it follows from \cite[Proposition 5.10]{AF92} that $\Delta(\lambda)_\ochn$ is an indecomposable $\bhn$-module.
\end{proof}

\subsection{Revisit to $I(\lambda)$}\label{revisit to I}
\begin{prop} \label{I prop}  Let $\ggg=W(n),\bar{S}(n),\bar{H}(n),\ochn.$ Then all composition factors in $I(\lambda)$ lie in the same block.
\end{prop}

\begin{proof}    Note that $I(\lambda)=\bigoplus_{\mu\in \Upsilon(\lambda)}
	P(\mu)^{\oplus a_{\lambda\mu}},$ here $a_{\lambda\mu}\in\bbz_{>0}$. 
	Each $P(\mu)$ ($\mu\in \Upsilon(\lambda)$) is actually the projective cover of both $L(\mu)$ and $\Delta(\mu)$.
	In order to prove the proposition,  it suffices by the definition of blocks to prove
	\begin{align}\label{I lambda pf}
	\mu\sim\lambda,\,\, \forall\mu\in \Upsilon(\lambda).
	\end{align}
	In the following we will prove the proposition for the case of $\wn$ by verifying the formula (\ref{I lambda pf}) (consequently, the case of $\bar{S}(n)$ is easily solved).  For the case $\ochn$, we will prove the proposition by partially verifing
	the formula (\ref{I lambda pf}) and  accomplishing the remaining cases by using Corollary \ref{coro for H}. So the arguments will be divided into cases.
	
	(i) Assume $\ggg=W(n)$.   Take $\mu\in \Upsilon(\lambda)$.
	For $\Delta(\mu)=U(\ggg)\otimes_{U(\sfp)}L^0(\mu)$, keeping the notations in  Lemma \ref{basic nat rep}, we see that
	$\Delta(\mu)$  contains a $\ggg_0$-submodule  $M^+(\mu)$ (see Remark \ref{g0 regarding}). From Lemma \ref{basic nat rep}(3),   there is a $\ggg_0$-maximal  vector $m_\lambda$ in $M^+(\mu)$, i.e. $\nnn^+m_\lambda=0$, $Hm_\lambda=\lambda(H)m_\lambda$ for any $H\in \bar\hhh$. By Lemma \ref{comp} we know $\mu\sim\lambda$, as desired.
	

	(ii) Assume $\ggg=\bar{S}(n)$. For any $\mu\in \Upsilon(\lambda)$, it follows from Lemma \ref{lemm for block S(n)}, Proposition \ref{prop for bolck S(n)} and Remar \ref{proj remark}(3) that $\mu\sim\lambda$. Hence, the assertion for $\ggg=\bar{S}(n)$ is proven.
	
	(iii) Assume $\ggg=\ochn$. By the definition of $\Upsilon(\lambda)$ (see (\ref{upsilon})), we can write  $\Upsilon(\lambda)=\bigcup_{i=0}^n \Upsilon_i(\lambda)$ with
	\begin{align}\label{upsilon i}
	\Upsilon_i(\lambda)=\{\mu\in \Upsilon(\lambda)\mid (\bigwedge\nolimits^i\ggg_{-1}\otimes_\bbc L^0(\lambda):L^0(\mu))_{\ggg_0}\ne 0 \}.
	\end{align}
	
	By the same arguments as (i), it follows from Lemmas \ref{basic nat rep ch} and \ref{comp} that
	for $\mu\in \Upsilon(\lambda)$:
	\begin{align}
	&\mu\sim\lambda-2\delta \mbox{ if }\mu\in \bigcup_{i\geq 3}\Upsilon_i(\lambda);\label{qq1}\\
	&\mu-(n-2)\delta\sim\lambda-2\delta\mbox{ if }\mu\in \Upsilon_1(\lambda);\label{qq2}\\
	&\mu-(n-4)\delta\sim\lambda-2\delta \mbox{ if }\mu\in \Upsilon_2(\lambda).\label{qq3}
	\end{align}
	As all standard modules are indecomposable,  the above formula (\ref{qq1})  implies that all $\Delta(\mu)$ for $\mu\in \bigcup_{i\geq 3}\Upsilon_i(\lambda)$ lie in the same block as $L(\lambda-2\delta)$ does. Especially,	
 by Remark \ref{proj remark}(3)($4^\circ$),
We have the following result:
 $$L(\lambda+\epsilon_{1}+\cdots+\epsilon_k-(n-k)\delta) \mbox{~and~} L(\lambda-2\delta) \mbox{~lie~ in ~the~ same~ block.}$$
 Furthermore, we divide the following arguments into two different cases.
	
	(Case 1) For $\ggg=\ochn$ with $n=2r$. In this case, $r\geq 3$ by the assumption that $n\geq5$.
	
		\textbf{Claim 1:} $L(\lambda)$ and $L(\lambda+\sum_{i=1}^r\epsilon_i+r\delta)$ lie in the same block.
		Let $k=r$ and $k=0,$
		by the above result we know that  both $L(\lambda+\sum_{i=1}^r\epsilon_i-r\delta)$ and $L(\lambda-n\delta)$ lie in the same block. Due to the arbitrariness of $\lambda,$ one can change $\lambda$ to $\lambda-n\delta,$ then the claim follows.
		
		\textbf{Claim 2:} $L(\lambda)$ and $L(\lambda+\sum_{i=1}^r\epsilon_i+(r-2)\delta)$ share one block. This claim can be checked by the fact that  $D_H(\Pi_{i=1}^r\xi_i)\otimes v^0_\lambda$ is a $\ggg_0$-maximal vector
and Lemma \ref{comp}.

By the arbitrariness of $\lambda$ (or by translating $\lambda$ to $\lambda-(\sum_{i=1}^r\epsilon_i+(r-2)\delta)$ in the previous claims), we have that $L(\lambda)$ and $L(\lambda\pm 2\delta)$ lie in the same block. Furthermore, we see that  $L(\lambda)$, $L(\lambda\pm 2\delta)$ and $L(\lambda\pm n\delta)$ lie in the same block.
	

	(Case 2) For $\ggg=\ochn$ with $n=2r+1$, by a direct verification, the standard module $\Delta(\lambda)$ admits $\ggg_0$-maximal vectors  $D_H(\prod_{i=1}^r\xi_i)\otimes v^0_\lambda$ and  $D_H((\prod_{i=1}^r\xi_i)\xi_{2r+1})\otimes v^0_\lambda$.
Hence by Lemma \ref{comp} we get $L(\lambda)$, $L(\lambda+\sum_{i=1}^r\epsilon_i+(r-2)\delta)$ and $L(\lambda+\sum_{i=1}^r\epsilon_i+(r-1)\delta)$ share the same block. By the arbitrariness of $\lambda$  again (or by translating $\lambda$ to $\lambda-(\sum_{i=1}^r\epsilon_i)+(2-r)\delta$ in the above), we have that in this case, $L(\lambda)$ and $L(\lambda\pm \delta)$ lie in the same block. Consequently,  $L(\lambda)$, $L(\lambda\pm 2\delta)$ and $L(\lambda\pm n\delta)$ share the same block.
	
With the above arguments, we can directly  deduce   that   not only for $\mu\in \Upsilon_{\geq 3}(\lambda)$ but also for $\mu\in \Upsilon_1(\lambda)\cup\Upsilon_2(\lambda)$, all $L(\mu)$ lie in the same block as $L(\lambda)$ does.
Hence we indeed prove that all composition factors in $\Delta(\mu)$ for $\mu\in \Upsilon(\lambda)$, thereby all composition factors in $I(\lambda)$, lie in the same block. We have proven the proposition in this case.

	(iv) Assume $\ggg=\bar H(n)$. 
 Recall that $I(\lambda)_{\bar H(n)}$ has a $\Delta(\mu)_{\bar H(n)}$-filtration and $L(\mu)_{\bar H(n)}$ is the head of $\Delta(\mu)_{\bar H(n)}.$
	By Lemma \ref{indecom blo 1} and the indecomposability of $\Delta(\mu)_{\bar H(n)},$
	we need to show that all $L(\mu)_{\bar H(n)}, \mu\in\Upsilon(\lambda)_\bhn ,$ belong to the same block.	
	
	Recall that both $\bhn$ and $\ochn$ have the same $0$-graded spaces $ \frak{so}(n)\oplus\bbc\sfd$. The parameters of isomorphism classes of irreducible modules for $\comi_\ochn$ and for $\comi_\bhn$ are the same, arising from $\{L^0(\lambda)\mid \lambda\in \Lambda^+\}$ for $\ggg_0$. Since $\bhn$ and $\ochn$ has the same $-1$-graded spaces,  $\Upsilon(\lambda)_\bhn=\Upsilon(\lambda)_\ochn$.

 By the arguments in (iii), we have known that all $\Delta(\mu)_\ochn,\mu\in\Upsilon(\lambda)_\bhn,$ lie in the same block in $\comi_\ochn$ as  $L(\lambda)_\ochn$ does. Hence all  $\bhn$-modules $\Delta(\mu)_\ochn$, $\mu\in \Upsilon(\lambda)_\ochn$, lie in the same block of $\comi_\bhn$.
Because $\Delta(\mu)_\ochn$  admits an $\bhn$-irreducible quotient $L(\mu)_\bhn$, By Lemma \ref{indecom blo 1} and Corollary \ref{coro for H}, all $L(\mu)_{\bar H(n)}, \mu\in\Upsilon(\lambda)_\bhn ,$ belong to the same block.	
 So the desired result follows.

Summing up, we finish the proof.
\end{proof}

When $\ggg=\bar H(2r),$ set
$\aleph_r\in\{0,1\}$ satisfying  ${\aleph_r}\equiv r\mod 2$ for $\ggg=\bar H(2r)$.
Then we put forward some additional new notations
$$\begin{cases}
\Theta_B:=\epsilon_1+\cdots+\epsilon_{r-1}+\epsilon_r, &\mbox{ for }H(2r+1).\\
\Theta_{D,{\aleph_r}}:=\epsilon_1+\cdots+\epsilon_{r-1}+\epsilon_r +{\aleph_r}\delta,&\mbox{ for } H(2r).
\end{cases}$$
(along with already appointing
$\Xi:=\sum_{i=1}^n\epsilon_i, \mbox{ for }\ggg=X(n), \mbox{ with } X\in \{W,\bar S\}$).

Set $$\tilde\delta:=\begin{cases} \delta, &\mbox{ for } H(2r+1);\\
2\delta,  &\mbox{ for } H(2r).
\end{cases}
$$
Then we have the following corollary.
\begin{coro}\label{height one lem}  
	For any $L(\lambda)\in \bfe$, the following statements hold.
	\begin{itemize}
		\item[(1)] When  $\ggg=X(n)$ with $X\in\{W,\bar S, \bar{H}\}.$ If there exists $-\beta\in \wt(\bigwedge(\ggg_{-1}))$ such that   $\lambda-\beta\in \Upsilon(\lambda)$,   then $L(\lambda)$ and $L(\lambda-\beta)$ share the same block. 
		\item[(2)] When  $\ggg=X(n)$ with $X\in\{W,\bar S\}$,  then
			$L(\lambda)$ and $L(\lambda+\sum_{i=k}^n\epsilon_i)$, $1\leq k\leq n$, lie in the same block. In particular,
			$L(\lambda)$ and $L(\lambda-\Xi)$ lie in the same block.
		\item[(3)] When  $\ggg=\bar H(n)$, then $L(\lambda)$ and $L(\lambda-\tilde\delta)$ lie in the same block.
		\item[(4)] When  $\ggg=\bar H(2r+1)$, then $L(\lambda)$ and $L(\lambda+\sum_{i=1}^{k}\epsilon_i-(n-k)\delta)$ lie in the same block.
		In particular,
		$L(\lambda)$ and $L(\lambda+\sum_{i=1}^{r}\epsilon_i)$ lie in the same block.
		\item[(5)] When  $\ggg=\bar H(2r)$, then  $L(\lambda)$ and $L(\lambda-\Theta_{D,{\aleph_r}})$ lie in the same block (${\aleph_r}\in\{0,1\}$ with ${\aleph_r}\equiv r\mod 2$).
	\end{itemize}
\end{coro}
\begin{proof} 

	(1) This is a direct consequence of  the proof for Proposition \ref{I prop}.
	
	(2) This is the consequence of Remark \ref{proj remark}(3)($3^\circ$) and the result in (1).
			
	(3) When $n=2r$  ({\sl{resp}}. $n=2r+1$), one can check it by the same arguments as the process for proof of Proposition \ref{I prop}(iii) for case 1(resp. case 2).  

	(4)-(5) By Remark \ref{proj remark} $(3)(4^{\circ}),$
 $I(\lambda)$ admits the composition factor  $L(\lambda+ \sum_{i=1}^k\epsilon_i-(n-k)\delta)$.  So $L(\lambda)$ and $L(\lambda+ \sum_{i=1}^k\epsilon_i-(n-k)\delta)$ lie in the same block due to Proposition \ref{I prop}. Thanks to (3), we further have that $L(\lambda)$ and $L(\lambda+\sum_{i=1}^r\epsilon_i)$ lie in the same block for $H(2r+1)$ and $H(2r)$ with even $r$. Similarly, for $H(2r)$ with odd $r$, we can check that $L(\lambda)$ and $L(\lambda+\sum_{i=1}^r\epsilon_i+\delta)$ lie in the same block.
\end{proof}




\subsection{Depth Lemma and parity Lemma}

We will analyse the relation of depths for simple objects in a block. Suppose that $L(\lambda)$ is given, and $\dpt(L(\lambda))=d$. Then by the construction of $P(\lambda)$ (see Remark \ref{proj remark}(1)), the depth of each composition factor is consequently determined. Conversely, for any given composition factor $L(\mu')=L(\mu')_{d'}$ in $P(\lambda')$, the depth of $P(\lambda')$ (thereby the depth of $L(\lambda')$) is definitely determined by the predefined  depth of $L(\mu')$.  From this fact and the definition of blocks  we can easily have the following depth lemma.
We firstly introduce some new notations before the following lemma.
 Let $\mu=\mu_{1}\epsilon_{1}+\mu_{2}\epsilon_{2}+\cdots+\mu_{n}\epsilon_{n}$ be an element of $\bar\hhh^*$ for $\ggg=X(n)$ with $X\in \{W,\bar S\}$, and $\mu=\mu_{1}\epsilon_{1}+\mu_{2}\epsilon_{2}+\cdots+\mu_{r}\epsilon_{r}+c\delta$ for $\ggg=\bar H(n)$. We define the length of $\mu,$ which is denoted by $\ell(\mu)$, as below

 \begin{align}\label{length}
 \ell(\mu)=\begin{cases} \sum_{i=1}^{n}\mu_{i}, &\mbox{ for }\ggg=X(n), \mbox{ with } X\in \{W,\bar S\};\\
 c &\mbox{ for }\ggg=\bar H(n).
 \end{cases}
 \end{align}
 Obviously, $$ \ell(\lambda\pm\mu)=\ell(\lambda)\pm\ell(\mu).$$

\begin{lem} \label{dpt lem} (\textsc{Depth Lemma})
	
		\item[(1)] If $L(\mu)$ and $L(\nu)$ are in the same block, then
$\dpt(L(\mu))-\dpt(L(\nu))=\ell(\mu-\nu)$.

		\item[(2)] 	For any $\lambda\in \Lambda^+$, and different $d_1, d_2\in \bbz$,
		$L(\lambda)_{d_1}$ and $L(\lambda)_{d_2}$ do not lie in the same block.

\end{lem}

\pf (1) The proof is divided into the following steps.

\textsc{Claim I}: If $(\Delta(\lambda):L(\mu))\neq0$ and $(\Delta(\lambda):L(\nu))\neq0,$ then
\begin{equation}\label{eq}
\dpt(L(\mu))-\dpt(L(\nu))=\ell(\mu-\nu).
\end{equation}

Set $\lfloor L^0(\lambda)\rfloor=d.$ So $\Delta(\lambda),L(\lambda)$ are all of depth $d.$
Recall that $\Delta(\lambda)=U(\ggg)\otimes_{U(P)}L^0(\lambda)\cong U(\ggg_{\geq1})\otimes_{\mathbb{C}}L^0(\lambda)$ as a vector space.
So if $v\in U(\ggg_{\geq1})_i\otimes_{\mathbb{C}}L^0(\lambda)$ is a homogeneous element of $\Delta(\lambda),$ then $\mathrm{degree}(v)=d+i.$

Now let $L(\mu)$ be an irreducible $U(\ggg)$-module with $(\Delta(\lambda):L(\mu))\neq0.$
Then there exists an inclusion of submodule $\Delta(\lambda)\supseteq M \supseteq N \supseteq0$ such that $M/N\cong L(\mu).$ Let $v_{\mu}\in M/N$ be a maximal vector of $L(\mu)$.  If $v_{\mu}\in U(\ggg_{\geq1})_i\otimes_{\mathbb{C}}L^0(\lambda),$ then
\begin{equation}\label{rel1}
\dpt(L(\mu))=i+d=\dpt(L(\lambda))+\ell(\mu-\lambda).
\end{equation}
Similarly, we have
\begin{equation}\label{rel2}
\dpt(L(\nu))=\dpt(L(\lambda))+\ell(\nu-\lambda).
\end{equation}
Consequently, the equality (\ref{eq}) holds due to $\ell(\mu-\lambda)-\ell(\nu-\lambda)=\ell(\mu-\nu)$. The first claim is proven.

\textsc{Claim II}: If $(P(\lambda):L(\mu))\neq0$ and $(P(\lambda):L(\nu))\neq0,$ then $\dpt(L(\mu))-\dpt(L(\nu))=\ell(\mu-\nu).$

Since $P(\lambda)$ is a direct summand of $I(\lambda)$ (see Theorem \ref{projective thm}), it suffices to prove this claim for $I(\lambda),$ i.e. If $(I(\lambda):L(\mu))\neq0$ and $(I(\lambda):L(\nu))\neq0,$ then $\dpt(L(\mu))-\dpt(L(\nu))=\ell(\mu-\nu).$ Set $\dpt(L(\lambda))=d.$
Assume that
$$I(\lambda)=M_1 \supseteq M_2\supseteq M_3\supseteq\cdots\supseteq M_l\supseteq0$$
is the descending sequence such that $M_i/M_{i+1}\cong \Delta(\lambda_i)$ shown in Theorem \ref{projective thm}. Then we have that $\ell(\Delta(\lambda_i))=d+\ell(\lambda_{i}-\lambda).$
Denote by $s=\mbox{max}\{i \mid L(\mu)$ is a subquotient of $M_i \},$
$t=\mbox{max}\{j \mid L(\nu)$ is a subquotient of $M_j \}.$

If $s=t,$ then there exists the following down sequence
$$M_s\supseteq N_1\supseteq N_2\supseteq M_{s+1}$$
such that $N_1/N_2\cong L(\mu).$ because
$$N_1/N_2\hookrightarrow M_s/N_2\cong M_s/M_{s+1}\diagup N_2/M_{s+1}\cong \Delta(\lambda_s)\diagup N_2/M_{s+1} ,$$
$L(\mu)$ can be realized as a sub-quotient of $\Delta(\lambda_s)_{d+\ell(\lambda_s-\lambda)}.$ Meanwhile,
$L(\nu)$ can be also realized as a sub-quotient of $\Delta(\lambda_s)_{d+\ell(\lambda_s-\lambda)}.$
Thus $L(\mu)$ and $L(\nu)$ are two sub-quotients of $\Delta(\lambda_s)_{d+\ell(\lambda_s-\lambda)}.$
Then \textsc{Claim I} implies
\textsc{Claim II}.

If $s\neq t,$ assume $s<t$ without loss of generality.
Then by the above discuss, $L(\mu)$ (resp. $L(\nu)$) is a sub-quotient of $\Delta(\lambda_s)_{d+\ell(\lambda_s-\lambda)}$ (resp. $\Delta(\lambda_t)_{d+\ell(\lambda_t-\lambda)}).$
So by the equality (\ref{rel1}) we have

\begin{equation}\label{rel3}
\dpt(L(\mu))=d+\ell(\lambda_s-\lambda)+\ell(\mu-\lambda_s)
\end{equation}

\begin{equation}\label{rel4}
\dpt(L(\nu))=d+\ell(\lambda_t-\lambda)+\ell(\nu-\lambda_t)
\end{equation}

Then the desired assertion follows from  (\ref{rel3})-(\ref{rel4}).

Now the statement (1) of the theorem holds due to the definition of blocks in Subsection \ref{def of block}.

For (2), this is a direct consequence of (1).
\qed

Because $L(\lambda)$ (resp. $\Delta(\lambda)$, $P(\lambda)$) is generated by $v^0_\lambda,$ which is a maximal vector of $L^0(\lambda)$,
the super structure of $L(\lambda)$ (resp. $\Delta(\lambda)$, $P(\lambda)$) is completely determined by the predefined parity $|v^0_\lambda|$ of $v^0_\lambda$. By abuse of the notions and notations with the context being clear, we say that $L(\lambda)$ is of parity $|v^0_\lambda|$, denote $\pty(L(\lambda)):=|v^0_\lambda|$, or write $L(\lambda)=L(\lambda)^\iota$ for $\iota=|v^0_\lambda|$.
Meanwhile, we have the following parity Lemma.

\begin{lem} \label{pty lem} (\textsc{Parity Lemma}) Keep the notations as above. The following statements hold.
\begin{itemize}	
		\item[(1)] If $L(\mu)$ and $L(\nu)$ are in the same block, then $|v^0_\mu|-|v^0_\nu|=\overline{\ell(\mu-\nu)}$ where $\overline{\ell(\mu-\nu)}\in\bbz_2$ denotes the parity of $\ell(\mu-\nu)$.
		\item[(2)] 	For any $\lambda\in \Lambda^+$, and different parities $\iota_1, \iota_2\in \bbz_2$,
		$L(\lambda)^{\iota_1}$ and $L(\lambda)^{\iota_2}$ do not lie in the same block.
\end{itemize}
	\end{lem}
\begin{proof} By arguments similar to the proof of Lemma \ref{dpt lem},  the lemma is readily justified.
\end{proof}

\subsection{Blocks of $\comi$ for $\ggg=W(n)$ or $\bar S(n)$} \label{block WS sec}
In this subsection, we focus our concern on $W(n)$ and $\bar S(n)$. Recall the notation $\Xi=\sum_{i=1}^n\epsilon_i$.
Let
$\lambda=\lambda_{1}\epsilon_{1}+
\lambda_{2}\epsilon_{2}+\cdots+\lambda_{n}\epsilon_{n}$ {\color{red}be} an element of $\Lambda^+$.
Write $\lambda$ in the following form
\begin{align}\label{WS weight form}
\lambda&=\lambda_{n}\Xi+(\lambda_{1}-\lambda_{n})\epsilon_{1}+\cdots+(\lambda_{n-1}-\lambda_{n})\epsilon_{n-1}\cr
&=\lambda_n(\epsilon_{1}+\epsilon_{2}+\cdots+\epsilon_{n})+\alpha,
\mbox{~ with~}\lambda_n\in\bbc,\; \alpha\in Q^+:=(\sum_{i=1}^{n-1}\bbz_{\geq 0}\epsilon_i)\cap \Lambda^+. \cr
\end{align}
Denote by $Q$ the root lattice of $\ggg$ with respect to the root system $\Phi(\ggg)$ (see \S\ref{root sys}). Then set
\begin{align*}
\comi(c,i)=\{L(\lambda)\in\bfe \mid \lambda\in (c+\bbz)\Xi+Q,\; &\dpt(L(c\Xi))=i;\\
&\dpt(L(\lambda))=i+\ell(\lambda-c\Xi)
\}.
\end{align*}
It further splits into
$$\comi(c,i)=\comi(c,\bz,i)\cup \comi(c,\bo,i)
$$
where
\begin{align*}
\comi(c,\iota,i)=\{L(\lambda)\in\comi(c,i)\mid\; &\pty(L(c\Xi))=\iota;\\ &\pty(L(\lambda))=\iota+\overline{\ell(\lambda-c\Xi)} \}
\end{align*}
for $\iota\in \bbz_2$. Here $\overline{\ell(\lambda-c\Xi)}\in \bbz_2$ denotes the parity of $\ell(\lambda-c\Xi)$.

Let $\mu=\mu_{1}\epsilon_{1}+\mu_{2}\epsilon_{2}+\cdots+\mu_{n}\epsilon_{n}$ be an element of $\bar\hhh^*$ for $\ggg=X(n)$ with $X\in \{W,\bar S\}$.  We define the height of $\mu$, which is denoted by $\hht(\mu)$, as $\hht(\mu)=\sum_{i=1}^{n}\mu_{i}$.

\begin{thm}\label{block thm WS}
	Assume that  $\ggg=X(n)$ with $X\in \{W,\bar S\}$. The complete set of all different blocks in $\comi$ is described as follows
	$$\{\comi(c,\iota, i)\mid (c,\iota,i)\in \mathbb{C}/\mathbb{Z}\times \bbz_2\times \bbz \}.$$

\end{thm}
\begin{proof} 
Firstly, we will prove that simple objects belonging to $\comi(c,\iota, i)$ are indeed in the same block.

 For any given $L(\lambda)\in \comi(c,\iota,i)$, naturally $\lambda\in \Lambda^+$. By (\ref{WS weight form}) and Corollary \ref{height one lem}(2), we can write $\lambda=c\Xi+\alpha$ for some $\alpha\in Q^+$ without loss of generality.
 We will prove that $L(\lambda)$ lies in the block where $L(c\Xi)$ lies by induction on $\hht(\alpha)$.

  When $\hht(\alpha)=0$, then $\alpha=0$ because $\alpha\in Q^+.$ So the conclusion is true.

  When $\hht(\alpha)>0$, suppose that the conclusion has been true for the situation of being less than $\hht(\alpha)$.
  Assume $\alpha=\sum_{i=1}^{n-1}a_i\epsilon_i$ with $a_i\in\bbz_{\geq 0}.$ Then there exists $1\leq t\leq n-1$ such that $a_t>a_{t+1}$, where we make convention that $a_n=0$. Take $\beta=
   \sum_{k=t+1}^{n}\epsilon_{k}$. Consider $\lambda':=\lambda+\beta$ which lies in $\Lambda^+$.
    By Corollary \ref{height one lem}(2), $L(\lambda')$ and  $L(\lambda'-\beta)=L(\lambda)$ lie in the same block. Note that  $\alpha+\beta=\sum_{i=1}^{t} a_i\epsilon_i+ \sum_{k=t+1}^{n} (a_k+1)\epsilon_i=\Xi+\gamma$, where $\gamma=\alpha-\sum_{i=1}^t\epsilon_i\in Q^+$. So we have $\lambda'=(c+1)\Xi+\gamma$.
 By Corollary  \ref{height one lem}(2), $L(\lambda')$ and $L(c\Xi+\gamma)$
   lie in the same block.
Furthermore, $\hht(\gamma)<\hht(\alpha)$.
Thus, by inductive hypothesis, $L(c\Xi+\gamma)$ and $L(c\Xi)$ already lie in the same block.
Hence, $L(\lambda)$ and $L(c\Xi)$ finally turn out to lie in the same block.

Secondly, for any $L(\lambda)\in\bfe$, we see that $L(\lambda)\in \comi(c,\iota,i)$ for some $c\in\mathbb{C}, \iota\in\mathbb{Z}_2, i\in\mathbb{Z}$ by (\ref{WS weight form}). Moreover, we will prove that if a simple object $L(\mu)$ lies in the block where  $L(c\Xi)^{\gamma}_i$ lies, then $L(\mu)$ must lie in $\comi(c,\iota,i)$.  For this, we only need to note the following two facts:

(i) For any indecomposable projective module $P(\lambda)$ with $\lambda=c\Xi+\alpha$ for $c\in\bbc$ and $\alpha\in Q$,  and its  composition factor $L(\mu)$, by Remark \ref{proj remark}  we have $\mu-\lambda\in \sum_{i=1}^n\bbz \epsilon_i$, thereby $\mu\in c\Xi+Q$.

(ii) If $L(\mu')$ is a composition factor of $P(\lambda')$ and $\mu'\in c\Xi+Q$. By Theorem \ref{projective thm}, $P(\lambda')$ is a direct summand of $I(\lambda')$ and $L(\mu')$ is a composition factor of $I(\lambda')$. Hence $\lambda'\in c\Xi+Q$.

Thus, by the definition of blocks and taking Depth Lemma and Parity Lemma into account, we have proven that if a simple object $L(\mu)$ lies in the block where  $L(c\Xi)^{\iota}_i$ lies, then $L(\mu)$ must lie in $\comi(c,\iota,i)$.
The proof  is completed.
\end{proof}

\vskip-5pt

\subsection{Blocks of $\comi$ for $\ggg=\bar H(n)$}\label{block H sec}
In this case, $n=2r$ or $n=2r+1$.
Recall that the notation ${\aleph_r}\in\{0,1\}$ satisfies ${\aleph_r}\equiv r\mod 2$ for $H(2r)$.  And recall  that there is a standard dual $\delta$ of $\sfd$ in $\bar\hhh=\hhh+\bbc\sfd$.
Let $\lambda=\lambda_{1}\epsilon_{1}+\lambda_{2}\epsilon_{2}+\cdots+\lambda_{r}\epsilon_{r}+c\delta$ be an element of $\bar\hhh^*,$ We define the height of $\lambda$, which is denoted by $\hht(\lambda)$, as $\hht(\lambda)=\sum_{i=1}^{r}\lambda_{i}.$

Recall the notations  $\Theta_{D,\aleph_r}=\epsilon_1+\cdots+\epsilon_{r-1}+\epsilon_r+{\aleph_r}\delta$ for $\bar H(2r)$, and $\Theta_B=\epsilon_1+\cdots+\epsilon_{r-1}+\epsilon_r$ for $\bar H(2r+1)$.
For $\lambda\in \Lambda^+\subset \bar\hhh^*$, it can be further presented as
\begin{align}\label{H weight form-1}
\lambda
=\begin{cases}\lambda_r\Theta_{D,{\aleph_r}}+c\delta+\sum_{i=1}^{r-1} (\lambda_i-\lambda_r)\epsilon_i, &\mbox{ if }n=2r;\\
\lambda_r\Theta_B+c\delta+\sum_{i=1}^{r-1} (\lambda_i-\lambda_r)\epsilon_i,&\mbox{ if }n=2r+1
\end{cases}
\end{align}
satisfying that $\sum_{i=1}^{r-1} (\lambda_i-\lambda_r)\epsilon_i\in \sum_{i=1}^{r-1}\bbz_{\geq 0}\epsilon_i\cap \Lambda^+$  for both $\bar H(2r)$ and $\bar H(2r+1)$.
So for $\lambda\in \Lambda^+$, by (\ref{H weight form-1}) we can write
\begin{align}\label{H weight form}
\lambda=c\delta+d\Theta+\alpha  \mbox{~with~} c,d\in\bbc,\;\alpha=\gamma+\hht(\gamma)\delta
\end{align}
where $\gamma\in Q^+:=\sum_{i=1}^{r-1}\bbz_{\geq 0}\epsilon_i\cap \Lambda^+$ and $\gamma+\hht(\gamma)\delta\in Q:=\bbz\Phi(\ggg)$, here $\bbz\Phi(\ggg)$ denotes the root lattice of $\ggg$.

In the following, {\sl{we will simply write $\Theta=\Theta_B$ or $\Theta_{D,{\aleph_r}}$ according to the situation $n=2r+1$ or $n=2r$ respectively.}}

\begin{lem}\label{exp uni} (\textsc{Independence Lemma})\label{indep lem} Let $\ggg=\bar H(n)$ and  $\lambda\in \Lambda^+$. Then the expression  of $\lambda$ in (\ref{H weight form}) is unique.
\end{lem}
\begin{proof} Suppose $\lambda=c_i\delta+d_i\Theta+\alpha_i$, $i=1,2$. We need to prove that $c_1=c_2$, $d_1=d_2$ and $\alpha_1=\alpha_2$. We know $d_1=d_2=\lambda_r$. So we have $\lambda-\lambda_r\Theta=c_1\delta+\alpha_1=c_2\delta+\alpha_2$,
Hence $(c_1-c_2)\delta+(\alpha_1-\alpha_2)=0$. According to (\ref{H weight form}), assume that $\alpha_i=\gamma_i+\hht(\gamma_i)\delta$, $i=1,2$ with $\gamma_i\in Q^+$. Then $(c_1+\hht(\gamma_1)-c_2-\hht(\gamma_2))\delta=\gamma_2-\gamma_1$. Since $\gamma_1,\gamma_2\in Q^+,$ $\gamma_2-\gamma_1=0,$
we have $\gamma_1=\gamma_2$. Consequently, $\alpha_1=\alpha_2$ and $c_1=c_2$.
\end{proof}

\subsubsection{Case $H(2r+1)$}
In this case $\tilde \delta=\delta$ and $\Theta=\Theta_B=\sum_{i=1}^r\epsilon_i$.
Recall that $Q=\bbz\Phi(\ggg)$ is the root lattice of $\ggg$. By Lemma \ref{indep lem} it does make sense to set
\begin{align*}
\comi(c,d,i)=\{L(\lambda) \mid \lambda\in  (c+\bbz)\delta+(d+\bbz)\Theta+Q, \; &\dpt(L(c\delta+d\Theta))=i;\\
&\dpt(L(\lambda))=i+\ell(\lambda-c\delta-d\Theta)
\}.
\end{align*}
It further splits into
$$\comi(c,d,i)=\comi(c,d,\bz,i)\cup \comi(c,d,\bo,i)
$$
where
\begin{align*}
\comi(c,d,\iota,i)=\{L(\lambda)\in\comi(c,d,i)\mid \; &\pty(L(c\delta+d\Theta))=\iota;\\ &\pty(L(\lambda))=\iota+\overline{\ell(\lambda-c\delta-d\Theta)} \}
\end{align*}
for $\iota\in \bbz_2$. Here $\overline{\ell(\lambda-c\delta-d\Theta)}\in \bbz_2$ denotes the parity of $\ell(\lambda-c\delta-d\Theta)$.

\begin{thm}\label{block thm H odd} Assume $\ggg=\bar H(2r+1)$. The complete set of all different blocks in $\comi$ is listed as follows
	$$\{\comi(c,d,\iota, i)\mid (c,d,\iota,i)\in (\mathbb{C}/\mathbb{Z})^2\times \bbz_2\times \bbz \}.$$
\end{thm}
\begin{proof} We will take the same strategy as the proof of Theorem \ref{block thm WS}.
	 For any given $L(\lambda)\in \comi(c,d,\iota,i)$, we first prove that $L(\lambda)$ lies in the block where $L(c\delta+d\Theta)$ lies.  By Corollary \ref{height one lem} and Lemma \ref{exp uni}, we can write  $\lambda=c\delta+d\Theta+\alpha$ for some $\alpha=\gamma+\hht(\gamma)\delta\in Q$ with $\gamma=\sum_{i=1}^{r-1} a_i\epsilon_i\in Q^+$. By definition, we know  $\hht(\alpha)=\hht(\gamma)\geq0$.
Thus, we will  accomplish the proof by taking induction on $\hht(\alpha)$. When $\hht(\alpha)=0$, then $\alpha=0$ because $\gamma=\sum_{i=1}^{r-1} a_i\epsilon_i\in Q^+$. So the statement holds.

Suppose $\hht(\alpha)>0$, and suppose that the conclusion has been true for the situation of being less than $\hht(\alpha)$.  In this case, we can write
$\gamma=\sum_{i=1}^{r-1} a_i\epsilon_i$ with $a_i\in \bbz_{\geq 0}$ such that $a_1\geq a_2\geq \cdots \geq a_{r-1} \geq a_r=0.$
 Because  $\hht(\alpha)>0$ and $a_r=0,$ there exists at least one $t\in\{1,\ldots,r-1\}$  satisfying  $a_t>a_{t+1}$. Take $\beta=\sum_{i=1}^t\epsilon_i-(n-t)\delta$ and $\lambda'=\lambda-\beta.$  Because $a_t>a_{t+1}$, $\lambda'\in\Lambda^+.$
By Corollary \ref{height one lem}(4),  $L(\lambda')$ and $L(\lambda'+\beta)=L(\lambda)$ share the same block.
On the other hand, $\lambda'= c\delta+d\Theta+(\alpha-\beta)$ with $\alpha-\beta= (\gamma-\sum_{i=1}^t\epsilon_i)+(\hht(\gamma)+n-t)\delta$. Obviously, $\gamma-\sum_{i=1}^t\epsilon_i\in Q^+$ and
$\hht(\alpha-\beta)=\hht(\gamma)-t<\hht(\alpha)$.
Thus, $L(\lambda')$ and $L(c\delta+d\Theta)$ lie in the same block by inductive hypothesis. Hence, $L(\lambda)$ and $L(c\delta+d\Theta)$ finally lie in the same block.

Conversely, we have the following clear observation.

(i) Let  $P(\lambda)$ be any indecomposable projective module where  $\lambda=c\delta+d\Theta+\alpha$ with $c,d\in\bbc$ and $\alpha\in Q$.
By the construction of $P(\lambda)$ (Remark \ref{proj remark}(1)), all weights  of $P(\lambda)$ are in  $\lambda+\bbz\delta+Q$.
So if $L(\mu)$ is a composition factor of $P(\lambda)$, then $\mu\in (c+\bbz)\delta+(d+\bbz)\Theta+Q$.

(ii) If $L(\mu')$ is a composition factor of $P(\lambda')$ and $\mu'\in (c+\bbz)\delta+(d+\bbz)\Theta+Q$, then by Remark \ref{proj remark} again, we have $\lambda'\in (c+\bbz)\delta+(d+\bbz)\Theta+Q$.

Thus, by the definition of blocks and taking Depth Lemma and Parity Lemma into account, we have proven that if a simple object $L(\mu)$ lies in the block where  $L(c\delta+d\Theta)^{\iota}_i$ lies, then $L(\mu)$ must lie in $\comi(c,d,\iota,i)$.
The proof  is completed.
\end{proof}
\subsubsection{Case $\bar H(2r)$}\label{2r subsub}
In this case, $\Theta=\Theta_{D,\aleph_r}$.
Recall that $\ggg$ admits the root lattice $Q$ (see \S\ref{block WS sec}). In contrast with the block structure of $H(2r+1)$, there is a crucial difference in the case of $H(2r)$, that is, $L(\lambda)$ and $L(\lambda+\delta)$ do not lie in the same block. The following lemma is a clue to it.

\begin{lem}\label{no delta} Let $\ggg=\bar H(2r)$. Then the following statements hold.
\begin{itemize}

\item[(1)] The root lattice $Q$ contains $\pm2\delta$, but does not contain $\pm\delta$.

\item[(2)]
	If $L(\mu)$ and $L(\nu)$ are in the same block, then $(\mu-\nu)\in Q.$ In particular,
	$L(\lambda)$ and $L(\lambda\pm \delta)$ can not belong to the same block.

\item[(3)] Let $\beta=\beta_1\epsilon_{1}+\cdots+\beta_r\epsilon_r$ be an element of $Q\cap \Lambda^+.$ Then there exist $m\in\bbz$ and $\gamma\in Q^+$ such that
$$\beta=2m\delta+\beta_r\Theta+\gamma+\hht(\gamma)\delta.$$
\end{itemize}
\end{lem}

\begin{proof} (1) Recall that the root system is
\begin{align*}
\Phi=\{ \pm\epsilon_{i_1}\pm\cdots\pm \epsilon_{i_k} +l\delta\mid \;&1\leq i_1<i_2 <\cdots< i_k \leq r;\\
 &k-2<l<n-2,  l-k \in 2\bbz\}.
\end{align*}
 It is easily seen that $\pm\delta$ does not appear in the $\bbz$-linear combinations of roots.

(2) Consider $I(\lambda)$. Any of its weights is of the form $\lambda+\alpha$ for some $\alpha\in Q$.
Because $P(\lambda)$ is a direct summand of $I(\lambda)$, if $L(\mu)$ and $L(\nu)$ are two composition factors of $P(\lambda),$ then $(\mu-\nu)\in Q.$ The statement (2) follows due to the definition of blocks and the statement (1).
	
(3) Since $\epsilon_{1}-\delta,\epsilon_{2}-\delta,\cdots,\epsilon_{r}-\delta,2\delta$ belong to $Q,$ we can check that $\epsilon_1+\cdots+\epsilon_{r-1}+\epsilon_r+{\aleph_r}\delta\in Q.$
Hence,
\begin{align}\label{qq}
&\beta-\beta_r\Theta\nonumber\\
=&(\beta_1-\beta_r)\epsilon_{1}+(\beta_2-\beta_r)\epsilon_{2}+\cdots+(\beta_{r-1}-\beta_r)\epsilon_{r-1}+(\beta_1+\cdots+\beta_{r-1}
-(r-1)\beta_r)\delta\nonumber\\
&+(-\beta_1-\cdots-\beta_{r-1}
+(r-1)\beta_r-\beta_r\aleph_r)\delta\in Q\cap \Lambda^+.
    \end{align}
Write $\gamma:=(\beta_1-\beta_r)\epsilon_{1}+(\beta_2-\beta_r)\epsilon_{2}+\cdots+(\beta_{r-1}-\beta_r)\epsilon_{r-1}\in Q^+$ and $\gamma_i:=\beta_i-\beta_r$. Since
\begin{equation*}
\beta=\beta_1\epsilon_1+\cdots +\beta_r\epsilon_r\in Q,
\end{equation*}
by (1) we see that $\beta_1+\cdots+\beta_{r}$ is even.  Then there exists $m\in\bbz$ such that  $$-\beta_1-\cdots-\beta_{r-1}
+(r-1)\beta_r-\beta_r\aleph_r=-\beta_1-\cdots-\beta_{r}+\beta_r(r-\aleph_r)=2m. $$
By (\ref{qq}), we have
$$\beta-\beta_r\Theta=\gamma+\hht(\gamma)\delta+2m\delta$$
The statement (3) follows.
\end{proof}

By Lemma \ref{indep lem} it does make sense to set
\begin{align*}
\comi(c,d,i)=\{L(\lambda) \mid \;&\lambda\in  (c+2\bbz)\delta+(d+\bbz)\Theta+Q, \\ &\dpt(L(c\delta+d\Theta))=i;\\
&\dpt(L(\lambda))=i+\ell(\lambda-c\delta-d\Theta)
\}.
\end{align*}
It further splits into
$$\comi(c,d,i)=\comi(c,d,\bz,i)\cup \comi(c,d,\bo,i)
$$
where
\begin{align*}
\comi(c,d,\iota,i)=\{L(\lambda)\in\comi(c,d,i)\mid \; &\pty(L(c\delta+d\Theta))=\iota;\\ &\pty(L(\lambda))=\iota+\overline{\ell(\lambda-c\delta-d\Theta)} \}
\end{align*}
for $\iota\in \bbz_2$. Here $\overline{\ell(\lambda-c\delta-d\Theta)}\in \bbz_2$ denotes the parity of $\ell(\lambda-c\delta-d\Theta)$.

\begin{thm}\label{block thm H even} Assume $\ggg=\bar H(2r)$. The complete set of all different blocks in $\comi$ is listed as follows
	$$\{\comi(c,d,\iota, i)\mid (c,d,\iota,i)\in  \bbc\slash 2\bbz \times \bbc\slash \bbz\times\bbz_2\times \bbz \}.$$
\end{thm}
\begin{proof} 
 For any given $L(\lambda)\in \comi(c,d,\iota,i)$, we first prove that $L(\lambda)$ lies in the block where $L(c\delta+d\Theta)$ lies.
 Assume that $\lambda=(c+2m_1)\delta+(d+m_2)\Theta+\beta, \beta\in Q,$ is an element of $\Lambda^+.$ By the expression of $\Theta,$ we can deduce $\beta\in\Lambda^+\cap Q.$ Assume $\beta=\sum_{i=1}^{r}\beta_i\epsilon_{i},$ then $\beta_i\in\bbz.$
By Lemma \ref{no delta}(3),  there exist $m\in\bbz$ and $\gamma\in Q^+$ such that
$\beta=2m\delta+\beta_r\Theta+\gamma+\hht(\gamma)\delta.$ So $\lambda=(c+2m_1+2m)\delta+(d+m_2+\beta_r)\Theta+\gamma+\hht(\gamma)\delta.$
By Corollary \ref{height one lem}(3) and (5),
we can write  $\lambda=c\delta+d\Theta+\alpha$ directly
for some $\alpha=\gamma+\hht(\gamma)\delta\in Q$ with $\gamma=\sum_{i=1}^{r-1} \gamma_i\epsilon_i\in Q^+$ without loss of generality.
Thus, we can  accomplish the proof similarly by taking induction on $\hht(\alpha)$.
By taking the same arguments as in the proof of Theorem \ref{block thm H odd} (here we omit the details) we can prove that $L(\lambda)$ and $L(c\delta+d\Theta)$ lie in the same block.
Readers need only to notice that $n=2r$ is even now.

What remains is to prove conversely  that if a simple object $L(\mu)$ lies in the block where  $L(c\delta+d\Theta)^{\iota}_i$ lies, then $L(\mu)$ must lie in $\comi(c,d,\iota,i)$.  For this, it suffices to observe the following facts. 

(i) Let  $P(\lambda)$ be any indecomposable projective module where $\lambda=c\delta+d\Theta+\alpha$ with $c,d\in\bbc$ and $\alpha\in Q$. We claim that any composition factor of $I(\lambda)$, say $L(\mu)$, must belong to the set $\lambda+2\bbz\delta+\bbz\Theta+Q$.
Recall that $I(\lambda)$ admits a $\Delta$-flag with subquotients $\Delta(\tau)$ for $\tau \in \Upsilon(\lambda)$. So $L(\mu)$ must be a composition factor of some $\Delta(\tau)$. By the definition of $\Upsilon(\lambda)$ we can assume $\tau=\lambda-\gamma$ with  $\gamma=\sum_{j=1}^k\pm\epsilon_{i_j}+m\delta,$ where $i_j\in\{1,2,...,r\}$ satisfying $m\geq k$ and $m-k\in 2\bbz$. 
So $\tau=\lambda-(\sum_{j=1}^k\pm\epsilon_{i_j}+m\delta)=\lambda-(\sum_{j=1}^k\pm\epsilon_{i_j}+k\delta+(m-k)\delta).$
Thus,  $\tau\in \lambda+2\bbz\delta+Q$.

Next we investigate $L(\mu)$ from $\Delta(\tau)$. Note that by the definition of standard modules, all weights of $\Delta(\tau)$ must lie in $\tau +\bbz_{\geq 0}\Phi(\ggg_{\geq 1})$ where $\Phi(\ggg_{\geq 1})$ meas the root system of $\ggg_{\geq 1}$. So $\mu$ lies in $\tau+2\bbz\delta+\bbz\Theta+Q$.  The claim is true. So the claim is naturally true for $P(\lambda)$.

(ii) If $L(\mu')$ is a composition factor of $P(\lambda')$ and $\mu'\in c\delta+d\Theta+Q$, then $L(\mu')$ is naturally a composition factor of $I(\lambda')$. By the same reason as in (i)  we have $\lambda'\in (c+ 2\bbz)\delta+(d+\bbz)\Theta+Q$.
Thus, by the definition of blocks and taking Depth Lemma and Parity Lemma into account, 
we have that $L(\mu)$ indeed lies in $\comi(c,d,\iota,i)$.
Summing up, we finish the proof.
\end{proof}

\begin{rem} {\rm(1)} According to the proof, it is not hard to see that any irreducible module sharing the same block as  $L(c\delta+d\Theta+\alpha)$ must be of the form $L(\mu)$ with $\mu\in c\delta+d\Theta+Q$.

{\rm(2)} As a direct consequence of the above theorem, we know that $L(\lambda)$ and $L(\lambda\pm \delta)$ do not lie in the same block as mentioned at the beginning of the sub-subsection \S\ref{2r subsub}.

{\rm(3)} On the basis of Proposition \ref{I prop}, one easily knows that  Theorems \ref{block thm H odd} and \ref{block thm H even} are valid in the case  when $\ggg=\ochn$ ($n=2r$ or $n=2r+1$).

\end{rem}

\subsection{Application to the category of finite-generated  modules over $\ggg$}

We are going to consider  blocks of the category of finite-generated  modules over $\ggg$. Denote this category by $\gmf$, whose objects are by definition,  finite-generated modules, and whose morphisms are required to be even.

Recall that the forgetful functor $\cf$ (see Remark \ref{def rem}(5))  makes $\comi$ into the $U(\ggg)$-module category  $\cf(\comi)$ whose objects are only subjected to  weighted-structure, and locally-finiteness over $U(\sfp)$. This is to say,  all objects in $\cf(\comi)$ inherit all structures in $\comi$ except $\bbz$-gradation. Then the isomorphism classes of simple objects both in $\cf(\comi)$ and in  $\cf(\comif)$ are parameterized by $\Lambda^+$ respectively, still denoted by $\{L(\lambda)\mid \lambda\in \Lambda^+\}$.

\begin{lem}\label{fin ind proj lem}
\begin{itemize}
\item[(1)] Any object of $\cf(\comif)$ can be naturally regarded as an object in $\comi$. Any morphism in $\cf(\comif)$ can be lifted to  $\comi$.

\item[(2)] For any $P(\lambda)$ in $\comi$, $\cf(P(\lambda))$ is still indecomposable and  projective in $\cf(\comif)$.
    \end{itemize}
\end{lem}

\begin{proof} (1) For any given object $M$ in $\cf(\comif)$ and any given integer $d$,  we will show that $M$ can be endowed with a $\bbz$-gradation related to $d$. By the same arguments as  (\ref{standard fil 0}) in Theorem \ref{projective thm 2}, we have that $M$ admits a filtration of finite length
\begin{align}
M=M^{(0)}\supset M^{(1)}\supset M^{(2)} \supset\cdots\supset M^{(t-1)}\supset M^{(t)}=0
\end{align}
 such that $M^{(i-1)}\slash M^{(i)}$ is isomorphic to a non-zero quotient of  $\cf(\Delta(\lambda_i))$ associated with  some irreducible $U(\sfp)$-module $L^0(\lambda_i)=U(\nnn^-)v^0_{\lambda_i}$ with $\lambda_i\in \Lambda^+$, $i=1,\cdots, t$, with $t$ being the standard length $l(M)$.
 If $l(M)=1$,  $M=U(\ggg)v^0_{\lambda_1}$ which is easily endowed with a $\bbz$-gradation, provided that $L^0(\lambda_1)$ is predefined to be of grading $d$. In general, we can define such a gradation on $M$ by induction on $l(M)$. Suppose that $t=l(M)>1$, and the gradation is defined already for less than $t$. Especially,  $M^{(1)}$ is supposed to be already endowed with a $\bbz$-gradation associated with $d$, hence all gradations of $v^0_{\lambda_i}$ ($i=2,...,t$) are actually predefined, denoted by $g_i$. For any $m\in M$,  $m\equiv m_1 \mod M^{(1)}$ for $m_1\in U(\ggg)m_{\lambda_1}$ with $m_{\lambda_1}$ being a pre-image of $v^0_{\lambda_1}$. Then we can define the gradation of $m_{\lambda_1}$ to be $g_1$ such that $g_1$ is compatible with $\lambda_i$ for $i=2,...,t$, this is to say, if $\lambda_1-\lambda_i\in Q$, then $g_1=g_i+\ell(\lambda_1-\lambda_i)$. Thus, $m_1$, thereby $M$ is endowed with a $\bbz$-gradation. We have proven the first part of (1).

Suppose that $\phi: M\rightarrow N$ is a homomorphism  in $\cf(\comif)$ . In the way just mentioned above, $M$ can be endowed with a $\bbz$-gradation, thereby we can naturally endow a $\bbz$-gradation on $\phi(M)$ such that $\phi$ is lifted to be a morphism in $\comi$.  Hence we have proven the second part of (1).

(2) Let $P(\lambda)$ be the projective cover of $L(\lambda)\in \bfe$ with $\dpt(L(\lambda))=d$.
Due to Remark \ref{proj remark}(1), we can assume that $P(\lambda)=\sum_{g\in\bbz}P(\lambda)_g$ is generated by some $\lambda$-weighted vector $v_0$ and the grading of $v_0$ is $d$. For any given surjective morphism $\phi:M\rightarrow N$ in $\cf(\comif)$,  and a nonzero morphism $\psi: \cf(P(\lambda))\rightarrow N$ in $\comif$, we want to prove that there is a lift $\bar\psi: \cf(P(\lambda))\rightarrow M$.
We begin with the definition of grading shift functor. Let $L$ be a $\zz$-graded module belonging to $\comi$ and $d\in\zz$. Define a grading shift functor $[d]:L\mapsto L[d],$ such that as a vector space, $L[d]=L$, but the $\zz$-grading of $L[d]$ is changed through $L[d]_i=L_{i-d}.$ We can check that $P(\lambda)[d_0]$ is the projective cover of $L(\lambda)[d_0]\in \bfe.$
By (1), the surjective morphism $\phi:M\rightarrow N$ can be lifted to a surjective morphism in $\comi$ which becomes $\dot\phi: \dot M\rightarrow \dot N$. We suppose that $\psi(v_0)$ has a gradation $d_0$ in $\dot N$. Then by a suitable shift, we can re-endow a $\bbz$-gradation on $\cf(P(\lambda))$ such that $v_0$ is of gradation $d_0$,  getting a new object $\dot P(\lambda)$ in $\comi$. By the arguments in the previous paragraph, we see that  $\dot P(\lambda)$ is still indecomposable and projective in $\comi$. So we really have a morphism $\dot \psi: \dot P(\lambda)\rightarrow \dot N$ in $\comi$. The projectiveness of $\dot P(\lambda)$ entails that there exists a lift $\bar{\dot \psi}:\dot P(\lambda)\rightarrow \dot M$ of $\dot\psi$. After applying the forgetful functor $\cf$, we get the desired lift $\bar \psi$ of $\psi$. The proof is completed.
\end{proof}

By the above lemma, we can similarly define  blocks in $\cf(\comi)$ as below.

Set $\cf(\comi(c)):=\{\cf(L(\lambda)) \mid \lambda\in  c\Xi+Q\}$ when $\ggg=W(n)$ or $\bar S(n)$. Then we have in the same sense as in \S\ref{block WS sec}, that
\begin{align*}
\cf(\comi(c))=\cf(\comi(c,\bz))\cup \cf(\comi(c,\bo)).
\end{align*}
Similarly, set $\cf(\comi(c,d)):=\{\cf(L(\lambda)) \mid \lambda\in  c\delta+d\Theta+Q\}$ when $\ggg=\bar H(n)$. We have in the same sense as in \S\ref{block H sec}, that
\begin{align*}
\cf(\comi(c,d))=\cf(\comi(c,d,\bz))\cup \cf(\comi(c,d,\bo)).
\end{align*}
Then we have the following direct consequence by Theorems \ref{block thm WS}, \ref{block thm H odd} and \ref{block thm H even}.

\begin{coro} The complete classification  of all different blocks in $\cf(\comi)$ is listed as follows:
	
{\rm(1)} If $\ggg=W(n)$, or $\bar S(n)$, then it is
$$\{\cf(\comi(c,\gamma))\mid (c,\gamma)\in (\mathbb{C}/\mathbb{Z})\times \bbz_2 \}.$$

{\rm(2)} If $\ggg=\bar H(n)$, then it is
\begin{align*}
&\{\cf(\comi(c,d,\gamma))\mid (c,d,\gamma)\in (\mathbb{C}/\mathbb{Z})^2\times
\bbz_2 \} \mbox{ if }n=2r+1; \mbox{ and }\cr
&\{\cf(\comi(c,d,\gamma))\mid (c,d,\gamma)\in \mathbb{C}/2\mathbb{Z}\times \mathbb{C}/\mathbb{Z}\times \bbz_2 \} \mbox{ if }n=2r.
\end{align*}
\end{coro}

Obviously, $\gmf$ is a full subcategory of $\cf(\comi)$. We can introduce  blocks of $\gmf$ as follows.

	\begin{defn} A block $\mathbf{B}$ of $\gmf$ is a subcategory of $\gmf$, satisfying that for any $B\in \mathbf{B}$,  all its composition factors lie in the same block of $\cf(\comi)$.
\end{defn}

We finally obtain the block theorem for $\gmf$ as follows.
\begin{thm} \label{finite block thm} The following statements hold.
\begin{itemize}
\item[(1)] For $\ggg=W(n)$ or $\bar S(n)$,
$$ \gmf=\bigoplus_{(c,\gamma)\in \mathbb{C}/\mathbb{Z}\times \bbz_2} \mathbf{B}(c,\gamma) $$

\item[(2)] For $\ggg=\bar H(n)$,
$$ \gmf=\begin{cases}\bigoplus_{(c,d,\gamma)\in (\mathbb{C}/\mathbb{Z})^2\times \bbz_2} \mathbf{B}(c,d,\gamma) \mbox{ when } n=2r+1;\cr
\bigoplus_{(c,d,\gamma)\in \mathbb{C}/2\mathbb{Z}\times \mathbb{C}/\mathbb{Z}\times \bbz_2} \mathbf{B}(c,d,\gamma) \mbox{ when } n=2r.
\end{cases} $$
\end{itemize}
\end{thm}

\begin{rem} \label{covers} {\rm(1)} In our setup, Theorem \ref{finite block thm} essentially covers the main result of \cite{Shom02} on  blocks of the category of finite-dimensional modules over $W(n)$.


{\rm(2)} By the same arguments as in \cite{Shom02}, one can show that all blocks of $\comi$ are wild.
\end{rem}

\section{Tilting modules and character formulas}\label{Costandard M}
Keep the same notations as in Sections \ref{yubei} and \ref{cat o}. In particular, $\delta$ is the linear dual of $\sfd$ in ${\bar\hhh}^*$ when $\ggg=\bar{H}(n)$ (See \S\ref{toral extension}). 

\subsection{} Thanks to  Lemma \ref{pre proj}, we can apply the arguments in \cite{Brun04} to our category $\comi$. We first recall some properties for standard and co-standard modules.

\begin{lem} \label{basic lem O}
Keep the assumption as above. The following results hold in the category $\comi$:
\begin{itemize}

\item[(1)] The category $\comi$ has enough injective objects.
\item[(2)] Assume that $\Delta(\lambda)$ has depth $d$. Then $\Delta(\lambda)$ is the projective cover of $L(\lambda)$ in  $\comid$.
\item[(3)]$\dim\Hom_{\comi}(\Delta(\lambda),\nabla(\mu))=
\delta_{\lambda,\mu}$ for $\lambda,\mu\in\Lambda^+.$
\item[(4)] $\mathrm{Ext}_{\comi}^1(\Delta(\lambda),\nabla(\mu))=0$ for $\lambda,\mu\in\Lambda^+$.
\end{itemize}
\end{lem}

\begin{proof}  For (1), readers can refer to \cite[Lemma 2.1]{Brun04}. For (2),(3),(4), readers can refer to \cite[Lemma 3.6]{Brun04}.
\end{proof}

\subsection{Tilting modules} Thanks to Lemma \ref{semi-inf}, the category $\comi$
is associated with a semi-infinite character of $\ggg$. So we can apply Soergel's tilting module theory to our category $\comi$. The following lemma  asserts the existence of the so-called indecomposable tilting modules $T(\lambda)$ for $\lambda\in\bfe$.

\begin{lem} \label{tilt lem}$($\cite[Theorem 5.2]{Soe98} and \cite[Theorem 5.1]{Brun04}$)$\label{tilt}
For any given $L^0(\lambda)=L^0(\lambda)_d$ $((\lambda,d)\in\mathbf{E}=\Lambda^+\times \bbz)$,  there exists a unique up to isomorphism  indecomposable object $T(\lambda)\in\comi$ such that
\begin{itemize}
\item[(1)] $\mathrm{Ext}_{\comi}^{1}(\Delta(\mu),T(\lambda))=0$ for any $\mu\in\bfe.$
\item[(2)] $T(\lambda)$ admits a $\Delta$-flag starting with $\Delta(\lambda)$ at the bottom.
\end{itemize}
\end{lem}

\begin{defn}
An  object $T$ in $\comi$ is called a tilting module  if it  satisfies (1) and (2) in Lemma \ref{tilt} as $T(\lambda)$ does. In particular, the indecomposable tilting object $T(\lambda)$ is called the indecomposable tilting module associated with $\lambda\in\bfe$.
\end{defn}

In the following,  we will  investigate the flags of standard
modules for indecomposable tilting modules, by means of Soergel reciprocity and the Kac-module realizations of co-standard modules. 

\subsection{Soergel reciprocity}
 By \cite[Corollary 5.8]{Brun04}, we have the following reciprocity for indecomposable tilting modules.

\begin{prop}\label{iso-1}
Let $\lambda,\mu\in \bfe$, and $w_0$ be the longest element of the Weyl group of $\ggg_0$. Denote by $[T:\Delta(\lambda)]$ the multiplicity of $\Delta(\lambda)$ in the $\Delta$-flag of a given tilting module $T$.
The following statements hold.
\begin{itemize}
\item[(1)] If $\ggg=W(n),$ then
	$$[T(\mu):\Delta(\lambda)]=(\nabla(-w_0\lambda+ \Xi):L(-w_0\mu+ \Xi)).$$
\item[(2)] If $\ggg=\bar{S}(n)$ or $\bar{H}(n)$, then
	$$[T(\mu):\Delta(\lambda)]=(\nabla(-w_0\lambda):L(-w_0\mu)).$$
\end{itemize}
\end{prop}

\begin{proof}
	Note that the character $\mathcal{E}_X$ gives rise to a one-dimensional $\ggg_0$-module $\mathbb{C}_{-\mathcal{E}_X}$, and we have the following $\ggg_0$-module isomorphism
	\[
	\mathbb{C}_{-\cae_X}\cong\left\{
	\begin{array}{ll}
	L^0(\Xi), & \mbox{if}~\ggg=W(n),\\
	L^0(0),&\mbox{if}~\ggg=\bar{S}(n),\bar{H}(n).
	\end{array}
	\right.
	\]
	Then the statements are consequences of \cite[Corollary 5.8]{Brun04}.
\end{proof}

With the aid of Proposition \ref{iso2}, the above Soergel reciprocity 
can be rewritten below.

\begin{prop}\label{iso1}
Let $\lambda,\mu\in\Lambda^+$. The following statements hold.
\begin{itemize}
\item[(1)] If $\ggg=W(n),$ then
		$$[T(\mu):\Delta(\lambda)]=(K(-w_0\lambda+2\Xi):L(-w_0\mu+ \Xi)).$$
\item[(2)] If $\ggg=\bar{S}(n),$ then
$$[T(\mu):\Delta(\lambda)]=(K(-w_0\lambda+\Xi):L(-w_0\mu)).$$
\item[(3)] If $\ggg=\bar{H}(n),$ then
$$[T(\mu):\Delta(\lambda)]=(K(-w_0\lambda+n\delta):L(-w_0\mu)).$$
\end{itemize}
\end{prop}

\begin{proof} This is a direct consequence of Propositions \ref{iso-1} and \ref{iso2}.
\end{proof}

\subsection{Definition of character formulas for $\comif$} \label{character f}


By Theorem \ref{iso-2} and Proposition \ref{iso1} we have seen that the multiplicities of $\Delta(\lambda)$ in $P(\mu)$ or $T(\mu)$ can be attributed to the Cartan invariants of some finite-dimensional Kac-module, so
 $P(\lambda)$ and $T(\lambda)$ belong to $\comif$.  In this section, we compute the character formulas for those $P(\lambda)$ and $T(\lambda)$,  on the basis of degenerate BGG reciprocity (Theorem \ref{deg bggthm}) and Soergel reciprocity (Propositions \ref{iso-1} and \ref{iso1}) respectively. In the following, we first introduce the formal characters of modules in the category $\comif$.

Recall that associated with the standard triangular decomposition $\ggg_0=\nnn^-\oplus \bar\hhh\oplus\nnn^+$, $\ggg_0$ admits a positive root system $\Phi_0^+$.
Furthermore, denote by $\Phig$ the root system of $\ggg_{\geq 1}$ relative to $\bar\hhh$, i.e.,
$\Phig:=\{\alpha\in {\bar\hhh}^*\mid (\ggg_{\geq 1})_\alpha\ne 0\}$ where
$$(\ggg_{\geq 1})_\alpha=\{x\in\ggg_{\geq 1}\mid [h,x]=\alpha(h)x,\,\,\forall\,h\in\bar\hhh\}.$$
Then we have $\ggg_{\geq 1}=\sum\limits_{\alpha\in\Phig}\ggg_{\alpha}$.
Associated with $\lambda\in \Lambda^+$, we define a subset of ${\bar\hhh}^*$:
$$D(\lambda)=\{\mu\in {\bar\hhh}^*\mid \mu\succeq \lambda\},$$
where $\mu\succeq \lambda$ means that $\mu-\lambda$ lies in $\bbz_{\geq 0}$-span of $\Phig\cup \Phi_0^+$. Now we define a $\bbc$-algebra $\mathcal{A}$, whose elements are series of the form $\sum_{\lambda\in \bar\hhh^*}c_\lambda e^\lambda$ with $c_\lambda\in \bbc$ and $c_\lambda=0$ for $\lambda$ outside the union of a finite number of sets of the form $D(\mu)$. Then $\mathcal{A}$ naturally becomes a commutative associative algebra if we define $e^\lambda e^\mu=e^{\lambda+\mu}$, and identify $e^0$ with the identity element. All formal exponentials $\{e^\lambda\}$ are linearly independent, and then in one-to-one correspondence with $\bar\hhh^*$. For a semisimple $\bar\hhh$-module $W=\sum_{\lambda\in\bar\hhh^*}W_\lambda$, if  the weight spaces are all finite-dimensional, then we can define $\ch(W)=\sum_{\lambda\in\bar\hhh^*}(\dim W_\lambda) e^\lambda$. In particular, if $V$ is an object in  $\comif$, then $\ch(V)\in \mathcal{A}$.  We have the following fact.

\begin{lem}\label{cha fact}  The following statements hold.
\begin{itemize}
\item[(1)] Let $V_1, V_2$ and $V_3$ be three $\ggg$-modules in the category $\comif$. If there is an exact sequence of $\ggg$-modules $0\rightarrow V_1\rightarrow V_2\rightarrow V_3\rightarrow 0 $, then $\ch(V_2)=\ch(V_1)+\ch(V_3)$.
\item[(2)] Suppose that $W=\sum_{\lambda\in\bar\hhh^*}W_\lambda$ is a semisimple $\bar\hhh$-module with finite-dimensional weight spaces, and $U=\sum_{\lambda\in\bar\hhh^*}U_\lambda$ is a finite-dimensional $\bar\hhh$-module. If $\ch(W)=\sum_{\lambda\in\bar\hhh^*}c_{\lambda}e^\lambda$ falls in $\mathcal{A}$, then $\ch(W\otimes_\bbc U)$ must fall in $\mathcal{A}$ and $\ch(W\otimes_\bbc U)=\ch(W)\ch(U)$.
\end{itemize}
\end{lem}

Let us investigate the formal character of a standard module $\Delta(\lambda)$ for $\lambda\in \bfe$.
Recall $\Delta(\lambda)=U(\ggg_{\geq1})\otimes_{\bbc}L^0(\lambda)$. As a $U(\ggg_{\geq1})$-module, $\Delta(\lambda)$ is a free module of rank $\dim L^0(\lambda)$ generated by $L^0(\lambda)$. By Lemma \ref{cha fact}(2), we have $\ch(\Delta(\lambda))=\ch(U(\ggg_{\geq1}))\ch L^0(\lambda)$ for $\lambda\in \bfe$.
Note that  $$\Phi_{\geq 1}=\Phi_{\bar{0}}^{\geq1}\cup\Phi_{\bar{1}}^{\geq1}, \mbox{~where~}
\Phi_{\bar{i}}^{\geq1}=\Phi_{\geq 1}\cap \Phi_{\bar{i}}, i\in\mathbb{Z}_2.$$
Set
$$\varTheta:=\prod\limits_{\alpha\in\Phi_{\bar{1}}^{\geq1}}(1+e^{\alpha})\prod\limits_{\alpha\in\Phi_{\bar{0}}^{\geq1}}
(1-e^{\alpha})^{-1}.$$
Then we  further have  $\ch(\Delta(\lambda))=\varTheta\ch L^0(\lambda)$.

\subsection{Character formulas of $T(\lambda)$}
As a direct consequence of the forthcoming Propositions \ref{prop type W}, \ref{prop type S} and \ref{prop type H} in the Appendix B, along with Lemma \ref{cha fact}, Soergel reciprocity leads to the following theorem on character formulas for indecomposable tilting modules.

\begin{thm}\label{char form T} Let $\ggg=X(n)$ for $X\in\{W,\bar{S},\bar{H}\}$, and $\lambda\in\Lambda^+$. The character formulas for tilting modules $T(\lambda)$ are listed as follows.
	
	(1) If $\ggg=W(n)$, then
	\begin{align*}
	\ch T(\lambda)=\begin{cases}
	\varTheta(\ch L^0(\lambda)+\ch L^0(2\lambda)),
	&\mbox{ if } \lambda= \Xi;\cr
	\varTheta(\ch L^0(\lambda)+\ch L^0(\lambda+\epsilon_1)), &\mbox{ if }\lambda=2 \Xi+ a\epsilon_1 \mbox{ with }a\geq0;\cr
	\varTheta(\ch L^0(\lambda)+\ch L^0(\lambda+\epsilon_n)),  &\mbox{ if }\lambda= \Xi+b\epsilon_{n} \mbox{ with }b\leq-1; \cr
	\varTheta(\ch L^0(\lambda)), &\mbox{ if } \lambda\notin\Omega.
	\end{cases}
	\end{align*}
	
	(2) If $\ggg=\bar{S}(n)$, then
	\begin{align*}
	\ch T(\lambda)
	=\begin{cases}
	\varTheta(\ch L^0(\lambda)+\ch L^0(\lambda+\Xi)+\ch L^0(\lambda+\Xi-\epsilon_{n})+\ch L^0(\lambda+\epsilon_{1})),\cr
	\hskip2cm\mbox{ if } \lambda=k\Xi;\cr
	\varTheta(\ch L^0(\lambda)+\ch L^0(\lambda+\epsilon_n)+\ch L^0(\lambda+\epsilon_{1}+\epsilon_{n})), \cr
	\hskip2cm\mbox{ if }\lambda=k\Xi-\epsilon_{n};\cr
	\varTheta(\ch L^0(\lambda)+\ch L^0(\lambda+\epsilon_{n})), \cr
	\hskip2cm\mbox{ if }
	\lambda=k\Xi+b\epsilon_{n} \mbox{ with } b\in\mathbb{Z}_{\leq -2};\cr
	\varTheta(\ch L^0(\lambda)+\ch L^0(\lambda+\epsilon_{1})), \cr
	\hskip2cm\mbox{ if }\lambda=k\Xi+a\epsilon_{1} \mbox{ with }a\in\mathbb{Z}_{\geq1};\cr
	\varTheta(\ch L^0(\lambda)), \mbox{ if } \lambda\notin\Omega.
	\end{cases}
	\end{align*}

	(3) If $\ggg=\bar{H}(n)$, then
	
	\begin{align*}
	\ch T(\lambda)
	=\begin{cases}
	\varTheta(\ch L^0(\lambda)+\ch L^0(\lambda+n\delta)+\ch L^0(\epsilon_{1}+(k+n+1)\delta)+\ch L^0(\epsilon_{1}+(k+3)\delta)),\cr
	\hskip2cm\mbox{ if } \lambda=k\delta;\cr
	\varTheta(\ch L^0(\lambda)+\ch L^0(\lambda+2\delta)+\ch L^0(\lambda+\epsilon_{1}+3\delta)+\ch L^0(\lambda-\epsilon_{1}-\delta)),\cr
	\hskip2cm\mbox{ if }\lambda=k\delta+a\epsilon_{1} \mbox{ with }a\in\mathbb{Z}_{\geq1};\cr
	\varTheta(\ch L^0(\lambda)), \mbox{ if } \lambda\notin\Omega.
	\end{cases}
	\end{align*}
\end{thm}

\subsection{Character formulas of $P(\lambda)$}
According to the degenerate BGG reciprocity (Theorem \ref{deg bggthm}), one can compute the character formulas of indecomposable projective modules precisely by the same method as Theorem \ref{char form T}. We omit the details and list the formulas  as below.

\begin{thm}\label{char form P}
	Let $\ggg=X(n)$ for $X\in\{W,\bar{S},\bar{H}\}$, and $\lambda\in\Lambda^+$. The character formulas for
	indecomposable projective modules $P(\lambda)$ are listed as follows.
	
	(1) If $\ggg=W(n)$, then
	\begin{align*}
	\ch P(\lambda)=\begin{cases}
	\varTheta(\ch L^0(0)+\ch L^0(-\Xi)),
	&\mbox{ if } \lambda= 0;\cr
	\varTheta(\ch L^0(\lambda)+\ch L^0(\lambda-\epsilon_1)), &\mbox{ if }\lambda=a\epsilon_1 \mbox{ with }a\geq1;\cr
	\varTheta(\ch L^0(\lambda)+\ch L^0(\lambda-\epsilon_n)),  &\mbox{ if }\lambda= -\epsilon_{1}-\epsilon_{2}-\cdots-\epsilon_{n-1}+a\epsilon_{n} \mbox{ with }a\leq-1; \cr
	\varTheta(\ch L^0(\lambda)), &\mbox{ otherwise }.
	\end{cases}
	\end{align*}
	
	(2) If $\ggg=\bar{S}(n)$, then
	\begin{align*}
	\ch P(\lambda)
	=\begin{cases}
	\varTheta(\ch L^0(\lambda)+\ch L^0(\lambda-\Xi)+\ch L^0(\lambda-\Xi+\epsilon_{1})+\ch L^0(\lambda-\epsilon_{n})),\cr
	\hskip2cm\mbox{ if } \lambda=k\Xi;\cr
	\varTheta(\ch L^0(\lambda)+\ch L^0(\lambda-\epsilon_1)+\ch L^0(\lambda-\epsilon_{1}-\epsilon_{n})), \cr
	\hskip2cm\mbox{ if }\lambda=\epsilon_{1}+k\Xi;\cr
	\varTheta(\ch L^0(\lambda)+\ch L^0(\lambda-\epsilon_{1})), \cr
	\hskip2cm\mbox{ if }
	\lambda=a\epsilon_{1}+k\Xi \mbox{ with } a\in\mathbb{Z}_{\geq 2};\cr
	\varTheta(\ch L^0(\lambda)+\ch L^0(\lambda-\epsilon_{n})), \cr
	\hskip2cm\mbox{ if }\lambda=k\Xi+c\epsilon_{n} \mbox{ with }c\in\mathbb{Z}_{\leq-1};\cr
	\varTheta(\ch L^0(\lambda)), \mbox{ otherwise } .
	\end{cases}
	\end{align*}

	(3) If $\ggg=\bar{H}(n)$, then
	\begin{align*}
	\ch P(\lambda)
	=\begin{cases}
	\varTheta(\ch L^0(\lambda)+\ch L^0(\lambda-n\delta)+\ch L^0(\lambda+\epsilon_{1}+(1-n)\delta)+\ch L^0(\lambda+\epsilon_{1}-\delta)),\cr
	\hskip2cm\mbox{ if } \lambda=k\delta;\cr
	\varTheta(\ch L^0(\lambda)+\ch L^0(\lambda-2\delta)+\ch L^0(\lambda+\epsilon_{1}-\delta)+\ch L^0(\lambda-\epsilon_{1}-\delta)),\cr
	\hskip2cm\mbox{ if }\lambda=a\epsilon_{1}+k\delta \mbox{ with }a\in\mathbb{Z}_{\geq1};\cr
	\varTheta(\ch L^0(\lambda)), \mbox{otherwise }.
	\end{cases}
	\end{align*}
\end{thm}

\subsection{Bar-typical weights and  indecomposable projective tilting modules}
 Call a weight $\lambda\in \bar\hhh^*$  bar-atypical if $\lambda\in\Omega^{\bar a}$
defined as below
 \begin{align*}
 \Omega^{\bar a}&=
 \begin{cases}
 \{\pm\Xi+b\epsilon_n\mid b\in\bbz_{\leq 0}\}\cup\{d\Xi+a\epsilon_1\mid d=0,2; a\in\bbz_{\geq 0}\}, &\mbox{ for } W(n);\cr
 \{a\epsilon_1+k\Xi \mid a\in \bbz_{\geq 1}, k\in\mathbb{C}\}\cup \{k\Xi+c\epsilon_n \mid c\in \bbz_{\leq 0}, k\in\mathbb{C}\},
 &\mbox{ for }\bar S(n);\cr
 \{a\epsilon_1+k\delta \mid a\in \bbz_{\geq 0}, k\in\mathbb{C}\}, &\mbox{ for }\bar H(n).
 \end{cases}
 \end{align*}
Call a weight $\lambda\in \bar\hhh^*$ bar-typical, if $\lambda\notin \Omega^{\bar a}$.
\begin{prop} \label{typ til} If $\lambda\in \Lambda^+$ is bar-typical, then $P(\lambda)=T(\lambda)=\Delta(\lambda)$. Conversely, if $P(\lambda)=T(\lambda)$,  then $\lambda$ must be bar-typical.
\end{prop}

\begin{proof} The first part of the proposition is a direct consequence of the above theorems.
As to the second part, we only need to verify that when $\lambda\in \Omega^{\bar a}$, $P(\lambda)$ is not a tilting module. In this case, it is really true that  $P(\lambda)=\Delta(\lambda)$ and $T(\lambda)=\Delta(\lambda)$ do not simultaneously happen.
By Propositions \ref{prop type W}, \ref{prop type S}, \ref{prop type H} in Appendix B, Theorem \ref{iso-2} and Proposition \ref{iso-1},
we can see that
$[P(\lambda):\Delta(\lambda)]=1$ and $[T(\lambda): \Delta(\lambda)]=1$ in their $\Delta$-flags. However,  $\Delta(\lambda)$ is a quotient of $P(\lambda)$  and a submodule of $T(\lambda)$ (see Lemma \ref{tilt lem}). This implies that $P(\lambda)\not\cong T(\lambda)$ in this case. The proof is completed.
\end{proof}

\section{Appendix A: A proof for the existence of semi-infinite characters}

(1) Assume  $\ggg=W(n)$.  Let us first check that the linear map $\cae_X$ is indeed a homomorphism of Lie algebras. For any basis elements $\xi_iD_j, \xi_sD_t\in\ggg_0$,  $$\cae_W([\xi_iD_j, \xi_sD_t])=\cae_W(\delta_{js} \xi_iD_t-\delta_{ti}\xi_sD_j)=0.$$
So $\cae_X$ is  a character.

Let $\xi_{k_1}\xi_{k_2}\cdots \xi_{k_{i+1}}D_s$ be an element in  $\ggg_{i}, i\geq 2$. We have the following two cases.

Case (i): $s\neq k_j,\,\forall\,1\leq j\leq i+1$.

In this case,
$$\xi_{k_1}\xi_{k_2}\cdots \xi_{k_{i+1}}D_s=(-1)^i[\xi_{k_2}\cdots \xi_{k_{i+1}}D_s,\xi_s\xi_{k_1}D_{k_1}].$$

Case (ii): $s=k_j$ for some $j\in\{1,\cdots, i+1\}$.

In this case, without loss of generality, we can assume $j=i+1$, i.e., $s=k_{i+1}$. Then we have
$$\xi_{k_1}\xi_{k_2}\cdots \xi_{k_{i+1}}D_s=[\xi_1\xi_{k_2}\cdots \xi_{k_{i}}D_{k_i},\xi_{k_i}\xi_{k_{i+1}}D_{s}].$$
It follows that $\ggg_i$ is included in $[\ggg_{i-1},\ggg_1]$ for any  $i\geq 2$.  By induction on $i$, we see that (SI-1) holds for $W(n)$.
For (SI-2), we can check it through direct calculation in the following.

Without loss of generality, we can assume $x=\xi_k\xi_iD_j, y=D_s$. We divide the proof into the following three cases.

Case (i): $s\neq k$ and $s\neq i$.

In this case, $[x,y]=0$. And we have
\[
[x,[y,z]]=\left\{
\begin{array}{lll}
\xi_iD_j,& \mbox{if}~z=\xi_sD_k;\\
-\xi_kD_j, & \mbox{if}~z=\xi_sD_i;\\
0,&\mbox{if}~z\in\{\xi_uD_t\mid1\leq u, t\leq n\}\setminus\{\xi_sD_k, \xi_sD_i\}.
\end{array}
\right.
\]
It follows that
$\textsf{str}((\textsf{ad}x\circ \textsf{ad}y)|_{\ggg_0})=0=\cae_W([x,y])$.

Case (ii): $s=k$ and $i=j$.

In this case, $[x,y]=\xi_jD_j$, and we have
\[
[x,[y,z]]=\left\{
\begin{array}{lll}
\xi_jD_j,& \mbox{if}~z=\xi_sD_s;\\
-\xi_sD_j, & \mbox{if}~z=\xi_sD_j;\\
0,&\mbox{if}~z\in\{\xi_uD_t\mid1\leq u, t\leq n\}\setminus\{\xi_sD_s, \xi_sD_j\}.
\end{array}
\right.
\]
It follows that
$\textsf{str}((\textsf{ad}x\circ \textsf{ad}y)|_{\ggg_0})=-1=\cae_W([x,y])$.

Case (iii): $s=k$ and $i\neq j$.

In this case, $[x,y]=\xi_iD_j$, and we have
\[
[x,[y,z]]=\left\{
\begin{array}{lll}
\xi_iD_j,& \mbox{if}~z=\xi_sD_s;\\
-\xi_sD_j, & \mbox{if}~z=\xi_sD_i;\\
0,&\mbox{if}~z\in\{\xi_uD_t\mid1\leq u, t\leq n\}\setminus\{\xi_sD_s, \xi_sD_i\}.
\end{array}
\right.
\]
It follows that
$\textsf{str}((\textsf{ad}x\circ \textsf{ad}y)\mid_{\ggg_0})=0=\cae_W([x,y])$.

Thus, (SI-2) holds for $W(n)$. Consequently, $\mathcal{E}_{W}$ is a semi-infinite character for $W(n)$.

(2) Assume $\ggg=\bar{S}(n)$.
For (SI-1), one can refer to \cite[Proposition 4.1.1]{Kac77}. Moreover, since $\ggg_0$ coincides with $W(n)_0$, and $\textsf{str}$ is linear, it follows that $\mathcal{E}_{\bar{S}}$ is a semi-infinite character for $\bar{S}(n)$.

(3) Assume $\ggg=\bar{H}(n)$ or $\ochn$.
For (SI-1), one can refer to \cite[Proposition 4.1.1]{Kac77}. For (SI-2), we can check it through direct calculation in the following.

Without loss of generality, we can assume $x=D_H(\xi_i\xi_j\xi_k)$ and $y=D_s$. We divide the proof into the following two cases.

Case (i): $s=i$.

In this case, $[x,y]=D_H(\xi_j\xi_k)$, and
\[
[x,[y,z]]=\left\{
\begin{array}{llll}
D_H(\xi_j\xi_k), & \mbox{if}~z=D_H(\xi_s\xi_{s^{\prime}})\, \text{or}\, \sfd;\\
D_H(\xi_s\xi_k), & \mbox{if}~z=D_H(\xi_s\xi_{j^{\prime}});\\
-D_H(\xi_s\xi_j), & \mbox{if}~z=D_H(\xi_s\xi_{k^{\prime}});\\
0,&\mbox{if}~z=D_H(\xi_s\xi_t)\,\text{for}\,t\neq s^{\prime}, j^{\prime}, k^{\prime};\\
0,&\mbox{if}~z=D_H(\xi_l\xi_m)\,\text{for}\, l\neq s, m\neq s.
\end{array}
\right.
\]
It follows that
$\textsf{str}((\textsf{ad}x\circ \textsf{ad}y)|_{\ggg_0})=0=\cae_{\bar H}([x,y])$.

Case (ii): $s\neq i,j,k$.

In this case, $[x,y]=0$, and
\[
[x,[y,z]]=\left\{
\begin{array}{ll}
\delta_{it^{\prime}}D_H(\xi_j\xi_k)-\delta_{jt^{\prime}}D_H(\xi_i\xi_k)+\delta_{kt^{\prime}}D_H(\xi_i\xi_j), & \mbox{if}~z=D_H(\xi_s\xi_t);\\
0,&\mbox{if}~z= \sfd\,\text{or}\,D_H(\xi_l\xi_m)\,\text{for}\,l\neq s, m\neq s.
\end{array}
\right.
\]
It follows that
$\textsf{str}((\textsf{ad}x\circ \textsf{ad}y)\mid_{\ggg_0})=0=\cae_{\bar H}([x,y])$.

Thus, (SI-2) holds both for $\bar{H}(n)$ and $\ochn$. Hence, $\mathcal{E}_{\bar{H}}$ (resp. $\mathcal{E}_{\och}$) is a semi-infinite character for $\bar{H}(n)$ (resp. $\ochn$).

\section{Appendix B: Computations for character formulas} \label{appendix}

In this appendix, we list the composition factors of Kac-module which is contributed to compute the character formulas of tilting modules and indecomposable projective modules.  Recall that we have introduced the set $\Omega$ of the so-called Serganova atypical weights in  subsection \ref{kac}.
\subsection{The case of $W(n)$}\label{Wn}
\begin{lem}\label{Tongyi}
	Let $\lambda\in\Lambda^+$. Then the following statements hold.
	\begin{itemize}
		\item [(1)] If $\lambda\neq a\epsilon_i+\epsilon_{i+1}+\cdots+\epsilon_{n},$ $L(\lambda- \Xi)\cong\overline{L}(\lambda).$
		\item [(2)] If  $\lambda=a\epsilon_i+\epsilon_{i+1}+\cdots+\epsilon_{n}$ and $\lambda\neq0,$ then $L(\lambda- \Xi+\epsilon_i)\cong\overline{L}(\lambda)$.
		\item[(3)] If $\lambda=0,$  $L(0)\cong\overline{L}(0)$.
	\end{itemize}
\end{lem}

Based on \cite[Theorem 7.6]{Ser05} and Lemma \ref{Tongyi}, the following lemma holds.
\begin{lem}\label{mult}
	Let $\lambda, \mu \in\Lambda^+$. Then the following statements hold.
	\begin{itemize}
		\item[(1)] If $\lambda=0$, then there is the following exact sequence
		$$0\rightarrow L(-\Xi)\rightarrow K(0)\rightarrow L(0)\rightarrow 0.$$
		\item[(2)] If $\lambda=a\epsilon_{n}$, $a<0$, then there is the following exact sequence
		$$0\rightarrow L(a\epsilon_{n}-\Xi)\rightarrow K(a\epsilon_{n})\rightarrow L((a+1)\epsilon_{n}-\Xi)\rightarrow 0.$$
		\item[(3)] If $\lambda=\epsilon_{1}+\epsilon_2+\cdots+\epsilon_{n},$ then there is the following exact sequence
		$$0\rightarrow L(0)\rightarrow K(\epsilon_{1}+\epsilon_2+\cdots+\epsilon_{n})\rightarrow L(\epsilon_{1})\rightarrow 0.$$

		\item[(4)] If $\lambda=a\epsilon_{1}+\epsilon_2+\cdots+\epsilon_{n},a\geq2,$ then there is the following exact sequence
		$$0\rightarrow L((a-1)\epsilon_{1})\rightarrow K(a\epsilon_{1}+\epsilon_2+\cdots+\epsilon_{n})\rightarrow L(a\epsilon_{1})\rightarrow 0.$$
		\item[(5)] If $(K(\lambda):L(\mu))\neq 0,$ then
		$(K(\lambda):L(\mu))=1$.
		
	\end{itemize}
\end{lem}

By the definition of $\comi$ we only need to consider the weights belonging to $\Lambda^+,$ i.e., the weights $\lambda=\lambda_1\epsilon_1+\lambda_2\epsilon_2+\cdots+\lambda_n\epsilon_n$ such that $\lambda_1-\lambda_2,\lambda_2-\lambda_3,\cdots,\lambda_{n-1}-\lambda_n$
are all non-negative integers. Obviously, the following lemma holds.

\begin{lem}\label{typical}
	Let $\lambda$ be a weight belonging to $\Lambda^+$ such that $-w_0\lambda+2 \Xi$ is Serganova atypical. Then $\lambda$ has to be one of the following two forms
	\begin{itemize}
		\item[(1)] $\lambda=(2-a)\epsilon_1+2\epsilon_2+\cdots+2\epsilon_n,\,\,\text{for\,\,some}\,\,a\in\mathbb{Z}_{\leq 0}$.
		\item[(2)] $\lambda=\epsilon_1+\epsilon_2+\cdots+\epsilon_{n-1}+(2-b)\epsilon_n,\,\,
		\text{for\,\,some}\,\,b\in\mathbb{Z}_{\geq 1}$.
	\end{itemize}
	In case {\rm(1)}, $-w_0\lambda+2 \Xi=a\epsilon_{n}$, while in case {\rm(2)},  $-w_0\lambda+2\Xi=b\epsilon_{1}+\epsilon_{2}+\cdots+\epsilon_{n}$.
\end{lem}

\begin{proof} Assume
	\begin{equation*}\label{omega}
		-w_0\lambda+2\Xi=a\epsilon_i+\epsilon_{i+1}+\cdots+\epsilon_{n}.
	\end{equation*}
	It follows that
	\begin{equation*}\label{lam}
		\lambda=\epsilon_1+\epsilon_2+\cdots+\cdots+\epsilon_{n-i}+(2-a)\epsilon_{n-i+1}+2\epsilon_{n-i+2}+\cdots+2\epsilon_{n-1}+2\epsilon_{n}.
	\end{equation*}
	Since $\lambda$ is an element in $\Lambda^+,$ $\lambda$ has to be one of the following two forms:
	\begin{equation*}\label{form1}
		\lambda=(2-a)\epsilon_1+2\epsilon_2+\cdots+2\epsilon_n,a\in\mathbb{Z}_{\leq 0},
	\end{equation*}
	or
	\begin{equation*}\label{form2}
		\lambda=\epsilon_1+\epsilon_2+\cdots+\epsilon_{n-1}+(2-b)\epsilon_n,b\in\mathbb{Z}_{\geq 1}.
	\end{equation*}
	Consequently, $-w_0\lambda+2\Xi=a\epsilon_{n}$ or $b\epsilon_{1}+\epsilon_{2}+\cdots+\epsilon_{n}$, respectively.   \qed

	Now we are in the position to determine the multiplicities of standard modules appearing in each tilting module.
	\begin{prop}\label{maint1}
		Let $\lambda,\mu\in \Lambda^+.$ Then the following statements hold.
		\begin{itemize}
			\item[(1)] In the case $\lambda=2\Xi$, $[T(\mu):\Delta(\lambda)]\neq0$ if and only if $\mu=\Xi$ or $\mu=\lambda.$
			
			\item[(2)] In the case $\lambda=(2-a)\epsilon_1+2\epsilon_2+\cdots+2\epsilon_n,a\in\mathbb{Z}_{\leq -1}$,
			$[T(\mu):\Delta(\lambda)]\neq0$ if and only if $\mu=\lambda-\epsilon_{1}$ or $\mu=\lambda.$
			
			\item[(3)] In the case $\lambda=\epsilon_1+\epsilon_2+\cdots+\epsilon_{n-1}+(2-b)\epsilon_n,b\in\mathbb{Z}_{\geq 1}$,
			$[T(\mu):\Delta(\lambda)]\neq0$ if and only if $\mu=\lambda-\epsilon_{n}$ or $\mu=\lambda.$
			
			\item[(4)] In the case that	
			$\lambda$ is not any one of the forms in Cases (i), (ii), (iii),
			$[T(\mu):\Delta(\lambda)]\neq0$ if and only if $\lambda=\mu.$
			
		\end{itemize}
		
		Moreover, if $[T(\mu):\Delta(\lambda)]\neq0,$ $[T(\mu):\Delta(\lambda)]=1.$

	\end{prop}
	
	\pf (1) Let $\lambda=(2-a)\epsilon_1+2\epsilon_2+\cdots+2\epsilon_n,a\in\mathbb{Z}_{\leq 0}$. By Proposition \ref{iso1} and Lemma \ref{typical}, we have
	$$[T(\mu):\Delta(\lambda)]=(K(a\epsilon_{n}):L(-w_0\mu+\Xi)).$$
	
	(1-i) $a=0$.
	
	In this case,
	\[
	\begin{array}{rl}
		&[T(\mu):\Delta(\lambda)]\neq 0\\
		\Longleftrightarrow& (K(0):L(-w_0\mu+\Xi))\neq0\\
		{\Longleftrightarrow}& L(-w_0\mu+ \Xi)\cong L(0)\mbox{~or~}L(-w_0\mu+ \Xi)\cong L(-\Xi) \;\; (\mbox{by Lemma }\ref{mult})
		\\
		{\Longleftrightarrow}&-w_0\mu+\Xi=0\mbox{~or~}-w_0\mu+\Xi=-\Xi
		\;\;
		\\
		\Longleftrightarrow&\mu= \Xi\mbox{~or~}\mu=2 \Xi\\
		i.e., &\mu= \Xi\mbox{~or~}\mu=\lambda.\\
		
	\end{array}
	\]	
	
	(1-ii) $a\leq -1$.
	
	In this case,
	\[
	\begin{array}{rl}
		&[T(\mu):\Delta(\lambda)]\neq 0\\
		\Longleftrightarrow& (K(a\epsilon_{n}):L(-w_0\mu+\Xi))\neq0\\
		{\Longleftrightarrow}& L(-w_0\mu+ \Xi)\cong L(a\epsilon_n-\epsilon_{n-1}-\cdots-\epsilon_{1})\mbox{~or~}\\
        &L(-w_0\mu+ \Xi)\cong L((a-1)\epsilon_n-\epsilon_{n-1}-\cdots-\epsilon_{1})
		\;\; (\mbox{by Lemma }\ref{mult})
		\\
		{\Longleftrightarrow}&-w_0\mu+\Xi=a\epsilon_n-\epsilon_{n-1}-\cdots-\epsilon_{1}
		\\
		&\mbox{~or~}-w_0\mu+\Xi=(a-1)\epsilon_n-\epsilon_{n-1}-\cdots-\epsilon_{1}
		\;\;
		\\
		\Longleftrightarrow&\mu=(1-a)\epsilon_1+2\epsilon_2+\cdots+2\epsilon_n
		\mbox{~or~}\mu=(2-a)\epsilon_1+2\epsilon_2+\cdots+2\epsilon_n\\
		i.e.,&\mu=\lambda-\epsilon_1\mbox{~or~}\mu=\lambda.\\
	\end{array}
	\]	
	
	For the results in (2)-(4), we can calculate them similarly.	
\end{proof}

As a direct consequence, the following proposition holds.
\begin{prop}\label{prop type W}
	Let $\ggg=W(n)$ and $\mu\in\Lambda^+$.  Then the following statements hold.
	\begin{itemize}
		\item [(1)] If $\mu=\epsilon_1+\epsilon_2+\cdots+\epsilon_{n}$, we have the following exact sequence:
		$$0\rightarrow\Delta(\mu)\rightarrow
		T(\mu)\rightarrow\Delta(2\mu)\rightarrow0.$$
		\item [(2)] If $\mu=a\epsilon_1+2\epsilon_2+\cdots+2\epsilon_{n}$ with $a\geq2$, then we have the following exact sequence:
		$$0\rightarrow\Delta(\mu)\rightarrow
		T(\mu)\rightarrow\Delta(\mu+\epsilon_{1})\rightarrow0.$$
		\item [(3)] If $\mu=\epsilon_1+\epsilon_2+\cdots+\epsilon_{n-1}+b\epsilon_{n}$ with $b\leq0$, then we have the following exact sequence:
		$$0\rightarrow\Delta(\mu)\rightarrow
		T(\mu)\rightarrow\Delta(\mu+\epsilon_{n})\rightarrow0.$$
		\item [(4)] Otherwise, $T(\mu)=\Delta(\mu)$.
	\end{itemize}
\end{prop}

\subsection{The case $\bar{S}(n)$}

Let $\lambda$ be an element in $\Omega.$ Then it is easy to see that $\lambda$ belongs to $\Lambda^+$  if and only if
$$\lambda=b\epsilon_{1}+ a\epsilon_{2}+a\epsilon_{3}+\cdots+a\epsilon_{n}\,\,\text{with}\,\, (b-a)\in\mathbb{Z}_{\geq 0},$$
or
$$\lambda=a\epsilon_{1}+a\epsilon_{2}+\cdots+ a\epsilon_{n-1}+c\epsilon_{n}\,\,\text{with}\,\,(a-c)\in\mathbb{Z}_{\geq 0}. $$

The following result follows directly from \cite[Lemma 5.1]{Ser05}.
\begin{lem}\label{Sniso}
	Let $\lambda\in\Lambda^+$. Then the following statements hold.
	\begin{itemize}
		\item [(1)] If $\lambda=a \Xi-\epsilon_{n},$ $\overline{L}(\lambda)\cong L(\lambda- \Xi+\epsilon_1+\epsilon_{n}),$ i.e.,
		$$\overline{L}(a \Xi-\epsilon_{n})\cong L(a \Xi-\epsilon_{n}-\epsilon_{n-1}-\cdots-\epsilon_2).$$
		\item [(2)] If $\lambda=a \Xi,$ $\overline{L}(\lambda)\cong L(\lambda).$
		\item[(3)] If $\lambda=a \Xi-b\epsilon_{n}$ for $b\in\mathbb{Z}_{\geq 2}$, then $\overline{L}(\lambda)\cong L(\lambda-\Xi+\epsilon_n)$.
		\item[(4)] If $\lambda=a \Xi+b\epsilon_{1}$ for $b\in\mathbb{Z}_{\geq 1}$, then $\overline{L}(\lambda)\cong L(\lambda-\Xi+\epsilon_1)$.	
		\item[(5)] If $\lambda\notin\Omega,$ $\overline{L}(\lambda)\cong L(\lambda- \Xi).$
	\end{itemize}	
\end{lem}

Based on the results in \cite[\S8]{Ser05} and Lemma \ref{Sniso}, we get the following lemma.
\begin{lem}\label{mults}
	Let $\lambda\in\Lambda^+$. Then the following statements hold.
	\begin{itemize}
		\item[(1)] If $\lambda=a \Xi,$ then we have the following exact sequences
		$$0\rightarrow  M\rightarrow K(a\Xi)\rightarrow L(a \Xi)\rightarrow0,$$
		$$0\rightarrow  L((a-1)\Xi)\rightarrow M\rightarrow L((a-1) \Xi+\epsilon_{1})\rightarrow0.$$
		
		
		\item[(2)] If $\lambda=a \Xi+\epsilon_{1},$ then we have the following exact sequences
		$$0\rightarrow  M\rightarrow K( a\Xi+\epsilon_{1})\rightarrow L((a-1) \Xi+2\epsilon_{1})\rightarrow0,$$
		$$0\rightarrow  L((a-1) \Xi+\epsilon_{1})\rightarrow M\rightarrow L(a \Xi)\rightarrow0.$$
		
		\item[(3)] If $\lambda=a \Xi-\epsilon_{n},$  then we have the following exact sequences
		$$0\rightarrow  M\rightarrow K(a \Xi-\epsilon_{n})\rightarrow L((a-1) \Xi+\epsilon_{1})\rightarrow0,$$
		$$0\rightarrow  L((a-1) \Xi-\epsilon_{n})\rightarrow M\rightarrow L((a-1) \Xi)\rightarrow0.$$
		
		\item[(4)] If $\lambda=b\epsilon_{1}+a\Xi,b\in\zz_{\geq2}$,
		then we have the following exact sequence
		$$0\rightarrow L((a-1) \Xi+b\epsilon_{1})\rightarrow K(\lambda)\rightarrow L((a-1) \Xi+(b+1)\epsilon_{1})\rightarrow0.$$
		
		\item[(5)] If $\lambda=a\Xi-c\epsilon_{n},c\in\zz_{\geq2}$,
		then we have the following exact sequence
		$$0\rightarrow  L((a-1) \Xi-c\epsilon_{n})\rightarrow K(\lambda)\rightarrow  L((a-1) \Xi-(c+1)\epsilon_{n})\rightarrow0.$$
		
		\item[(6)] If $(K(\lambda):L(\mu))\neq 0,$ then
		$(K(\lambda):L(\mu))=1$.
	\end{itemize}
\end{lem}

Similar arguments as in the proof of Proposition \ref{maint1} yield the following  result on the multiplicities of standard modules in each tilting module for $\bar{S}(n)$.

\begin{prop}\label{prop type S}
	Let $\ggg=\bar{S}(n)$ and $\lambda$ be an element in $\Lambda^+$. Then
	$[T(\lambda):\Delta(\mu)]\neq0$ implies  $[T(\lambda):\Delta(\mu)]=1$. Furthermore,  the following statements hold.
	\begin{itemize}
		\item[(1)] Assume that $\lambda$ is Serganova atypical.
		\begin{itemize} 		
			\item[(1-i)] If $\lambda=k\Xi$, then
			\begin{align*}
				[T(\lambda):\Delta(\mu)]\neq0
				\Longleftrightarrow \mu\in \{\lambda,\lambda+\Xi,\lambda+ \Xi-\epsilon_{n},\lambda+\epsilon_{1}\}.
			\end{align*}
			\item[(1-ii)] If $\lambda=k \Xi-\epsilon_{n}$, then
			\begin{align*}
				[T(\lambda):\Delta(\mu)]\neq0
				\Longleftrightarrow&\mu\in \{\lambda+\epsilon_n,\lambda,\lambda+\epsilon_{1}+\epsilon_{n}\}.
			\end{align*}
			\item[(1-iii)] If $\lambda=k \Xi+b\epsilon_{n}$ with $b\in\mathbb{Z}_{\leq -2}$, then
			$$[T(\lambda):\Delta(\mu)]\neq0\Longleftrightarrow \mu\in \{\lambda,\lambda+\epsilon_{n}\}.$$			
			\item[(1-iv)] If $\lambda=k \Xi+a\epsilon_{1}$ with $a\in\mathbb{Z}_{\geq1}$, then
			$$[T(\lambda):\Delta(\mu)]\neq0\Longleftrightarrow \mu\in \{\lambda,\lambda+\epsilon_{1}\}.$$
		\end{itemize}
		\item[(2)] In the case that $\lambda$ is Serganova typical,  $T(\lambda)=\Delta(\lambda)$.
	\end{itemize}
\end{prop}

\subsection{The case $\bar{H}(n)$}

\begin{lem}
	Let $\lambda\in \Lambda^+$ be a Serganova  atypical weight. Then $\lambda=a\epsilon_{1}+m\delta$ for some  $a\in\mathbb{Z}_{\geq0}$.
\end{lem}
\pf  With respect to our choice of positive roots, we can get that if $\lambda=\lambda_1\epsilon_{1}+\lambda_2\epsilon_{2}+\cdots+\lambda_r\epsilon_{r}+b\delta$ is an element of $\Lambda^+$, then it must satisfy the following conditions:
\begin{itemize}
	\item[(i)] when $n=2r$, then $\lambda_1\geq\lambda_2\geq\cdots\geq\lambda_{r-1}\geq|\lambda_{r}|,$
	$\lambda_i-\lambda_{j}\in\zz ~and~ \lambda_i\in\frac{1}{2}\zz;$
	
	\item[(ii)] when $n=2r+1$, then $\lambda_1\geq\lambda_2\geq\cdots\geq\lambda_{r-1}\geq\lambda_{r},$
	$\lambda_i-\lambda_{j}\in\zz ~and~ \lambda_r\in {1\over 2}\zz_{\geq0}.$
\end{itemize}
Consequently, from the expression of $\Omega$ in \S\ref{kac} the lemma follows. \qed

The following result follows from \cite[Lemma 5.1]{Ser05}.
\begin{lem}\label{iso-H}
	Let $\lambda\in\Lambda^+.$ Then the following statements hold.
	\begin{itemize}
		\item[(1)] If $\lambda$ is Serganova typical, then $\bar{L}(\lambda)\cong L(\lambda-n\delta)$.
		\item[(2)] If $\lambda$ is Serganova atypical and $\lambda\neq a\delta,$ then
		$\bar{L}(\lambda)\cong L(\lambda+(2-n)\delta)$.
		\item[(3)] If $\lambda= a\delta,$
		$\bar{L}(\lambda)\cong L(\lambda)$.
	\end{itemize}
\end{lem}

The following description on composition factors of Kac modules with Serganova atypical weights follows from  Lemma \ref{iso-H} and \cite[Section 9]{Ser05}.
\begin{lem}\label{comp-H}
	Let $\lambda\in\Omega$. Then the following statements hold.
	\begin{itemize}
		\item[(1)] If $\lambda=a\delta,$ then the irreducible composition factors of $K(\lambda)$ are
		$$L(a\delta), L((a-n)\delta), L(\epsilon_{1}+(a+1-n)\delta).$$
		\item[(2)] If $\lambda=\epsilon_{1}+a\delta,$ then
		the irreducible composition factors of $K(\lambda)$ are
		$$L((a-1)\delta), L((a+1-n)\delta), L(\lambda+\epsilon_{1}+(1-n)\delta),L(\lambda+(2-n)\delta), L(\lambda-n\delta).$$
		\item[(3)] If $\lambda=b\epsilon_{1}+a\delta,b\in\mathbb{Z}_{\geq 2}$, then
		the irreducible composition factors of $K(\lambda)$ are
		$$L(\lambda+(2-n)\delta), L(\lambda-n\delta), L(\lambda+\epsilon_{1}+(1-n)\delta), L(\lambda-\epsilon_{1}+(1-n)\delta).$$
		\item[(4)] If $(K(\lambda):L(\mu))\neq 0,$ then
		$(K(\lambda):L(\mu))=1$.
		
	\end{itemize}
\end{lem}

Let $\lambda=a\epsilon_{1}+m\delta$ and $\mu=b\epsilon_{1}+l\delta$ be elements in $\Lambda^+$,  we have $-\omega_0\lambda=\lambda+(2a-2m)\delta, -\omega_0\mu=\lambda+(2b-2l)\delta$.
So $[T(\mu):\Delta(\lambda)]=(K(\lambda+(2a-2m+n)\delta):L(\mu+(2b-2l)\delta))$ due to Proposition \ref{iso1}. Then we obtain the following  result on the multiplicities of standard modules in each tilting module for $\bar{H}(n)$.

\begin{prop}\label{prop type H}
	Let $\ggg=\bar{H}(n)$ and $\lambda\in\Lambda^+$. Then $[T(\lambda):\Delta(\mu)]\neq0$ implies $[T(\lambda):\Delta(\mu)]=1$.
	Moreover, the following statements hold.
	\begin{itemize}
		\item[(1)] Assume that $\lambda$ is Serganova atypical.
		\begin{itemize}
			\item[(1-i)] If $\lambda=m\delta$,  then
			$$[T(\lambda):\Delta(\mu)]\neq0 \Longleftrightarrow \mu\in \{\lambda,
			\lambda+n\delta, \epsilon_{1}+(m+n+1)\delta, \epsilon_{1}+(m+3)\delta\}.$$
			\item[(1-ii)] If  $\lambda=a\epsilon_{1}+m\delta,a\geq1$, then $$[T(\lambda):\Delta(\mu)]\neq0\Longleftrightarrow \mu\in \{\lambda, \lambda+2\delta, \lambda+\epsilon_{1}+3\delta, \lambda-\epsilon_{1}-\delta\}.$$
		\end{itemize}			
		\item[(2)] If $\lambda$ is Serganova typical, $T(\lambda)=\Delta(\lambda).$
	\end{itemize}
\end{prop}


\begin{thebibliography}{99}

\bibitem{AF92} F. W. Anderson and K. R. Fuller, {\itshape Rings and categories of modules}, Grad. Texts in Math., 13, Second edition, Springer-Verlag, New York, 1992.




\bibitem{Ar97} S. M. Arkhipov, {\itshape Semi-infinite cohomology of associative algebras and bar duality}, Int. Math. Res. Not.  IMRN (1997), no. 17, 833-863.


\bibitem{BKN} I. Bagci, J. R. Kujawa and D. K. Nakano, {\itshape Cohomology and support varieties for Lie superalgebras of type $W(n)$}. Int. Math. Res. Not. IMRN (2008), Art. ID rnn115, 42 pp.

\bibitem{BF93} A. D. Bell and R. Farnsteiner, {\itshape On the theory of Frobenius extensions and its applications to Lie superalgebras}, Trans. Amer. Math. Soc. 335 (1993), 407-424.


\bibitem{BGG}  I. N. Bernstein, I. M. Gelfand and S. I. Gelfand, {\itshape On a category of g-modules}, Funktsional. Anal. i Prilozhen. 10 (1976), no. 2, 1-8; English transl., Funct. Anal. Appl. 10 (1976), 87-92.

\bibitem{BL83} I. N. Bernstein and D. A. Leites, {\itshape Irreducible representations of finite-dimensional Lie superalgebras of type
	$W$}, Selecta Math. Soviet. 3 (1983/84), no. 1, 63-68.



\bibitem{Brun04} J. Brundan, {\itshape Tilting modules for Lie superalgebras}, Comm. Algebra 32 (2004), no. 6, 2251-2268.

\bibitem{DSY18} F.-F. Duan, B. Shu and Y.-F. Yao, {\itshape The category O for Lie algebras of vector fields (I): Tilting modules and character formulas}, Publ. Res. Inst. Math. Sci. 56 (2020), 743-760.

\bibitem{CLW15} S.-J. Cheng, N. Lam and W. Wang, {\em The Brundan-Kazhdan-Lusztig conjecture for general linear Lie superalgebras},
Duke Math. J. 164 (2015), no. 4, 617-695.

\bibitem{CW12}  S.-J. Cheng and W. Wang, {\it Dualities and representations of Lie \textit{superalgebras}} Grad. Stud. Math., 144, Amer.
	Math. Soc., Providence, RI, 2012.



\bibitem{Fe84} B. L. Feigin, {\itshape Semi-infinite homology of Lie, Kac-Moody and Virasoro algebras}, Uspekhi Mat. Nauk 39 (1984), 195-196.



\bibitem{Hum08} J. E. Humphreys, {\itshape Representations of semisimple Lie algebras in the BGG category $\mathcal{O}$}, Grad. Stud. Math.,  94, Amer. Math. Soc., 2008.



\bibitem{Kac77} V. Kac, {\itshape  Lie superalgebras}, Adv. Math. 26 (1977), no. 1, 8-96.

\bibitem{Kum88} S. Kumar, {\itshape Proof of the Parthasarathy-Ranga Rao-Varadarajan conjecture}, Invent. Math. 93 (1988), no. 1, 117-130.

\bibitem{NT60} T. Nakayama and T. Tsuzuku, {\itshape On Frobenius extensions  $\uppercase\expandafter{\romannumeral1}$},  Nagoya. Math. J. 17 (1960), 89-110.

\bibitem{Rud73} A. N. Rudakov, {\itshape Irreducible representations of infinite-dimensional Lie algebras of Cartan type}, Izv. Akad. Nauk SSSR Ser. Mat. 38 (1974), 836-866; English transl., Math. USSR-Izv. 8 (1974), no. 4, 835-866.

\bibitem{Rud75} A. N. Rudakov, {\itshape Irreducible representations of infinite-dimensional Lie algebras of types $S$ and $H$}, Izv. Akad. Nauk SSS R Ser. Mat. 39 (1975), 496-511; English transl, Math. USSR-Izv. 9 (1975), no. 3, 465-480.

\bibitem{Ser05} V. Serganova, {\itshape On representations of Cartan type Lie superalgebras}, Amer.
Math. Soc. Transl. 213 (2005), 223-239.

\bibitem{Shap82}  A. V. Shapovalov, {\itshape Invariant differential operators and irreducible representations of finite dimensional Hamiltonian and Poisson Lie superalgebras}, Serdica 7 (1981), no. 4, 337-342.

\bibitem{Shom02}  N. Shomron, {\itshape Blocks of Lie superalgebras of type $W(n)$}, J. Algebra 251 (2002), no. 2, 739-750.


\bibitem{Soe98} W. Soergel, {\itshape Character formulas for tilting modules over Kac-Moody algebras}, Represent. Theory 2 (1998),  432--448.


\bibitem{Vor93}  A. A. Voronov, {\itshape Semi-infinite homological algebra}, Invent. Math. 113 (1993), no. 1, 103-146.
\end{thebibliography}
\end{document}